\documentclass[12pt]{article}
\usepackage{amsmath,amssymb,amsfonts,amsthm}
\usepackage[all]{xy}

\newcounter{item}[section]
\newcounter{kirshr}
\newcounter{kirsha}
\newcounter{kirshb}
\newenvironment{enumroman}{\setcounter{kirshr}{1}
\begin{list}{(\roman{kirshr})}{\usecounter{kirshr}} }{\end{list}}
\newenvironment{enumarab}{\setcounter{kirshb}{1}
\begin{list}{(\arabic{kirshb})}{\usecounter{kirshb}} }{\end{list}}
\newtheorem{theorem}{Theorem}[section]

\newtheorem{lemma}[theorem]{Lemma}
\newtheorem{corollary}[theorem]{Corollary}
\theoremstyle{definition}

\newtheorem{example}[theorem]{Example}
\newtheorem{definition}[theorem]{Definition}

\def\C{{\mathfrak{C}}}
\def\Fm{{\mathfrak{Fm}}}

\def\At{{\sf At}}
\def\N{{\cal N}}

\def\Nr{{\mathfrak{Nr}}}

\def\Sg{{\mathfrak{Sg}}}
\def\Fm{{\mathfrak{Fm}}}
\def\A{{\mathfrak{A}}}
\def\B{{\mathfrak{B}}}
\def\C{{\mathfrak{C}}}
\def\D{{\mathfrak{D}}}
\def\M{{\mathfrak{M}}}

\def\CA{{\bf CA}}
\def\QA{{\bf QA}}

\def\QEA{{\bf QEA}}

\def\Df{{\bf Df}}

\def\K{{\bf K}}
\def\K{{\bf K}}

\def\RCA{{\bf RCA}}

\def\Rd{{\mathfrak{Rd}}}
\def\(R)RA{{\bf (R)RA}}

\def\R{\mathbb{R}}

\def\Sc{{\bf Sc}}

 \def\CA{{\sf CA}}
\def\B{{\sf B}}
\def\G{{\sf G}}
\def\w{{\sf w}}
\def\y{{\sf y}}
\def\g{{\sf g}}

\def\r{{\sf r}}
\def\K{{\sf K}}
 \def\Cm{{\mathfrak{Cm}}}
\def\Nr{{\mathfrak{Nr}}}
\def\SNr{{\bf S}{\mathfrak{Nr}}}

\def\restr #1{{\restriction_{#1}}}
\def\cyl#1{{\sf c}_{#1}}
\def\diag#1#2{{\sf d}_{#1#2}}
\def\sub#1#2{{\sf s}^{#1}_{#2}}

\def\Ra{{\mathfrak{Ra}}}
\def\Ca{{\mathfrak{Ca}}}
\def\set#1{\{#1\} }
\def\Ra{{\mathfrak{Ra}}}
\def\Nr{{\mathfrak{Nr}}}
\def\Tm{{\mathfrak{Tm}}}
\def\A{{\mathfrak{A}}}
\def\B{{\mathfrak{B}}}
\def\C{{\mathfrak{C}}}
\def\D{{\mathfrak{D}}}

\def\CA{{\bf CA}}

\def\RCA{{\bf RCA}}
\def\G{{\bf G}}

\def\L{{\mathfrak{L}}}
\def\R{\cal{R}}

\def\PEA{{\bf PEA}}
\def\PA{{\bf PA}}
\def\Mat{{\sf Mat}}

\def\ws{winning strategy}

\def \set#1{\{#1\} }

\def\Nr{{\mathfrak{Nr}}}

\def\Sg{{\mathfrak{Sg}}}
\def\Fm{{\mathfrak{Fm}}}
\def\Rd{{\mathfrak{Rd}}}

\def\CA{{\bf CA}}
\def\RCA{{\bf RCA}}
\def\K{{\bf K}}
\def\L{{\bf L}}

\def\QA{{\bf QA}}
\def\QEA{{\bf QEA}}
\def\(R)RA{{\bf (R)RA}}

\def\R{\mathbb{R}}

\def\Nr{{\mathfrak{Nr}}}

\def\Sg{{\mathfrak{Sg}}}
\def\Fm{{\mathfrak{Fm}}}
\def\Rd{{\mathfrak{Rd}}}

\def\CA{{\bf CA}}
\def\RCA{{\bf RCA}}
\def\K{{\bf K}}
\def\L{{\bf L}}

\def\QA{{\bf QA}}
\def\QEA{{\bf QEA}}

\def\(R)RA{{\bf (R)RA}}

\def\R{\mathbb{R}}

\def\QA{{\bf QA}}
\def\QEA{{\bf QEA}}
\def\PA{{\bf PA}}

\def\R{\mathbb{R}}
\def\N{\mathbb{N}}

\def\R{\mathbb{R}}

\def\Sc{{\bf Sc}}

 \def\CA{{\sf CA}}

\def\M{{\mathfrak{M}}}

\def\At{{\mathfrak{At}}}
\def\GG{\cal {G}}

\def\G{{\mathfrak{G}}}

\def\K{{\bf K}}
\def\QA{{\bf QA}}
\def\QEA{{\bf QEA}}

\def\sub#1#2{{\sf s}^{#1}_{#2}}
\def\cyl#1{{\sf c}_{#1}}
\def\diag#1#2{{\sf d}_{#1#2}}

\def\s{{\sf s}}

\def\pa{$\forall$}
\def\pe{$\exists$}

\def\ef{Ehren\-feucht--Fra\"\i ss\'e}

\def\nodes{{\sf nodes}}

\def\restr #1{{\restriction_{#1}}}

\def\A{{\mathfrak{A}}}
\def\B{{\mathfrak{B}}}
\def\C{{\mathfrak{C}}}
\def\D{{\mathfrak{D}}}

\def\Fm{{\mathfrak{Fm}}}
\def\Ra{{\mathfrak{Ra}}}
\def\Nr{{\mathfrak{Nr}}}
\def\F{{\mathfrak{F}}}
\def\CA{{\bf CA}}
\def\RCA{{\bf RCA}}

\def\set#1{ \{#1\}}

\def\Ca{{\mathfrak Ca}}

\def\pe{$\exists$}
\def\pa{$\forall$}
\def\Cm{{\mathfrak Cm}}
\def\Sg{{\mathfrak Sg}}

\def\ls { L\"owenheim--Skolem}
\def\At{{\sf At}}
\def\Id{{\sf Id}}

\def\rng{{\sf rng}}
\def\dom{{\sf dom}}

\def\w{{\sf w}}
\def\g{{\sf g}}
\def\y{{\sf y}}
\def\r{{\sf r}}

\def\cyl#1{{\sf c}_{#1}}
\def\sub#1#2{{\sf s}^{#1}_{#2}}
\def\diag#1#2{{\sf d}_{#1#2}}

\def\d{Dedekind-MacNeille}
\def\K{{\sf K}}
\def\PEA{{\sf PEA}}
\def\swap#1#2{{\sf s}_{[#1, #2]}}

\def\ws{winning strategy}
\def\ef{Ehren\-feucht--Fra\"\i ss\'e}

\def\y{{\sf y}}
\def\g{{\sf g}}
\def\r{{\sf r}}
\def\w{{\sf w}}

\def\Sc{{\sf Sc}}

\def\CA{{\sf CA}}

\def\R{{\sf R}}

\def\RCA{{\sf RCA}}
\def\L{{\mathfrak{L}}}
\def\ls { L\"owenheim--Skolem}
\def\Z{{\mathbb{Z}}}
\def\Uf{{\sf Uf}}
\def\QA{{\sf QA}}
\def\QEA{{\sf QEA}}
\def\Df{{\sf Df}}
\def\PA{{\sf PA}}
\title{Atom canonicity, \d\ completions, neat embeddings and omitting types}
\author{Tarek Sayed Ahmed}

\begin{document}
\maketitle
\begin{abstract} Let $n$ be finite $>2$. We show that any class 
between $S\Nr_n\CA_{n+3}$ and $\sf RCA_n$ is not atom canonical, and any class 
containing the class of completely representable algebras and contained
in $S_c\Nr_n\CA_{n+3}$ is not elementary. We show that 
there is no finite variable universal axiomatization of many 
diagonal free reducts of representable cylindric algebras of dimension $n$,
like the varieties of representable diagonal-free cylindric algebras and Halmos' polyadic algebras (without equality).
We apply our hitherto obtained  algebraic results to show that the omitting types theorem fails for
finite variable fragments of first order logic with and without equality, 
having $n$ variables, even if we count in severely relativized models
as candidates for omitting single non-principle types. Finally, we show that for many cylindric-like algebras, 
like diagonal free cylindric algebras and Halmos' polyadic algebras with and without equality
the class of strongly representable atom structures of finite dimension $>2$ is not elementary.
\noindent  
\end{abstract}

\section{Introduction}

We follow the notation of \cite{1} which is in conformity with that of \cite{HMT1}.
Assume that we have a class of Boolean algebras with completely additive operators for which we have a semantical notion of representability
(like Boolean set algebras or cylindric set algebras).
A weakly representable atom structure is an atom structure such that at least one atomic algebra based on it is representable.
It is strongly representable  if all atomic algebras having this atom structure are representable.
The former is equivalent to that the term algebra, that is, the algebra generated by the atoms,
in the complex algebra is representable, while the latter is equivalent  to that the complex algebra is representable.

Could an atom structure be possibly weakly representable but {\it not} strongly representable?
Ian Hodkinson \cite{Hodkinson}, showed that this can indeed happen for both cylindric  algebras of finite dimension $\geq 3$, and relation algebras,
in the context of showing that the class of representable algebras, in both cases,  is not closed under \d\ completions.
In fact, he showed that this can be witnessed on an atomic algebras, so that
the variety of representable relation algebras algebras and cylindric algebras of finite dimension $>2$
are not atom-canonical. (The complex algebra of an atom structure is
the completion of the term algebra.)
This construction is somewhat complicated using a rainbow atom structure. It has the striking consequence
that there are two atomic algebras sharing the same
atom structure, one is representable the other is not.

This construction was simplified and streamlined,
by many authors, including the author \cite{weak}, but Hodkinson's  construction,
as we indicate below, has the supreme advantage that it has a huge potential to prove analogous
theorems on \d\ completions, and atom-canonicity
for several varieties of algebras
including properly the variety of representable cylindric-like algebras, whose members have a neat embedding property,
such as polyadic algebras with and without equality and Pinter's substitution algebras.
In fact, in such cases atomic {\it representable countable algebras}
will be constructed so that their \d\ completions are outside such varieties.

Restricted to the cylindric algebra case,  we show,
that for $n>2$ finite, there  is representable atomic algebra $\A\in \CA_n$ such that
$\Cm\At\A\notin \SNr_n\CA_{n+3}$ inferring that the varieties
$S\Nr_n\CA_{n+k}$, for $n>2$ finite and for any $k\geq 3,$
are not closed under \d\ completions.
Such results, as illustrated below
will have non-trivial (to say the least)  repercussions on omitting types for finite variable fragments of first order logic with and without 
equality.

Our construction
presented herein model-theoretically, in the spirit of Hodkinson's rainbow construction,
gives a polyadic atomic representable equality algebra of finite dimension $n>2$
such that  the diagonal free reduct
of its completion is not
representable.

Now that we have two distinct classes,
namely, the class of weakly representable atom structures and that of the strongly representable atom structures; 
the most pressing need is to try to
classify them.
Venema proved (in a more general setting) that the former is elementary,
while Hirsch and Hodkinson show that the latter is {\it not} elementary. Their proof is amazing
depending on an ultraproduct of Erdos probabilistic graphs \cite{hirsh}.

We know that there is a sequence of strongly representable atom structures whose ultraproduct is {\it only } weakly representable,
it is {\it not } strongly representable. This gives that the class $\K=\{\A\in \CA_n: \text { $\A$ is atomic and }
\Cm\At\A\in \sf RCA_n\}$ is not elementary, as well.

Here we extend Hirsch and Hodkinson's result to many cylindric-like algebras, 
answering a question of Hodkinson's for $\PA$ and $\PEA$, that was answered also by Bulian and Hodkinson \cite[Theorem 9.7]{bh}.
The proof presented here is based on algebras that are simpler than those used in\cite{bh}, which are 
tailored for a different purpose, namely, proving that,  any first order  axiomatization
of the classes of representable and cylindric algebras of finite dimension
$>2$ must contain 
infinitely many non-canonical sentences.

The algebras we use are based on atom structures of cylindric algebras constructed 
in \cite{hirsh} by noting that these algebras can be endowed with
polyadic operations (in an obvious way) and that they are generated by elements
whose dimension sets do not exhaust the dimension. 
The latter implies that an algebra is representable if and only if
its diagonal free reduct is representable.

Lately, it has become fashionable in algebraic logic to
study abstract algebras that posses  complete representations, witness  \cite{Sayed} for an extensive overview.
A representation of $\A$ is roughly an injective homomorphism from $f:\A\to \wp(V)$
where $V$ is a set of $n$-ary sequences; $n$ is the dimension of $\A$, and the operations on $\wp(V)$
are concrete  and set theoretically
defined, like the Boolean intersection and cylindrifiers or projections.
A complete representation is one that preserves  arbitrary disjuncts carrying
them to set theoretic unions.
If $f:\A\to \wp(V)$ is such a representation, then $\A$ is necessarily
atomic and $\bigcup_{x\in \At\A}f(x)=V$.

Let us focus on cylindric algebras for some time to come.
It is known that there are countable atomic $\RCA_n$s when $n>2$,
that have no complete representations;
in fact, the class of completely representable $\CA_n$s when $n>2$, is not even elementary \cite[Corollary 3.7.1]{HHbook2}.

Such a phenomena is also closely
related to the algebraic notion of {\it atom-canonicity}, as indicated, which is an important persistence property in modal logic
and to the metalogical property of  omitting types in finite variable fragments of first order logic
\cite[Theorems 3.1.1-2, p.211, Theorems 3.2.8, 9, 10]{Sayed}.
Recall that a variety $V$ of Boolean algebras with completely additive operators is atom-canonical,
if whenever $\A\in V$, and $\A$ is atomic, then the complex algebra of its atom structure, $\Cm\At\A$ for short, is also
in $V$.

If $\At$ is a weakly representable but
not strongly  representable, then
$\Cm\At$ is not representable; this gives that $\RCA_n$ for $n>2$ $n$ finite, is {\it not} atom-canonical.
Also $\Cm\At\A$  is the \d\ completion of $\A$, and so obviously $\RCA_n$ is not closed under \d\ completions.

On the other hand, $\A$
cannot be completely  representable for, it can be shown without much ado,  that
a complete representation of $\A$ induces a representation  of $\Cm\At\A$ \cite[Definition 3.5.1, and p.74]{HHbook2}.

Finally, if $\A$ is countable, atomic and has no complete representation
then the set of co-atoms (a co-atom is the complement of an atom), viewed in the corresponding Tarski-Lindenbaum algebra,
$\Fm_T$, as a set of formulas, is a non principal-type that cannot be omitted in any model of $T$; here $T$ is consistent
if $|A|>1$. This last connection  was first established by the  author leading up to \cite{ANT} and more, see e.g \cite{HHbook2}.

The reference \cite{Sayed} contains an extensive discussion of such notions.

In the context of cylindric algebras,  closure under complete neat embeddings and complete representability
was proved equivalent for countable atomic algebras by the author  \cite{Sayedneat, Sayed}.
The characterization also works for relation algebras, using the same method, which is a Baire category argument at heart
\cite{1}; later re-proved by Robin Hirsch using games \cite{r}.  It was also proved that all three conditions cannot be omitted;
atomicity, countability and complete embeddability.
There are examples, that show that such conditions cannot be dispensed with.
Hirsch and Hodkinson \cite{HH} prove
that the class of completely representable $\CA_n$s is not elementary, for any $n\geq 3.$
Below we shall strengthen such a result considerably.

We use techniques similar to that of Hirsch's in \cite{r} that deal with relation algebras,  and those of Hirsch and Hodkinson in \cite{HH}
on complete representations. The results in the latter
had to do with investigating the existence of complete representations for both relation and cylindric algebras  and for
this purpose, an infinite (atomic) game that tests complete representability was devised, and such a game was used
on so-called rainbow  algebras. More sophisticated games were also played on rainbow relation algebras in
\cite{r}; to get sharper result on relation algebra reducts of cylindric algebras
and complete representations. We
will use similar games addressing cylindric-like rainbow atom structures of finite dimension $>2$.

In \cite{HH} one game is used to test complete representability (the usual atomic game, a \ws\ for \pe\ in $k$
rounds can be coded in a first order sentence referred to
as the $k$th Lyndon condition).
In  \cite{r} {\it three} games were devised testing different neat embeddability properties.

Here we use only two games adapted to the cylindric paradigm. This suffices for our purposes.
The main result in \cite{HH}, namely, that the class of completely representable algebras of finite dimension
$n\geq 3$, is non elementary, follows from the fact that \pe\  cannot win the infinite length
game, but she can win all the finite rounded ones.

To obtain a stronger result on neat embeddings, namely, that any class containing the class of completely
representable  algebras and further 
contained in ${\bf S}_c\Nr_n\CA_{n+3}$ with $n\geq 3$, we use {\it two} distinct games, both having $\omega$ rounds, but one has a limited number of
pebbles, namely, $n+3$, that \pa\ can use, 
but he has the option to re-use them.  Here $S_c$ denotes the operation of forming complete subalgebras.
That is by $S_c\sf K$, where $\sf K$ is a class having a Boolean reduct, 
we understand the class of complete subalgebras of $\sf K$, that is $\A\in S_c\sf K$ if there exists
$\B\in \sf K$ such that $\A\subseteq \B$ and for all $X\subseteq \A$ whenever $\sum^{\A}X=1,$ 
then $\sum^{\B}X=1.$

Summarizing, our main results are:

\begin{enumarab}

\item For any finite $n>2$, any class $\sf K$ such that $\RCA_n\subseteq \sf K\subseteq S\Nr_n\CA_{n+3}$, $\sf K$  
is not atom-canonical, hence not closed under 
\d\ completions. This result is also proved for many cylindric-like algebras like Pinter's substitution algebras and 
Halmos polyadic algebras with and without equality, theorem \ref{can}.
This solves an open problem first announced by Hirsch and Hodkinson in \cite{HHbook} 
and re-appearing 
in the late \cite{1}. In the first reference the question is attributed to N\'emeti and the present author, 
and in the second reference, it appears among the open problems in \cite{Sayed}.

\item For any finite $n>2$, 
any class $\sf K$ such that $\sf K$ contains the class of completely representable $\CA_n$s 
and is contained in $S_c\Nr_n\CA_{n+3},$
$\sf K$ is not elementary; it is  not closed under ultraroots.  This result is also 
proved for the relatives of $\CA$s mentioned in the previous item, theorem \ref{rainbow}.

\item Call an atomic $\CA_n$ strongly representable if its \d\ completion is representable. 
We show that, for finite $n>2$,  there are uncountable atomic algebras in $\Nr_n\CA_{\omega}$ that are not completely representable, 
but any algebra in $\Nr_n\CA_{\omega}$ is strongly 
representable, theorem \ref{complete}.

\item There is no finite variable universal axiomatization for several classes of representable 
diagonal free reducts of cylindric algebras of finite dimension $>2$, like diagonal free cylindric algebras 
and quasi-polyadic algebras, theorem \ref{df}. This solves an old open problem that dates back to the eighties of the last century,
formulated partially in \cite{ST}.

\item Applying the algebraic results in the first two items to show that the omitting types theorem fails for first order logic with
finitely many variables, as long as the number of variables available are $>2$, 
even if we considerably broaden the permissable class of models omitting  single non-principal types, dealing
with clique guarded semantics, theorem \ref{OTT}.

\item Introducing, for each $n>m>2$, $m$ finite,
a new class of $m$ dimensional cylindric algebras denoted by ${\sf CAB}_{m,n}$.
This class is a strict  approximation to the class $\RCA_m$, when $n$ is finite, and it the $\CA$ analogue of relational
algebras embedding into complete and atomic relation algebras having $n$  dimensional
relational bases introduced by Maddux.
In particular, $\bigcap_{k}\sf CB_{m,m+k}=\RCA_m$.
We show that ${\sf CB}_{m,n}$ is a canonical variety for all $n>m$
that is not atom-canonical when $n\geq m+3$ and $m$ finite $>3$, theorems \ref{cann}, \ref{OTT2}.

\item For $2<n<\omega$ and $\mathcal{T}$ be any signature between $\Df_{n}$
and $\PEA_{n}$,  the class of strongly representable atom
structures of type $\mathcal{T}$ is not elementary, theorem \ref{el}.
This reproves a result of Bu and Hodkinson's \cite{bh}, except that we believe our proof is simpler. 
Though atomicity and complete additivity of an operator in
a Boolean algebra with operators are first order notions, we show that for any finite $n>2$ and 
any $\sf K$ between $\Df $ and $\PEA$,
the class of strongly representable $\sf K_n$s is not elementary.  More precisely, we show 
that there is a sequence of completely additive $\K_n$s
$(\A_i: i\in \omega)$ such that $\Cm\At\A_i\in {\sf RK}_{n}$ for each $i\in \omega$, but for any non-principal ultrafilter
on $\omega$,  $\Cm\At(\Pi_{i/U}\A_i)$ is not representable, theorem \ref{proper}.

\end{enumarab}

\section{Notation and Preliminaries}

We follow, as stated above, the notation in \cite{1}. But, for the reader's convenience,
we include the following list of notation that will be used throughout the  paper.
Other than that our notation is fairly standard or self explanatory.
Unusual notation will be explained at  its first occurrence in the text.

An \textit{ordinal} $\alpha$ is transitive set (i.e., any member of $\alpha$ is
also a subset of $\alpha$) that is well-ordered by $\in$. Every
well-ordered set is order isomorphic to a unique ordinal. For
ordinals $\alpha, \beta$, $\alpha < \beta$ we means $\alpha \in
\beta$. An ordinal is therefore the set of all smaller ordinals, so
for a finite ordinal $n$ we have $n = \{ 0,1, \dots, n-1 \}$ and the least infinite ordinal is
$\omega = \{ 0,1, 2, \dots \}$.

A \textit{cardinal} is an ordinal not in bijection with any smaller
ordinal, briefly an {\it initial ordinal} and the cardinality $|X|$ of a set $X$ is the unique
cardinal in bijection with $X$. Cardinals are ordinals and are
therefore ordered by $<$ (i.e., $\in$). The first few cardinals are
$ 0 = \phi, 1, 2, \dots, \omega$ (the first infinite ordinal),
$\omega_1$ (the first uncountable cardinal). A set will be said to be
\textit{countable} if it has cardinality $\leq \omega$,
\textit{uncountable} otherwise, and \textit{countable infinite} if
it has cardinality $\omega$. $2^{\omega}$ denotes the power of the continuum. 

For a set $X$,  $\wp(X)$ denotes the set of all subsets of $X$, i.e. the powerset of $X$.
Ordinals will be identified with the set of smaller ordinals.
In particular,  for finite $n$  we have $n=\{0,\ldots, n-1\}$.
$\omega$ denotes the least infinite
ordinal which is the set of all finite ordinals.
For two given sets $A$ and $B$, ${}^AB$ denotes the set of functions from $A$ to $B$,
and $A\sim B=\{x\in A: x\notin B\}.$
If $f\in {}^AB$ and $X\subseteq A$ then $f\upharpoonright X$
denotes the restriction of $f$
to $X$. We denote by $\dom f$ and $\rng f$ the domain and range of a given function
$f$, respectively.
We frequently identify $f$ with the sequence $\langle f_x:x \in \dom f\rangle$.
We write $fx$ or $f_x$ or $f(x)$ to denote the value of $f$ at $x$.
We define composition so that the righthand function acts first, thus
for given functions $f,g$, $f\circ g(x)=f(g(x))$, whenever the left hand side is defined, i.e
when $g(x)\in \rng f$.

For a non-empty set $X$, $f(X)$ denotes the image of $X$ under $f$, i.e
$f(X)=\{f(x):x\in X\}.$
$|X|$ denotes the cardinality of $X$ and $Id_X$, or simply $Id$ when $X$ is
clear from context, denotes the identity function on $X$.
A set $X$ is countable if $|X|\leq \omega$; if $X$ and $Y$ are sets then $X\subseteq_{\omega}Y$ denotes that $X$
is a finite subset of $Y$.

Algebras will be denoted by
Gothic letters, and when we write $\A$ then we will be tacitly assuming
that $A$ will denote  the universe
of $\A$.
However, in some occasions we will identify (notationally)
an algebra and its universe.

Fix some ordinal $n\geq 2$.
For $i, j<n$ the replacement $[i/j]$ is the map that is like the identity on $n$, except that $i$ is mapped to $j$ and the transposition
$[i, j]$ is the like the identity on $n$, except that $i$ is swapped with $j$.     A map $\tau:n\rightarrow n$ is  finitary if
the set $\set{i<n:\tau(i)\neq i}$ is finite, so  if $n$ is finite then all maps $n\rightarrow n$ are finitary.
It is known, and indeed not hard to show, that any finitary permutation is a product of transpositions
and any finitary non-injective map is a product of replacements.

The standard reference for all the classes of algebras to be dealt with is  \cite{HMT2}.
Each class in  $\set{\Df_n, \Sc_n, \CA_n, \PA_n, \PEA_n, \sf QEA_n, \sf QEA_n}$ consists of Boolean algebras with extra operators,
as shown in figure~\ref{fig:classes}, where $\diag i j$ is a nullary operator (constant), $\cyl i, \s_\tau,  \sub i j$ and $\swap i j$
are unary operators, for $i, j<n,\; \tau:n\rightarrow n$.

The algebras considered can be conceived
as a generalization from Boolean algebras to algebras of relations of higher rank and algebraic versions of different variants
of first order logic. We start by recalling the concrete versions of such algebras.
Such algebras consist of sets of sequences, and the operations are set-theoretic operations on such sets.
Let $\alpha$ be an ordinal. Let $U$ be a set. For $t, s\in {}^{\alpha}U$ and $i<\alpha$, write $t\equiv_i s$ if $t(j)=s(j)$ for all $i\neq j$.
Then we define for $i,j<\alpha,$ $\tau:\alpha\to \alpha$ and $X\subseteq {}^{\alpha}U$:
\begin{align*}
{\sf c}_iX&=\{s\in {}^{\alpha}U:  \exists t\in X, t\equiv _i s\},\\
{\sf s}_{\tau}X&=\{s\in {}^{\alpha}U: s\circ \tau\in X\},\\
{\sf d}_{ij}&=\{s\in {}^{\alpha}U: s_i=s_j\}.
\end{align*}
$[i/j]$ is the replacement on $\alpha$ that takes $i$ to $j$ and leaves every other thing fixed,
while $[i,j]$ is the transposition interchanging $i$ and $j$.

The extra concrete non-Boolean operations we deal with are as specified above.
For a set $X$, let $\B(X)=(\wp(X), \cup, \cap ,\sim, \emptyset, X)$
be the full Boolean set algebra with universe $\wp(X)$.
Let $\bf S$ be the operation of forming subalgebras,
and $\bf P$ be that of forming products.
\begin{align*}
{\sf RDf}_{\alpha}&={\bf SP}\{(\B(^{\alpha}U), {\sf c}_i)_{i< \alpha}: U \text{ is a set}\}.\\
{\sf RSc}_{\alpha}&={\bf SP}\{(\B(^{\alpha}U), {\sf c}_i, {\sf s}_{[i/j]})_{i,j< \alpha}: U \text{ is a set}\}.\\
{\sf RQA}_{\alpha}&={\bf SP}\{(\B(^{\alpha}U), {\sf c}_i, {\sf s}_{[i/j]}, {\sf s}_{[i,j]})_{i,j< \alpha}: U \text{ is a set}\}.\\
\RCA_{\alpha}&={\bf SP}\{(\B(^{\alpha}U), {\sf c}_i, {\sf s}_{[i/j]})_{i,j< \alpha}: U \text{ is a set}\}.\\
\sf RQEA_{\alpha}&={\bf SP}\{(\B(^{\alpha}U), {\sf c}_i, {\sf s}_{[i/j]}, {\sf s}_{[i,j]})_{i,j< \alpha}: U \text{ is a set}\}.
\end{align*}

All such classes are varieties, so that they are closed under forming homomorphic images, too.
For finite $n$, polyadic algebras are the same as quasi-polyadic algebra and for the infinite
dimensional case we restrict our attention to quasi-polyadic algebras in $\sf QA_n, \QEA_n$.
Each class is defined by a finite set of equation schema. Existing in a somewhat scattered form in the literature, finite
equational axiomatizations defining
$\sf Sc_n, \sf QA_n$ and $\QEA_n$ are given in the appendix of \cite{t}.
For $\CA_n$ we follow the standard axiomatization given in definition 1.1.1 in \cite{HMT1}.

For any operator $o$ of any of these signatures, we write $\dim(o)\; (\subseteq n)$
for the set of  dimension ordinals used by $o$, e.g. $\dim(\cyl i)=\set i, \; \dim (\sub i j)=\set{i, j}$.  An algebra $\A$ in $\QEA_n$
has operators that can define any operator of $\sf QA_n, \CA_n,\;\Sc_n$ and $\sf Df_n$. Thus we may obtain the
reducts $\Rd_{\sf K}\A$ for $\K\in\set{\QEA_n, \sf QA_n, \CA_n, \Sc_n, \sf Df_n}$ and it turns out that the reduct always
satisfies the equations defining the relevant class so $\Rd_{\sf K}\A\in \K$.
Similarly from any algebra $\A$ in any of the classes $\sf QEA_n, \sf QA_n, \sf CA_n, \sf Sc_n$
we may obtain the reduct $\Rd_\Sc\A\in\Sc_n$, which we write as $\Rd_{sc}\A$. We also write
$\Rd_{df}\A$ for the diagonal free reduct of $\A$, so that $\Rd_{df}\A\in \Df_n$.

\begin{figure}
\[\begin{array}{l|l}
\mbox{class}&\mbox{extra operators}\\
\hline
\Df_n&\cyl i:i<n\\
\Sc_n& \cyl i, \s_{[i/j]} :i, j<n\\
\CA_n&\cyl i, \diag i j: i, j<n\\
\PA_n&\cyl i, \s_\tau: i<n,\; \tau\in\;^nn\\
\PEA_n&\cyl i, \diag i j,  \s_\tau: i, j<n,\;  \tau\in\;^nn\\
\QA_n&  \cyl i, \s_{[i/j]}, \s_{[i, j]} :i, j<n  \\
\QEA_n&\cyl i, \diag i j, \s_{[i/j]}, \s_{[i, j]}: i, j<n
\end{array}\]
\caption{Non-Boolean operators for the classes\label{fig:classes}}
\end{figure}
Let $\K\in\set{\QEA, \QA, \CA, \Sc, \Df}$. Then for $n>2$ any ordinal,
${\sf RK}_{n}$ is not axiomatizable by a finite schema; in particular, ${\sf RK}_{n}\neq \K_{n}$.
Let $\A\in \K_n$ and let $2\leq m\leq n$ (possibly infinite ordinals).
The \emph{reduct to $m$ dimensions} $\Rd_m\K_n\subseteq \K_m$ is obtained from $\A$ by discarding all operators with indices $m\leq i<n$.
The \emph{neat reduct to $m$ dimensions}  is the algebra  $\Nr_m\A\in \K_m$ with universe $\set{a\in\A: m\leq i<n\rightarrow \cyl i a = a}$
where  all the operators are induced from $\A$ (see \cite[definition~2.6.28]{HMT1} for the $\CA$ case).

Let $\A\in\K_m, \; \B\in \K_n$.  An injective homomorphism $f:\A\rightarrow \B$ is a \emph{neat embedding} if the range of $f$
is a subalgebra of $\Nr_m\B$.
The notions of neat reducts and  neat embeddings have proved useful in
analyzing the number of variables needed in proofs,
as well as for proving representability results,
via the so-called neat embedding theorems.

A generalized space of dimension $n$ is a set of the form $\bigcup_{i\in I}{}^nU_i$ where $I$ is a non-empty indexing set,
$U_i\neq \emptyset$ for each $i\in I$ and $U_i\cap U_j=\emptyset$ 
for distinct $i,j\in I$. For any such space $V$ of dimension $n$, $\wp(V)$ 
with the Boolean operations of union, intersection and complementation, and cylindrifiers, substitution operators and diagonal elements 
(if in the signature) defined like set algebras is a 
$\K_n$; in fact an $\sf RK_n$. Conversely, for any $\A\in \sf RK_n$ there is a generalized space $V$ and an embedding 
from $f:\A\to \wp(V)$. We call $f$ a {\it a representation of $\A$}.

Let $\A\in \sf K_{n}$ and $f:\A\to \wp(V)$ be a representation of $\A$, where $V$ is 
a generalized space of dimension $n$.

A {\it complete representation} of $\A$ is a representation $f$ satisfying
$$f(\prod X)=\bigcap f[X]$$
whenever $X\subseteq \A$ and $\prod X$ is defined.

The action of the non-Boolean operators in a completely additive
atomic Boolean algebra with operators is determined by their behavior over the atoms, and
this in turn is encoded by
the atom structure of the algebra. We use $\sf BAO(s)$ short for Boolean algebra(s) with operators; operators are
operations that distribute
over the Boolean join in every component.

\begin{definition}(\textbf{Atom Structure})
Let $\A=\langle A, +,\cdot,  -, 0, 1, \Omega_{i}:i\in I\rangle$ be
an atomic Boolean algebra with operators $\Omega_{i}:i\in I$. Let
the rank of $\Omega_{i}$ be $\rho_{i}$. The \textit{atom structure}
$\At\A$ of $\A$ is a relational structure
$$\langle \At\A, R_{\Omega_{i}}:i\in I\rangle$$
where $\At\A$ is the set of atoms of $\A$, and $R_{\Omega_{i}}$ is a $(\rho(i)+1)$-ary relation over
$\At\A$ defined by
$$R_{\Omega_{i}}(a_{0},
\cdots, a_{\rho(i)})\Longleftrightarrow\Omega_{i}(a_{1}, \cdots,
a_{\rho(i)})\geq a_{0}.$$
\end{definition}
Similar 'dual' structure arise in other ways, too. For any not
necessarily atomic $\sf BAO$ $\A$,  its
\textit{ultrafilter frame} is the
structure
$$\A_{+}=\langle {\sf Uf}(\A),R_{\Omega_{i}}:i\in I\rangle,$$
 where ${\sf Uf}(\A)$
is the set of all ultrafilters of (the Boolean reduct of)
$\A$, and for $\mu_{0}, \cdots, \mu_{\rho(i)}\in
{\sf Uf}(\A)$, we put $R_{\Omega_{i}}(\mu_{0}, \cdots,
\mu_{\rho(i)})$ iff $\{\Omega(a_{1}, \cdots,
a_{\rho(i)}):a_{j}\in\mu_{j}$ for
$0<j\leq\rho(i)\}\subseteq\mu_{0}$.
\begin{definition}(\textbf{Complex algebra})
Conversely, if we are given an arbitrary structure
$\mathcal{S}=\langle S, r_{i}:i\in I\rangle$ where $r_{i}$ is a
$(\rho(i)+1)$-ary relation over $S$, we can define its
\textit{complex
algebra}
$$\Cm(\mathcal{S})=\langle \wp(S),
\cup, \setminus, \phi, S, \Omega_{i}\rangle_{i\in
I},$$
where $\wp(S)$ is the power set of $S$, and
$\Omega_{i}$ is the $\rho(i)$-ary operator defined
by$$\Omega_{i}(X_{1}, \cdots, X_{\rho(i)})=\{s\in
S:\exists s_{1}\in X_{1}\cdots\exists s_{\rho(i)}\in X_{\rho(i)},
r_{i}(s, s_{1}, \cdots, s_{\rho(i)})\},$$ for each
$X_{1}, \cdots, X_{\rho(i)}\in\wp(S)$.
\end{definition}
It is easy to check that, up to isomorphism,
$\At(\Cm(\mathcal{S}))\cong\mathcal{S}$ always, and
$\A\subseteq\Cm(\At\A)$ for any
completely additive atomic Boolean algebra with operators $\A$. If $\A$ is
finite then of course
$\A\cong\Cm(\At\A)$. For an atomic algebra $\A$, $\Tm\At\A$ denotes the {\it term algebra}; 
it is the subalgbra of $\Cm\At\A$ genertaed by the atoms, the two algebras share
the same atom structure. If $\A$ is completely additive then $\Tm\At\A\subseteq \A\subseteq \Cm\At\A$.\\

A variety $V$ of $\sf BAO$s
is {\it atom-canonical} if whenever $\A\in V$, and $\A$ is atomic, then $\Cm\At\A\in V$.
If $V$ is completely additive, then $\Cm\At\A$ is the \d\ completion of $\A$.
$\A$ is always dense in its \d\ completion and so the \d\ completion of $\A$ is atomic iff $\A$ is atomic.

Not all varieties we will encounter are completely additive; the varieties $\Sc_n$ and $\PA_n$ for $n>1$,
are not, but \d\ completions of atomic completely additive algebras remain in the 
variety, which is not the case with the variety of representable algebras 
as shown below.

The canonical extension of $\A$, namely, $\Cm(\Uf(\A))$ will be denoted by $\A^+$.
This algebra is always
complete and atomic. $\A$ embeds into $\A^+$ under the usual `Stone representability' function mapping $a\in \A$
to the set of all Boolean ultrafilters  of $\A$ containing $a$. The algebra
$\A^+$ is the \d\ completion of $\A$ if 
and only if $\A$ is finite.

For other operators on classes of algebras, $\bold H$ denotes closure under forming homomorphic images
and ${\bf Up}$ under ultraproducts.

For a class $\sf K$, and an algebra $\A$ of the same signature as
$\sf K$, $\A\in {\bf Ur}\bf K$ iff an ultrapower of $\A$ is in $\sf K$.
${\bf Up Ur}\K$ is the elementary closure of $\K$ that is the least elementary class containing $\K$;
$\K$ is elementary or synonymously
first order definable, if it is closed under ultraproducts and  ultraroots. More succinctly
$\K$ is elementary if ${\bf UpUr}\K=\K$.
$\K$ is a quasi-variety if it is closed under ${\bf SPUp}$ and a variety if it is closed
under ${\bf HSP}$. If $\sf K$ is a variety then it can be axiomatized by a set of equations; if its only a quasi-variety, then
it can be axiomatized by a set quasi-equations.
$\sf K$ is elementary if it is definable by a set of first order sentences.

\section{Rainbows, atom-canonicity}

We will show using the so called {\it blow up and blur construction}, a very indicative name suggested 
in \cite{ANT},
that for any finite $n>2$, any $\sf \K\in \{\Sc, \CA, \PA, \PEA\}$, and any 
$k\geq 3$,
$S\Nr_n\sf K_{n+k}$ is {\it not} atom canonical. We will blow up and blur a finite 
{\it rainbow algebra}.

We give the general idea for cylindric algebras, though the idea is much more universal as we will see.
Assume that $\sf RCA_n\subseteq \K$, and $\K$ is closed under forming subalgebras.
Start with a finite algebra $\C$ outside $\K$. Blow up and blur $\C$, by splitting
each atom to infinitely many, getting a new atom structure $\At$. In this process a (finite) set of blurs are used.

They do not blur the complex algebra, in the sense that $\C$ is there on this global level.
The algebra $\mathfrak{Cm}\At$ will not be in $\K$
because $\C\notin \K$ and $\C$ embeds into $\mathfrak{Cm}\At$.
Here the completeness of the complex algebra will play a major role,
because every element of $\C$,  is mapped, roughly, to {\it the join} of
its splitted copies which exist in $\mathfrak{Cm}\At$ because it is complete.

These precarious joins prohibiting membership in $\K$ {\it do not }exist in the term algebra, only finite-cofinite joins do,
so that the blurs blur $\C$ on 
this level; $\C$ does not embed in $\Tm\At.$

In fact, the the term algebra will  not only be in $\K$, but actually it will be in the possibly smaller $\sf RCA_n$.
This is where the blurs play their other role. Basically non-principal ultrafilters, the blurs
are used as colours to represent  $\Tm\At$.

In the process of representation we cannot use {\it only} principal ultrafilters,
because $\Tm\At$ {\it cannot be completely representable}, that is, it cannot have a representation that preserves
all (possibly infinitary) meets carrying them to set theoretic intersections, 
for otherwise this would give that $\mathfrak{Cm}\At$
is representable.

But the blurs will actually provide a {\it complete representation} of the {\it canonical extension}
of $\Tm\At$, in symbols $\Tm\At^+$; the algebra whose underlying set consists of all ultrafilters of $\Tm\At$. The atoms of $\Tm\At$
are coded in the principal ones,  and the remaining non- principal ultrafilters, or the blurs,
will be finite, used as colours to completely represent $\Tm\At^+$, in the process representing $\Tm\At$.

We start off with a conditional theorem: giving a concrete instance of a blow up and blur construction for relation algebras due to Hirsch and Hodkinson. 
The proof is terse  highlighting only the main ideas.
\begin{theorem}\label{decidability} Let $m\geq 3$. Assume that for any simple atomic relation algebra with atom structure $S$,
there is a cylindric atom structure $H$, constructed effectively from $S$,  such that:
\begin{enumarab}
\item If $\Tm S\in \sf RRA$, then $\Tm H\in \RCA_m$,
\item If $S$ is finite, then $H$ is finite,
\item $\Cm S$ is embeddable in $\Ra$ reduct of $\Cm H$.
\end{enumarab}
Then  for all $k\geq 3$, $S\Nr_m\CA_{m+k}$ is not closed under completions,
\end{theorem}

\begin{proof}
Let $S$ be a relation atom structure such that $\Tm S$ is representable while $\Cm S\notin \sf RA_6$.
Such an atom structure exists \cite[Lemmas 17.34-35-36-37]{HHbook}.

We give a brief sketch at how such algebras are constructed by allowing complete irreflexive graphs having an arbitrary finite set
nodes, slightly generalizing the proof in {\it op.cit}, though the proof idea is essentially the same.

Another change is that we refer to non-principal ultrafilters (intentionally) by {\it blurs} to emphasize the connection with the blow up and
blur construction in \cite{ANT} as well as with the blow up and blur construction outlined above, to be encountered 
in full detail in a litle while, witness theorem \ref{can}.

In all cases a finite algebra is blown up and blurred to give a representable algebra (the term algebra on the blown up and blurred finite atom
structure) whose \d\ completion does not have a neat embedding
property.

We use the notation of the above cited lemmas in \cite{HHbook2} without warning, and our proof will be very brief just stressing the main ideas.
$G_r^n$ denotes the usual atomic $r$ rounded 
game played on atomic networks having $n$ nodes of an atomic relation algebra, 
where $n,r\leq \omega$, and  $K_r$ ($r\in \omega)$ 
denotes the complete ireflexive graph with
$r$ nodes. 

Let $\R$ be the rainbow algebra $\A_{K_m, K_n}$, $m>n>2$.
Let $T$ be the term algebra obtained by splitting the reds. Then $T$ has exactly two blurs
$\delta$ and $\rho$. $\rho$ is a flexible non-principal ultrafilter consisting of reds with distinct indices and $\delta$ is the reds
with common indices.
Furthermore, $T$ is representable, but $\Cm\At T\notin  S\Ra\CA_{m+2}$, in particular, it is not representable 
\cite [Lemma 17.32]{HHbook2}.
Now we use  the \ef\ forth pebble game $EF_r^k(A, B)$
where $A$ and $B$ are relational structures. This game has 
$r$ rounds and $k$ pebbles. 
The rainbow theorem \cite[Theorem 16.5]{HHbook}
says that \pe\ has a \ws\ in the game $G_{1+r}^{2+p}(\A_{A, B})$ if and only if she has a \ws\ in $EF_r^p(A,B)$.

Using this theorem it is obvious that  \pe\ has a \ws\ over $\At\R$  in $m+2$ rounds,
hence $\R\notin \sf RA_{m+2}$, hence is not
in $S\Ra\CA_{m+2}$.  $\Cm\At T$ is also not in the latter class
for $\R$ embeds into it, by mapping every red
to the join of its copies. Let $D=\{r_{ll}^n: n<\omega, l\in n\}$, and $R=\{r_{lm}^n, l,m\in n, l\neq m\}$.
If $X\subseteq R$, then $X\in T$ if and only if $X$ is finite or cofinite in $R$ and same for subsets of $D$ \cite[lemma 17.35]{HHbook}.
Let $\delta=\{X\in T: X\cap D \text { is cofinite in $D$}\}$,
and $\rho=\{X\in T: X\cap R\text { is cofinite in $R$}\}$.
Then these are {\it the} non principal ultrafilters, they are the blurs and they are enough
to (to be used as colours), together with the principal ones, to represent $T$ as follows \cite[bottom of p. 533]{HHbook2}.
Let $\Delta$ be the graph $n\times \omega\cup m\times \{{\omega}\}$.
Let $\B$ be the full rainbow algebras over $\At\A_{K_m, \Delta}$ by
deleting all red atoms $r_{ij}$ where $i,j$ are
in different connected components of $\Delta$.

Obviously \pe\ has a \ws\ in ${\sf EF}_{\omega}^{\omega}(K_m, K_m)$, and so it has a \ws\ in
$G_{\omega}^{\omega}(\A_{K_m, K_m})$.
But $\At\A_{K_m, K_m}\subseteq \At\B\subseteq \At_{K_m, \Delta}$, and so $\B$ is representable.

One  then defines a bounded morphism from $\At\B$ to the the canonical extension
of $T$, which we denote by $T^+$, consisting of all ultrafilters of $T$. The blurs are images of elements from
$K_m\times \{\omega\}$, by mapping the red with equal double index,
to $\delta$, for distinct indices to $\rho$.
The first copy is reserved to define the rest of the red atoms the obvious way.
(The underlying idea is that this graph codes the principal ultrafilters in the first component, and the non principal ones in the second.)
The other atoms are the same in both structures. Let $S=\Cm\At T$, then $\Cm S\notin S\Ra\CA_{m+2}$ \cite[lemma 17.36]{HHbook}.

Note here that the \d\ completion of $T$ is not representable while its canonical extension is
{\it completely representable}, via the representation defined above.
However, $T$ itself is {\it not} completely representable, for a complete representation of $T$ induces a representation of its \d\ completion,
namely, $\Cm\At\A$.

Now let $H$ be the $\CA_m$ atom structure obtained from $\At T$ provided by the hypothesis of the 
theorem.
Then $\Tm H\in \RCA_m$. We claim that $\Cm H\notin {\bf S}\Nr_m\CA_{m+k}$, $k\geq 3$.
For assume not, i.e. assume that $\Cm H\in {\bf S}\Nr_m\CA_{m+k}$, $k\geq 3$.
We have $\Cm S$ is embeddable in $\Ra\Cm H.$  But then the latter is in ${\bf S}\Ra\CA_6$
and so is $\Cm S$, which is not the case.
\end{proof}

Hodkinson constructs atom structures of cylindric and polyadic algebras of any pre-assigned finite dimension $>2$ 
from atom structures
of relation algebras \cite{AU}. One could well be tempted to use such a construction with the above proof to obtain an analogous
result for cylindric and polyadic algebras. 
However, we emphasize that the next result {\it cannot} be obtained by lifting the relation algebra case 
\cite[ lemmas 17.32, 17.34, 17.35, 17.36]{HHbook}
to cylindric algebras 
using  Hodkinson's construction in \cite{AU} as it stands. It is true that Hodkinson constructs from 
every atomic relation algebra an atomic cylindric algebra of dimension $n$, for any $n\geq 3$, 
but the relation algebras {\it does not} embed into the $\sf Ra$ reduct of the constructed
cylindric algebra when $n\geq 6$. If it did, then the $\sf Ra$ result would lift as indeed is the case with 
$n=3$ \cite{snr}. Now we are faced with two options. Either modify 
Hodkinson's construction, implying that the embeddability of the given relation algebra
in the $\sf Ra$ reduct of the constructed cylindric algebras, or avoid completely 
the route via relation algebras. We tend to think that it is impossible to adapt Hodkinson's construction the way needed, because if $\A$
is a non representable relation algebra, and $\Cm \At(\A)$ embeds into the $\Ra$ reduct of a cylindric algebra of every dimension $>2$, then
$\A$ will be representable, which is a contradiction. 

Therefore we choose the second option. We instead start from scratch. We blow up and blur a finite  rainbow cylindric algebra.

In \cite{HHbook2} the rainbow cylindric algebra of dimension $n$ on a graph $\Gamma$ is denoted by $\R(\Gamma)$.
We consider $\R(\Gamma)$ to be in ${\sf PEA}_n$ by expanding it with the polyadic operations defined the obvious way (see below).
In what follows we consider $\Gamma$ to be the indices of the reds, and for a complete irreflexive graph
$\G$, by ${\sf PEA}_{\G, \Gamma}$
we mean the rainbow cylindric algebra $\R(\Gamma)$  of dimension $n$,
where ${\sf G}=\{\g_i: 1\leq i<n-1\}\cup \{\g_0^i: i\in \G\}$.

More generally, we consider a rainbow polyadic algebra based on relational structures 
$A, B$, to be the rainbow algebra
with signature the binary
colours (binary relation symbols)
$\{\r_{ij}: i,j\in B\}\cup \{\w_i: i<n-1\}\cup \{\g_i:1\leq i<n-1\}\cup \{\g_0^i : i\in A\}$ and $n-1$
shades of yellow ($n-1$ ary relation symbols)
$\{\y_S: S\subseteq_{\omega} A, \text { or } S=A\}.$ 

We look at models of the rainbow theorem as coloured graphs \cite{HH}.
This class is denoted by ${\sf CRG}_{A,B}$ or simply $\sf CRG$ when $A$ and $B$ are clear from context.

A coloured graph is a graph such that each of its edges is labelled by one of the first three colours mentioned above, namely, 
greens, whites or reds, and some $n-1$ hyperedges are also
labelled by the shades of yellow.
Certain coloured graphs will deserve 
special attention.

\begin{definition}
Let $i\in A$, and let $M$ be a coloured graph  consisting of $n$ nodes
$x_0,\ldots,  x_{n-2}, z$. We call $M$ an {\it $i$ - cone} if $M(x_0, z)=\g_0^i$
and for every $1\leq j\leq n-2$, $M(x_j, z)=\g_j$,
and no other edge of $M$
is coloured green.
$(x_0,\ldots, x_{n-2})$
is called {\it the center of the cone}, $z$ {\it the apex of the cone}
and {\it $i$ the tint of the cone.}
\end{definition}

\begin{definition}\label{def}
The class of coloured graphs $\sf CRG$ are

\begin{itemize}

\item $M$ is a complete graph.

\item $M$ contains no triangles (called forbidden triples)
of the following types:
\vspace{-.2in}
\begin{eqnarray}
&&\nonumber\\
(\g, \g^{'}, \g^{*}), (\g_i, \g_{i}, \w_i)
&&\mbox{any }1\leq i< n-1\  \\
(\g^j_0, \g^k_0, \w_0)&&\mbox{ any } j, k\in A\\
\label{forb:match}(\r_{ij}, \r_{j'k'}, \r_{i^*k^*})&& i,j,j',k',i^*, k^*\in B\\ \mbox{unless }i=i^*,\; j=j'\mbox{ and }k'=k^*
\end{eqnarray}
and no other triple of atoms is forbidden.

\item If $a_0,\ldots,   a_{n-2}\in M$ are distinct, and no edge $(a_i, a_j)$ $i<j<n$
is coloured green, then the sequence $(a_0, \ldots, a_{n-2})$
is coloured a unique shade of yellow.
No other $(n-1)$ tuples are coloured shades of yellow.

\item If $D=\set{d_0,\ldots,  d_{n-2}, \delta}\subseteq M$ and
$M\upharpoonright D$ is an $i$ cone with apex $\delta$, inducing the order
$d_0,\ldots,  d_{n-2}$ on its base, and the tuple
$(d_0,\ldots, d_{n-2})$ is coloured by a unique shade
$\y_S$ then $i\in S.$
\end{itemize}

One then can define a polyadic equality atom structure
of dimension $n$ from the class $\sf CRG$. It is a {\it rainbow atom structure}.  Rainbow atom structures  are what Hirsch and Hodkinson call
atom structures built from a class of models \cite{HHbook2}.
Our models are, according to the original more traditional view \cite{HH}
coloured graphs. So let $\sf CRG$ be the class of coloured graphs as defined above.
Let $$\At=\{a:n \to M, M\in \sf CRG: \text { $a$ is surjective}\}.$$
We write $M_a$ for the element of $\At$ for which
$a:n\to M$ is a surjection.
Let $a, b\in \At$ define the
following equivalence relation: $a \sim b$ if and only if
\begin{itemize}
\item $a(i)=a(j)\Longleftrightarrow b(i)=b(j),$

\item $M_a(a(i), a(j))=M_b(b(i), b(j))$ whenever defined,

\item $M_a(a(k_0),\dots, a(k_{n-2}))=M_b(b(k_0),\ldots, b(k_{n-2}))$ whenever
defined.
\end{itemize}
Let $\At$ be the set of equivalences classes. Then define
$$[a]\in E_{ij} \text { iff } a(i)=a(j).$$
$$[a]T_i[b] \text { iff }a\upharpoonright n\smallsetminus \{i\}=b\upharpoonright n\smallsetminus \{i\}.$$
Define accessibility relations corresponding to the polyadic (transpositions) operations as follows:
$$[a]S_{ij}[b] \text { iff } a\circ [i,j]=b.$$
This, as easily checked, defines a $\sf PEA_n$
atom structure. The complex algebra of this atom structure is denoted by ${\sf PEA}_{A, B}$ where $A$ is the greens and 
$B$ is the reds.
\end{definition}
Consider the following two games on coloured graphs, each with $\omega$ rounds, and limited number of pebbles
$m>n$. They are translations of $\omega$ atomic games played on atomic networks
of a rainbow algebra using a limited number of nodes $m$.
Both games offer \pa\ only one move, namely, a cylindrifier move.

From the graph game perspective both games \cite[p.27-29]{HH} build a nested sequence $M_0\subseteq M_1\subseteq \ldots $.
of coloured graphs.

First game $G^m$.
\pa\ picks a graph $M_0\in \sf CRG$ with $M_0\subseteq m$ and
$\exists$ makes no response
to this move. In a subsequent round, let the last graph built be $M_i$.
\pa\ picks
\begin{itemize}
\item a graph $\Phi\in \sf CRG$ with $|\Phi|=n,$
\item a single node $k\in \Phi,$
\item a coloured graph embedding $\theta:\Phi\smallsetminus \{k\}\to M_i.$
Let $F=\phi\smallsetminus \{k\}$. Then $F$ is called a face.
\pe\ must respond by amalgamating
$M_i$ and $\Phi$ with the embedding $\theta$. In other words she has to define a
graph $M_{i+1}\in C$ and embeddings $\lambda:M_i\to M_{i+1}$
$\mu:\phi \to M_{i+1}$, such that $\lambda\circ \theta=\mu\upharpoonright F.$
\end{itemize}
$F^m$ is like $G^m$, but \pa\ is allowed to resuse nodes.

$F^m$ has an equivalent formulation on atomic networks of atomic algebras.

Let $\delta$ be a map. Then $\delta[i\to d]$ is defined as follows. $\delta[i\to d](x)=\delta(x)$
if $x\neq i$ and $\delta[i\to d](i)=d$. We write $\delta_i^j$ for $\delta[i\to \delta_j]$.

\begin{definition}
Let $2< n<\omega.$ Let $\C$ be an atomic ${\sf PEA}_{n}$.
An \emph{atomic  network} over $\C$ is a map
$$N: {}^{n}\Delta\to \At\C,$$
where $\Delta$ is a non-empty set called a set of nodes,
such that the following hold for each $i,j<n$, $\delta\in {}^{n}\Delta$
and $d\in \Delta$:
\begin{itemize}
\item $N(\delta^i_j)\leq {\sf d}_{ij},$
\item $N(\delta[i\to d])\leq {\sf c}_iN(\delta),$
\item $N(\bar{x}\circ [i,j])= {\sf s}_{[i,j]}N(\bar{x})$ for all $i,j<n$.

\end{itemize}
\end{definition}
\begin{definition}\label{def:games}
Let $2\leq n<\omega$. For any ${\sf Sc}_n$
atom structure $\alpha$ and $n<m\leq
\omega$, we define a two-player game $F^m(\alpha)$,
each with $\omega$ rounds.

Let $m\leq \omega$.
In a play of $F^m(\alpha)$ the two players construct a sequence of
networks $N_0, N_1,\ldots$ where $\nodes(N_i)$ is a finite subset of
$m=\set{j:j<m}$, for each $i$.

In the initial round of this game \pa\
picks any atom $a\in\alpha$ and \pe\ must play a finite network $N_0$ with
$\nodes(N_0)\subseteq  m$,
such that $N_0(\bar{d}) = a$
for some $\bar{d}\in{}^{n}\nodes(N_0)$.

In a subsequent round of a play of $F^m(\alpha)$, \pa\ can pick a
previously played network $N$ an index $l<n$, a {\it face}
$F=\langle f_0,\ldots, f_{n-2} \rangle \in{}^{n-2}\nodes(N),\; k\in
m\sim\set{f_0,\ldots, f_{n-2}}$, and an atom $b\in\alpha$ such that
$$b\leq {\sf c}_lN(f_0,\ldots, f_i, x,\ldots, f_{n-2}).$$
The choice of $x$ here is arbitrary,
as the second part of the definition of an atomic network together with the fact
that $\cyl i(\cyl i x)=\cyl ix$ ensures that the right hand side does not depend on $x$.

This move is called a \emph{cylindrifier move} and is denoted
$$(N, \langle f_0, \ldots, f_{n-2}\rangle, k, b, l)$$
or simply by $(N, F,k, b, l)$.
In order to make a legal response, \pe\ must play a
network $M\supseteq N$ such that
$M(f_0,\ldots, f_{i-1}, k, f_{i+1},\ldots, f_{n-2}))=b$
and $\nodes(M)=\nodes(N)\cup\set k$.

\pe\ wins $F^m(\alpha)$ if she responds with a legal move in each of the
$\omega$ rounds.  If she fails to make a legal response in any
round then \pa\ wins.
\end{definition}

Next we adapt certain notions worked out for relation algebras in \cite{r} to the $\CA$ 
context culminating in theorem \ref{thm:n} that will be used several times,
to show that certain constructed algebras do not have a neat embedding property.

\begin{definition}\label{subs}
Let $n$ be an ordinal. An $s$ word is a finite string of substitutions $({\sf s}_i^j)$,
a $c$ word is a finite string of cylindrifications $({\sf c}_k)$.
An $sc$ word is a finite string of substitutions and cylindrifications.
Any $sc$ word $w$ induces a partial map $\hat{w}:n\to n$
by
\begin{itemize}

\item $\hat{\epsilon}=Id,$

\item $\widehat{w_j^i}=\hat{w}\circ [i|j],$

\item $\widehat{w{\sf c}_i}= \hat{w}\upharpoonright(n\smallsetminus \{i\}).$

\end{itemize}
\end{definition}

If $\bar a\in {}^{<n-1}n$, we write ${\sf s}_{\bar a}$, or more frequently
${\sf s}_{a_0\ldots a_{k-1}}$, where $k=|\bar a|$,
for an an arbitrary chosen $sc$ word $w$
such that $\hat{w}=\bar a.$
$w$  exists and does not
depend on $w$ by \cite[definition~5.23 ~lemma 13.29]{HHbook}.
We can, and will assume \cite[Lemma 13.29]{HHbook}
that $w=s{\sf c}_{n-1}{\sf c}_n.$
[In the notation of \cite[definition~5.23,~lemma~13.29]{HHbook},
$\widehat{s_{ijk}}$ for example is the function $n\to n$ taking $0$ to $i,$
$1$ to $j$ and $2$ to $k$, and fixing all $l\in n\setminus\set{i, j,k}$.]

We need some more technical lemmas which are  cylindric analogues of lemmas formulated for relation algebras
in \cite{r}.

The next definition is  formulated  for the least reduct ${\sf Sc}$s of $\PEA$s, so that it applies to all its expansions studied here.

\begin{definition}\label{def:hat}
For $m\geq 5$ and $\C\in\Sc_m$, if $\A\subseteq\Nr_n\C$ is an
atomic $\Sc_n$ and $N$ is an $\A$-network with $\nodes(N)\subseteq m$, then we define
$\widehat N\in\C$ by
\[\widehat N =
 \prod_{i_0,\ldots, i_{n-1}\in\nodes(N)}{\sf s}_{i_0, \ldots, i_{n-1}}N(i_0,\ldots, i_{n-1})\]
$\widehat N\in\C$ depends
implicitly on $\C$.
\end{definition}

In what follows we write $\A\subseteq_c \B$ if $\A$ is a complete subalgebra of $\B$, that is, if $\A\subseteq \B$ and 
whenever $X\subseteq \A$ is such that $\sum ^{\A}X=1$, then $\sum ^{\B}X=1$.

\begin{lemma}\label{lem:atoms2}
Let $n<m$ and let $\A$ be an atomic $\sf Sc_n$,
$\A\subseteq_c\Nr_n\C$
for some $\C\in\Sc_m$.  For all $x\in\C\setminus\set0$ and all $i_0, \ldots, i_{n-1} < m$ there is $a\in\At(\A)$ such that
${\sf s}_{i_0,\ldots, i_{n-1}}a\;.\; x\neq 0$.
\end{lemma}
\begin{proof}
We can assume, see definition  \ref{subs},
that ${\sf s}_{i_0,\ldots, i_{n-1}}$ consists only of substitutions, since ${\sf c}_{m}{\sf c}_{m-1}\ldots
{\sf c}_nx=x$
for every $x\in \A$. We have ${\sf s}^i_j$ is a
completely additive operator (any $i, j$), hence ${\sf s}_{i_0,\ldots, i_{n-1}}$
is too  (see definition~\ref{subs}).
So $\sum\set{{\sf s}_{i_0\ldots, i_{n-1}}a:a\in\At(\A)}={\sf s}_{i_0\ldots i_{n-1}}
\sum\At(\A)={\sf s}_{i_0\ldots, i_{n-1}}1=1$,
for any $i_0,\ldots, i_{n-1}<n$.  Let $x\in\C\setminus\set0$.  It is impossible
that ${\sf s}_{i_0\ldots, i_{n-1}}\;.\;x=0$ for all $a\in\At(\A)$ because this would
imply that $1-x$ was an upper bound for $\set{{\sf s}_{i_0\ldots i_{n-1}}a:
a\in\At(\A)}$, contradicting $\sum\set{{\sf s}_{i_0\ldots, i_{n-1}}a :a\in\At(\A)}=1$.
\end{proof}

For networks $M, N$ and any set $S$, we write $M\equiv^SN$
if $N\restr S=M\restr S$, and we write $M\equiv_SN$
if the symmetric difference $\Delta(\nodes(M), \nodes(N))\subseteq S$ and
$M\equiv^{(\nodes(M)\cup\nodes(N))\setminus S}N$. We write $M\equiv_kN$ for
$M\equiv_{\set k}N$.

We write $Id_{-i}$ for the function $\{(k,k): k\in n\smallsetminus\{i\}\}.$
For a network $N$ and a partial map $\theta$ from $n$ to $n$, that is $\dom\theta\subseteq n$, then
$N\theta$ is the network whose labelling is defined by $N\theta(\bar{x})=N(\tau(\bar{x}))$
where for $i\in n$, $\tau(i)=\theta(i)$ for $i\in \dom\theta$
and $\tau(i)=i$ otherwise.
Recall that $F^m$ is the usual atomic game
on networks, except that the nodes are $m$ and \pa\ can re use nodes.
Then:

\begin{theorem}\label{thm:n}
Let $\K$ be any class between $\Sc$ and $\PEA$. Let $n<m$, and let $\A$ be an atomic $\K_n.$
If $\A\in S_c\Nr_{n}\K_m, $
then \pe\ has a \ws\ in $F^m(\At\A)$ (the latter involves only cylindrifier moves so it applies to algebras in $\K$).
In particular, if $\A$ is completely representable, then \pe\ has a \ws\ in $F^{\omega}(\At\A).$
\end{theorem}
\begin{proof}
The proof of the first part is based on repeated use of
lemma ~\ref{lem:atoms2}. We first show (*):
\begin{enumarab}
\item For any $x\in\C\setminus\set0$ and any
finite set $I\subseteq m$ there is a network $N$ such that
$\nodes(N)=I$ and $x\;.\;\widehat N\neq 0$.
\item
For any networks $M, N$ if
$\widehat M\;.\;\widehat N\neq 0$ then $M\equiv^{\nodes(M)\cap\nodes(N)}N$.
\end{enumarab}
We define the edge labelling of $N$ one edge
at a time. Initially no hyperedges are labelled.  Suppose
$E\subseteq\nodes(N)\times\nodes(N)\ldots  \times\nodes(N)$ is the set of labelled hyper
edges of $N$ (initially $E=\emptyset$) and
$x\;.\;\prod_{\bar c \in E}{\sf s}_{\bar c}N(\bar c)\neq 0$.  Pick $\bar d$ such that $\bar d\not\in E$.
Then there is $a\in\At(\A)$ such that
$x\;.\;\prod_{\bar c\in E}{\sf s}_{\bar c}N(\bar c)\;.\;{\sf s}_{\bar d}a\neq 0$.

Include the edge $\bar d$ in $E$.  Eventually, all edges will be
labelled, so we obtain a completely labelled graph $N$ with $\widehat
N\neq 0$.
it is easily checked that $N$ is a network.

For the second part, if it is not true that
$M\equiv^{\nodes(M)\cap\nodes(N)}N$, then there are is
$\bar c \in{}\nodes(M)\cap\nodes(N)$ such that $M(\bar c )\neq N(\bar c)$.
Since edges are labelled by atoms we have $M(\bar c)\cdot N(\bar c)=0,$
so $0={\sf s}_{\bar c}0={\sf s}_{\bar c}M(\bar c)\;.\; {\sf s}_{\bar c}N(\bar c)\geq \widehat M\;.\;\widehat N$.

Next, we show that (**):

\begin{enumerate}
\item\label{it:-i}
If $i\not\in\nodes(N)$ then ${\sf c}_i\widehat N=\widehat N$.

\item \label{it:-j} $\widehat{N Id_{-j}}\geq \widehat N$.

\item\label{it:ij} If $i\not\in\nodes(N)$ and $j\in\nodes(N)$ then
$\widehat N\neq 0 \rightarrow \widehat{N[i/j]}\neq 0$,
where $N[i/j]=N\circ [i|j]$

\item\label{it:theta} If $\theta$ is any partial, finite map $n\to n$
and if $\nodes(N)$ is a proper subset of $n$,
then $\widehat N\neq 0\rightarrow \widehat{N\theta}\neq 0$.
\end{enumerate}

The first part is easy.
The second part is by definition of $\;\widehat{\;}$. For the third part, suppose
$\widehat N\neq 0$.  Since $i\not\in\nodes(N)$, by part~\ref{it:-i},
we have ${\sf c}_i\widehat N=\widehat N$.  By cylindric algebra axioms it
follows that $\widehat N\cdot {\sf d}_{ij}\neq 0$.  From the above
there is a network $M$ where $\nodes(M)=\nodes(N)\cup\set i$ such that
$\widehat M\cdot\widehat N\cdot {\sf d}_{ij}\neq 0$.  From the first part, we
have $M\supseteq N$ and $M=N[i/j]$.
Hence $\widehat{N[i/j]}\neq 0$.

For the final part
(cf. \cite[lemma~13.29]{HHbook}), since there is
$k\in n\setminus\nodes(N)$, \/ $\theta$ can be
expressed as a product $\sigma_0\sigma_1\ldots\sigma_t$ of maps such
that, for $s\leq t$, we have either $\sigma_s=Id_{-i}$ for some $i<n$
or $\sigma_s=[i/j]$ for some $i, j<n$ and where
$i\not\in\nodes(N\sigma_0\ldots\sigma_{s-1})$.
Now apply parts~\ref{it:-j} and \ref{it:ij}.

Now we prove the required. Suppose that
$\A\subseteq_c\Nr_n\C$ for some $\C\in\K_m,$ then \pe\ always
plays networks $N$ with $\nodes(N)\subseteq n$ such that
$\widehat N\neq 0$. In more detail, in the initial round, let \pa\ play $a\in \At\A$.
\pe\ plays a network $N$ with $N(0, \ldots, n-1)=a$. Then $\widehat N=a\neq 0$.
At a later stage suppose \pa\ plays the cylindrifier move
$(N, \langle f_0, \ldots, f_{n-2}\rangle, k, b, l)$
by picking a
previously played network $N$ and $f_i\in \nodes(N), \;l<n,  k\notin \{f_i: i<n-2\}$,
and $b\leq {\sf c}_k N(f_0,\ldots,  f_{i-1}, x, f_{i+1}, \ldots, f_{n-2})$.

Let $\bar a=\langle f_0\ldots f_{i-1}, k, f_{i+1}, \ldots f_{n-2}\rangle.$
Then by  (*), we have that ${\sf c}_k\widehat N\cdot {\sf s}_{\bar a}b\neq 0$
and so by item 1 in (**),   there is a network  $M$ such that
$\widehat{M}\cdot\widehat{{\sf c}_kN}\cdot {\sf s}_{\bar a}b\neq 0$.
Hence
$$M(f_0,\dots, f_{i-1}, k, f_{i-2}, \ldots, f_{n-2})=b,$$ and $M$ is the required response.

The last part follows from the fact that if $\A$ is completely representable,
then $\A\in S_c\Nr_n\sf QEA_{\omega}$, witness theorem \ref{complete} below.
\end{proof}
We shall also need:
\begin{theorem}\label{AU} Let $2<n<\omega$. Let $\K\in \{\Df, \Sc, \CA, \PA, \PEA\}$. 
If $\A\in \K_n$ is generated by $\{a\in \A: \Delta a\neq n\}$, then
$\A$ is completely representable if and only if $\Rd_{df}\A$ is completely representable.
\end{theorem}
\begin{proof} \cite[Proposition 4.10]{AU}.
\end{proof}
The idea of the proof of our next theorem is summarized in the following:

\begin{enumarab}

\item  We construct a labelled hypergraph $M$ that can be viewed as an $n$ homogeneous  model of a certain 
first order rainbow theory. This model gets its labels from a rainbow signature.
By $n$ homogeneous we mean that every partial isomorphism of $M$ of size $\leq n$ can be extended
to an automorphism of $M$.

\item We have a shade of red $\rho$, outside the signature; this is basically {\it the blur; a non- principal ultrafilter}, 
in the term algebra which consists of finite and cofinite
sets on a countable set, but
$\rho$  can be used as a label.
\item We build a relativized set algebra based on $M$, by discarding all assignments whose edges are labelled
by the shade of red getting $W\subseteq {}^nM$.

\item $W$ is definable in $^nM$ be an $L_{\infty, \omega}$ formula
hence the semantics with respect to  $W$ coincides with classical Tarskian semantics (when assignments are in $^nM$).
This is proved using certain $n$ back and forth systems.

\item The set algebra $\A$ based on $W$ (consisting of sets of sequences (without shades of reds labelling edges)
satisfying formulas in $L^n$ in the given signature)
will be an atomic simple representable algebra such that its completion is not representable; the atoms will be coloured graphs.
The completion $\C$, which is the complex algebra of the rainbow atom structure, 
will consist of interpretations of $L_{\infty,\omega}^n$ formulas; though represented over $W$,
it will not be, and cannot be, representable in the classical  sense. In fact, its $\Sc$ reduct will be outside $S\Nr_n\Sc_{n+3}$.

\item The last part in the previous item is proved by embedding a finite rainbow polyadic equality algebra whose $\Sc$ reduct is not in 
$S\Nr_n\Sc_{n+3}$ in $\C$; so that the 
term algebra which is obtained by blowing up and blurring this finite rainbow algebra is representable, 
while the $\Sc$ reduct of $\C$, the \d\ completion of $\A$, is outside
$S\Nr_n\Sc_{n+3}$, too.
\end{enumarab}

In the next proof some parts overlap with parts in \cite{Hodkinson}; 
we include them to make the proof self contained as much as possible 
referring to \cite{Hodkinson} every time we do this.

\begin{theorem}\label{can} Let $n$ be finite $>2$. Then there exists a countable representable atomic $\A\in \PEA_n$ such that
$\Rd_{sc}\Cm\At\A$ is not in $S\Nr_n\Sc_{n+3}$, and $\Rd_{df}\A$ is not completely representable. 
In particular, for any finite $n>2$, any ${\sf K}\in \{\Sc, \sf PA, \sf PEA, \sf CA\}$ any class
between $S\Nr_n{\sf K}_{n+3}$ and ${\sf RK}_n$ is not atom-canonical.
\end{theorem}
\begin{proof}
We blow up and blur a finite rainbow polyadic equality algebra, namely, $\R(\Gamma)$
where $\Gamma$ is the complete irreflexive graph $n$, and the greens
are  ${\sf G}=\{\g_i:1\leq i<n-1\}
\cup \{\g_0^{i}: 1\leq i\leq n+1\},$ we denote this finite rainbow algebra 
by ${\sf PEA}_{n+1, n}.$

Let $\At$ be the rainbow atom structure similar to that in \cite{Hodkinson} except that we have $n+1$ greens and
$n$ indices for reds, so that the rainbow signature  $L$  now consists of $\g_i: 1\leq i<n-1$, $\g_0^i: 1\leq i\leq n+1$,
$\w_i: i<n-1$,  $\r_{kl}^t: k<l\in n$, $t\in \omega$,
binary relations, and $\y_S$, $S\subseteq n+1$,
$n-1$ ary relations. We also have a shade of red $\rho$; the latter is a binary relation but is {\it outside the rainbow signature},
though  it is used to label coloured graphs during a certain game devised to prove representability 
of the term algebra \cite{Hodkinson}, and in fact \pe\ can win the $\omega$ rounded game
and build the $n$ homogeneous model $M$ by using $\rho$ whenever
she is forced a red, as will be shown in a while. 

So $\At$ is obtained from the rainbow atom structure of the algebra $\A$ defined in \cite[section 4.2 starting p. 25]{Hodkinson}  
truncating the greens to be finite (exactly $n+1$ greens) and everything else is the same. 
In \cite{Hodkinson} it shown that the complex algebra $\Cm\At\A$ is not representable; 
the result to be obtained now, because the greens are finite but still outfit the red, is sharper; 
it will imply that
$\Rd_{\Sc}\Cm\At\notin S\Nr_n\Sc_{n+3}$.


Now $\Tm\At\in \sf RPEA_n$; this can be proved exactly like  in \cite{Hodkinson}.
Strictly speaking the cylindric reduct of $\Tm\At$ can be proved representable like in 
\cite{Hodkinson}; representating the polyadic operations 
is straightforward, by swapping variables in representing formulas.

Let us spell out more details. 
The first three items are very similar to Hodkinson's arguments. The only difference between our atom structure and his is that we use finitely many 
greens, while he uses infinitely many. 
The greens do not contribute to this part of the proof; all the other colours do.
The reds are the most important in this part of the proof.

\begin{enumarab}
\item {\sf Constructing an $n$ homogeneous model}

We define {\it a new class} of coloured graphs $\GG$; 
they are obtained from  $\sf CGR$ by adding the shade of red $\rho$,
and new forbidden triples of reds involving the shade of red $\rho$, namely,
$(\r_{jk}^i, \r_{j'k'}^{i'}, \rho)$ for any $i,j,k,i',j',k'\in n$ and $(\r_{jk}^i, \rho, \rho)$ for any 
$i,j,k\in n+1$. Strictly speaking the reds here are different from the reds 
specified in the rainbow colours defined in \ref{def}, for they have superscripts coming
from $\omega$.
To deal with such reds we stipulate that  $(\r_{ij}^l, \r_{j'k'}^{l'}, \r_{i^*k^*}^{l''})$ for $i,j,j',k',i^*, k^*\in n$ 
is forbidden  unless $l=l'=l''$ and $i=i^*,\; j=j'\mbox{ and }k'=k^*.$

On the other hand ${\sf PEA}_{n+1, n}$ is the standard rainbow algebra as defined in \ref{def}.

Now one can view the complete undirected coloured graphs in $\GG$ as {\it first order} models for the 
rainbow signature {\it together with $\rho$ viewed as a binary relation}. 
(In case we have infinitely many greens like in \cite{Hodkinson} the rainbow theory \cite{HHbook2} is 
an $L_{\omega_1, \omega}$ theory.)

Using  the standard rainbow argument adopted in \cite{Hodkinson},  one shows that 
there is a countable $n$ homogeneous model  $M\in \GG$ with the following
property:\\
$\bullet$ If $\triangle \subseteq \triangle' \in \GG$, $|\triangle'|
\leq n$, and $\theta : \triangle \rightarrow M$ is an embedding,
then $\theta$ extends to an embedding $\theta' : \triangle'
\rightarrow M$.  

To prove this we use, like Hodkinson, a simple game.
Two players, $\forall$ and $\exists$, play a game to build a
labelled graph $M$. They play by choosing a chain $\Gamma_0
\subseteq \Gamma_1 \subseteq\ldots $ of finite graphs in $\GG$; the
union of
the chain will be the graph $M.$
There are $\omega$ rounds. In each round, $\forall$ and $\exists$ do
the following. Let $ \Gamma \in \GG$ be the graph constructed up to
this point in the game. $\forall$ chooses $\triangle \in \GG$ of
size $< n$, and an embedding $\theta : \triangle \rightarrow
\Gamma$. He then chooses an extension $ \triangle \subseteq
\triangle^+ \in \GG$, where $| \triangle^+ \backslash \triangle |
\leq 1$. These choices, $ (\triangle, \theta, \triangle^+)$,
constitute his move. $\exists$ must respond with an extension $
\Gamma \subseteq \Gamma^+ \in \GG$ such that $\theta $ extends to an
embedding $\theta^+ : \triangle^+ \rightarrow \Gamma^+$. Her
response ends the round.
The starting graph $\Gamma_0 \in \GG$ is arbitrary. We claim that \pe\ can always find a suitable
extension $\Gamma^+ \in \GG$.  Let $\Gamma \in \GG$ be the graph built at some stage, and suppose that 
$\forall$ choose the graphs $ \triangle \subseteq \triangle^+ \in
\GG$ and the embedding $\theta : \triangle \rightarrow \Gamma$.
Thus, his move is $ (\triangle, \theta, \triangle^+)$. We may assume with no loss of generality that $\forall$
actually played $ ( \Gamma \upharpoonright F, Id_F, \triangle^+)$,
where $\Gamma \upharpoonright F \subseteq \triangle^+ \in \GG$,
$\triangle^+ \backslash F = \{\delta\}$, and $\delta \notin \Gamma$.
Then $\forall$ has to build a labelled graph $ \Gamma^*
\supseteq \Gamma$, whose nodes are those of $\Gamma$ together with
$\delta$, and whose edges are the edges of $\Gamma$ together with
edges from $\delta$ to every node of $F$. The labelled graph
structure on $\Gamma^*$ is given by\\
$\bullet$ $\Gamma$ is an induced subgraph of $\Gamma^*$ (i.e., $
\Gamma \subseteq \Gamma^*$)\\
$\bullet$ $\Gamma^* \upharpoonright ( F \cup \{\delta\} ) =
\triangle^+$.
Now $ \exists$ must extend $ \Gamma^*$ to a complete
graph on the same node and complete the colouring yielding a graph
$ \Gamma^+ \in \GG$. Thus, she has to define the colour $
\Gamma^+(\beta, \delta)$ for all nodes $ \beta \in \Gamma \backslash
F$, in such a way as to meet the required conditions.  The strategy of \pe\ is as follows:
\begin{enumroman}

\item If there is no $f\in F$,
such that $\Gamma^*(\beta, f), \Gamma^*(\delta ,f)$ are coloured $\g_0^t$ and $\g_0^u$
for some $t,u$, then \pe\ defined $\Gamma^+(\beta, \delta)$ to  be $\w_0$.

\item Otherwise, if for some $i$ with $0<i<n-1$, there is no $f\in F$
such that $\Gamma^*(\beta,f), \Gamma^*(\delta, f)$ are both coloured $\g_i$, then \pe\
defines the colour $\Gamma^+(\beta,\delta)$ to
to be $\w_i$ say the least such.

\item Otherwise $\delta$ and $\beta$ are both the apexes on $F$ in $\Gamma^*$ that induce
the same linear ordering on (there are no green edges in $F$ because
$\Delta^+\in \GG$, so it has no green triangles).
Now \pe\ has no choice but to pick a red. The colour she chooses is $\rho.$

\end{enumroman}
This defines the colour of edges. Now for hyperedges,
for  each tuple of distinct elements
$\bar{a}=(a_0,\ldots, a_{n-2})\in {}^{n-1}(\Gamma^+)$
such that $\bar{a}\notin {}^{n-1}\Gamma\cup {}^{n-1}\Delta$ and with no edge $(a_i, a)$
coloured green in  $\Gamma^+$, \pe\ colours $\bar{a}$ by $\y_{S}$
where
$S=\{i <\omega: \text { there is a $i$ cone with base  } \bar{a}\}$.
Notice that $|S|\leq F$. This strategy works \cite[lemma 2.7]{Hodkinson}.

Now there are only countably many
finite graphs in $\GG$ up to isomorphism, and each of the graphs
built during the game is finite. Hence $\forall$ can  play
every possible $(\triangle, \theta, \triangle^+)$ (up to
isomorphism) at some round in the game. Suppose he does this, and
let $M$ be the union of the graphs played in the game. Then 
$M$ is as required \cite{Hodkinson}.

\item {\sf Relativization, back and forth systems ensuring that relativized semantics coincide with the classical semantics}

Let $W = \{ \bar{a} \in {}^n M : M \models ( \bigwedge_{i < j < n,
l < n} \neg \rho(x_i, x_j))(\bar{a}) \}$.
Here assignments who have a $\rho$ labelled edge are discarded.

For an $L^n_{\infty \omega}$-formula $\varphi $,  define
$\varphi^W$ to be the set $\{ \bar{a} \in W : M \models_W \varphi
(\bar{a}) \}$.  Then   set $\A$ to be the relativised set algebra with domain
$$\{\varphi^W : \varphi \,\ \textrm {a first-order} \;\ L^n-
\textrm{formula} \}$$  and unit $W$, endowed with the algebraic
operations ${\sf d}_{ij}, {\sf c}_i, $ ect., in the standard way. 

Let $\cal S$ be set algebra with domain  $\wp ({}^{n} M)$ and
unit $ {}^{n} M $. Then the map $h : \A
\longrightarrow \cal S$ given by $h:\varphi ^W \longmapsto \{ \bar{a}\in
{}^{n} M: M \models \varphi (\bar{a})\}$ can be checked to be well -
defined and one-one; this is a representation of $\A$ \cite[Proposition 3.13, Proposition 4.2]{Hodkinson}.
This follows from the fact that  classical semantics and relativized semantic with respect to $W$ 
coincide for first order formulas of the rainbow signature which follows from the  {\it $n$-homogeneity built into
$M$}, that  implies that the set of all partial
isomorphisms of $M$ of cardinality at most $n$ forms an
$n$-back-and-forth system. 

Let us elaborate some more. 
Let $\chi$ be a permutation of the set $\omega \cup \{ \rho\}$. Let
$ \Gamma, \triangle \in \sf \GG$ have the same size, and let $ \theta :
\Gamma \rightarrow \triangle$ be a bijection. We say that $\theta$
is a $\chi$-\textit{isomorphism} from $\Gamma$ to $\triangle$ if for
each distinct $ x, y \in \Gamma$,
\begin{itemize}
\item If $\Gamma ( x, y) = \r_{jk}^i$,
\begin{equation*}
\triangle( \theta(x),\theta(y)) =
\begin{cases}
\r_{jk}^{\chi(i)}, & \hbox{if $\chi(i) \neq \rho$} \\
\rho,  & \hbox{otherwise.} \end{cases}
\end{equation*}
\end{itemize}

\begin{itemize}
\item If $\Gamma ( x, y) = \rho$, then
\begin{equation*}
\triangle( \theta(x),\theta(y)) \in
\begin{cases}
\r_{jk}^{\chi(\rho)}, & \hbox{if $\chi(\rho) \neq \rho$} \\
\rho,  & \hbox{otherwise.} \end{cases}
\end{equation*}
\end{itemize}

For any permutation $\chi$ of $\omega \cup \{\rho\}$, $\Theta^\chi$
is the set of partial one-to-one maps from $M$ to $M$ of size at
most $n$ that are $\chi$-isomorphisms on their domains. We write
$\Theta$ for $\Theta^{Id_{\omega \cup \{\rho\}}}$.

Then like the proof of \cite[Lemma 3.10]{Hodkinson}, 
for any permutation $\chi$ of $\omega \cup \{\rho\}$, $\Theta^\chi$
is an $n$-back-and-forth system on $M$.

Using this we now  derive a connection between classical and
relativized semantics in $M$, over the set $W$:\\
Recall that $W$ is simply the set of tuples $\bar{a}$ in ${}^nM$ such that the
edges between the elements of $\bar{a}$ don't have a label involving
the red shade $\rho.$ Their labels come only from the rainbow signature. 
We can replace $\rho$-labels by suitable {\it red rainbow}
labels within an $n$-back-and-forth system. Thus, it can be arranged that the
system maps a tuple $\bar{b} \in {}^n M \backslash W$ to a tuple
$\bar{c} \in W$ and this will preserve any formula
containing no red rainbow relation symbols 
moved by the system.

Indeed, we can show that the
classical and $W$-relativized semantics agree.
$M \models_W \varphi(\bar{a})$ iff $M \models \varphi(\bar{a})$, for
all $\bar{a} \in W$ and all $L^n$-formulas $\varphi$, hence as claimed $\A$ defined above is representable as a set algebra.

The proof is by induction on $\varphi$ \cite[Proposition 3.13]{Hodkinson}. 
If $\varphi$ is atomic, the
result is clear; and the Boolean cases are simple.
Let $i < n$ and consider $\exists x_i \varphi$. If $M \models_W
\exists x_i \varphi(\bar{a})$, then there is $\bar{b} \in W$ with
$\bar{b} =_i \bar{a}$ and $M \models_W \varphi(\bar{b})$.
Inductively, $M \models \varphi(\bar{b})$, so clearly, $M \models_W
\exists x_i \varphi(\bar{a})$.
For the (more interesting) converse, suppose that $M \models_W
\exists x_i \varphi(\bar{a})$. Then there is $ \bar{b} \in {}^n M$
with $\bar{b} =_i \bar{a}$ and $M \models \varphi(\bar{b})$. Take
$L_{\varphi, \bar{b}}$ to be any finite subsignature of $L$
containing all the symbols from $L$ that occur in $\varphi$ or as a
label in $M \upharpoonright \rng(\bar{b})$. Choose a permutation $\chi$ of
$\omega \cup \{\rho\}$ fixing any $i'$ such that some $\r_{jk}^{i'}$
occurs in $L_{\varphi, \bar{b}}$ and moving $\rho$.
Let $\theta = Id_{\{a_m : m \neq i\}}$. Take any distinct $l, m \in
n \setminus \{i\}$. If $M(a_l, a_m) = \r_{jk}^{i'}$, then $M( b_l,
b_m) = \r_{jk}^{i'}$ because $ \bar{a} = _i \bar{b}$, so $\r_{jk}^{i'}
\in L_{\varphi, \bar{b}}$ by definition of $L_{\varphi, \bar{b}}$.
So, $\chi(i') = i'$ by definition of $\chi$. Also, $M(a_l, a_m) \neq
\rho$ because $\bar{a} \in W$. It now follows that
$\theta$ is a $\chi$-isomorphism on its domain, so that $ \theta \in
\Theta^\chi$.
Extend $\theta $ to $\theta' \in \Theta^\chi$ defined on $b_i$,
using the ``forth" property of $ \Theta^\chi$. Let $
\bar{c} = \theta'(\bar{b})$. Now by choice of of $\chi$, no labels
on edges of the subgraph of $M$ with domain $\rng(\bar{c})$ involve
$\rho$. Hence, $\bar{c} \in W$.
Moreover, each map in $ \Theta^\chi$ is evidently a partial
isomorphism of the reduct of $M$ to the signature $L_{\varphi,
\bar{b}}$. Now $\varphi$ is an $L_{\varphi, \bar{b}}$-formula.
We have $M \models \varphi(\bar{a})$ iff $M \models \varphi(\bar{c})$.
So $M \models \varphi(\bar{c})$. Inductively, $M \models_W
\varphi(\bar{c})$. Since $ \bar{c} =_i \bar{a}$, we have $M
\models_W \exists x_i \varphi(\bar{a})$ by definition of the
relativized semantics. This completes the induction, and proves that 
$h:\A\to \cal S$ above is indeed a representation of 
$\A$. However, it is {\it not} a complete representation. This will be clear from our subsequent discussion,
when we explicity describe the atoms of $\A$ and show that their union is not $^nM$; it is $W$.
In fact $\A$ has no complete classical representation; even more its $\Df$ reduct does not have
such a representation as stated in the 
theorem to be proved in a while.
\item {\sf The atoms and the complex algebra; the \d\ completion}

Recall that $L$ denotes the rainbow signature (without $\rho$).
The logics $L^n$ and $L^n_{\infty \omega}$ are taken in this
signature.

We show that $\A$ is atomic and   
we give the atoms, following Hodkinson,  a {\it syntactical description} in terms of special formulas
in the rainbow signature specified above taken in $L^n$; this will enable us to transparently define the 
embedding of $\sf PEA_{n+1, n}$ into the hitherto constructed 
complex algebra. 

Now every atom in the (representable) relativized set algebra $\A$ is {\it uniquely defined by an $\sf MCA$ formula} \cite{Hodkinson}.
 A formula $ \alpha$  of  $L^n$ is said to be $\sf MCA$
('maximal conjunction of atomic formulas') \cite[Definition 4.3]{Hodkinson}  if (i) $M \models \exists
x_0\ldots, x_{n-1} \alpha $ and (ii) $\alpha$ is of the form
$$\bigwedge_{i \neq j < n} \alpha_{ij}(x_i, x_j) \land \bigwedge\eta_{\mu}(x_0,\ldots, x_{n-1}),$$
where for each $i,j,\alpha_{ij}$ is either $x_i=x_i$ or $R(x_i,x_j)$ a binary relation symbol in the rainbow signature, 
and for each $\mu:(n-1)\to n$, $\eta_{\mu}$ is either $y_S(x_{\mu(0)},\ldots x_{\mu(n-2)})$ for some $y_S$ in the signature, 
if for all distinct $i,j<n$, $\alpha_{\mu(i), \mu(j)}$ is not equality nor green, otherwise it is
$x_0=x_0$.

A formula $\alpha$  being $\sf MCA$ says that the set it defines in ${}^n M$
is nonempty, and that if $M \models \alpha (\bar{a})$ then the graph
$M \upharpoonright \rng (\bar{a})$ is determined up to isomorphism
and has no edge whose label is of the form $\rho$. 
Now since we have for any permutation $\chi$ of $\omega \cup \{\rho\}$, $\Theta^\chi$
is an $n$-back-and-forth system on $M$, any two
tuples (graphs) satisfying $\alpha$ are isomorphic and one is mapped to the
other by the $n$-back-and-forth system $\Theta$ of partial isomorphisms from $M$ to $M$; they are the {\it same} coloured graph.

No $L^n_{\infty \omega}$- formula can distinguish any two graphs satisfying an $\sf MCA$ formula $\alpha$.
So $\alpha$
defines an atom of $\A$. Since the
$\sf MCA$ - formulas clearly cover $W$, the atoms defined by them are
dense in $\A$, hence $\A$ is atomic.  The coloured  graphs whose edges are not labelled by the shade of red 
$\rho$ (up to isomorphism) determined by $\sf MCA$ 
formulas are the atoms of $\A$.

In more detail, let $\varphi$ be any $L^n_{\infty\omega}$-formula, and $\alpha$ any
$\sf MCA$-formula. If $\varphi^W \cap \alpha^W \neq \emptyset $, then
$\alpha^W \subseteq \varphi^W $.
Indeed, take $\bar{a} \in  \varphi^W \cap \alpha^W$. Let $\bar{b} \in
\alpha^W$ be arbitrary. Clearly, the map $( \bar{a} \mapsto
\bar{b})$ is in $\Theta$. Also, $W$ is
$L^n_{\infty\omega}$-definable in $M$, since we have
$$ W = \{
\bar{a} \in {}^n M : M \models (\bigwedge_{i < j< n} (x_i = x_j \vee
\bigvee_{R \in L} R(x_i, x_j)))(\bar{a})\}.$$
We have $M \models_W \varphi (\bar{a})$
 iff $M \models_W \varphi (\bar{b})$. Since $M \models_W \varphi (\bar{a})$, we have
$M \models_W \varphi (\bar{b})$. Since $\bar{b} $ was arbitrary, we
see that $\alpha^W \subseteq \varphi^W$.
Let $$F = \{ \alpha^W : \alpha \,\ \textrm{an $\sf MCA$},
L^n-\textrm{formula}\} \subseteq \A.$$
Evidently, $W = \bigcup F$. We claim that
$\A$ is an atomic algebra, with $F$ as its set of atoms.
First, we show that any non-empty element $\varphi^W$ of $\A$ contains an
element of $F$. Take $\bar{a} \in W$ with $M \models_W \varphi
(\bar{a})$. Since $\bar{a} \in W$, there is an $\sf MCA$-formula $\alpha$
such that $M \models_W \alpha(\bar{a})$. Then $\alpha^W
\subseteq \varphi^W $. By definition, if $\alpha$ is an $\sf MCA$ formula
then $ \alpha^W$ is non-empty. If $ \varphi$ is
an $L^n$-formula and $\emptyset \neq \varphi^W \subseteq \alpha^W $,
then $\varphi^W = \alpha^W$. It follows that each $\alpha^W$ (for
$\sf MCA$ $\alpha$) is an atom of $\A$.
We can also conclude from this that $\A$ is simple, for if $\phi^W\in \A$ is non zero then there is an $\sf MCA$ formula 
$\alpha$ such that $\alpha^W\leq \phi^W$,
and $M\models \exists x_0\ldots x_{n-1}\alpha$, hence ${\sf c}_{(n)}\phi^{W}= {}^{n}M$.
Also $\A$ is {\it not} completely representable; this will follow from the fact that its 
completion is not representable. 

Let  $\C$  be the the relativized set algebra with domain
$$\{\varphi^W : \varphi \,\ \textrm {an} \;\ L_{\infty, \omega}^n-
\textrm{formula} \},$$  
and operations defined like for $\A$.
Then $\C$ is the \d\ completion of $\A$, that is $\C\cong\Cm\At\A$ 
\cite[Proposition 4.6]{Hodkinson}.

\item {\sf Embedding $\PEA_{n+1, n}$ in the complex algebra}

Hodkinson proves the non representability of the completion of his algebra $\A$ \cite{Hodkinson} 
using the greens, which are in his case infinite.
In case of the existence of a representation, he reaches a contradiction, by using infinitely many cones having a common base,  
to force an inconsistent triple of reds. So the greens play an essential role in this part of 
the proof. Here we show that if we truncate the greens to be finite, but as long as they outfit the red, 
then we still can force such an inconsistent triple using only {\it finitely} many 
cones, geting a sharper result. Roughly this finite number determines the number of dimensions $>n,$ which $\Cm\At\A$ cannot neatly embed 
into an algebra having this number as its dimension.

We have $\At=\At\A=\At\C$, and 
$\Tm\At\subseteq \A\subseteq \Cm\At\cong\C$; the former two are 
representable. We can proceed like \cite{Hodkinson} and show that though $\Rd_{ca}\C$ is representable in a relativized sense it is not
classically representable. But we can do better.

We show that the $\Sc$ reduct of the last is not in $S\Nr_n\Sc_{n+3}$.
This will be done by showing that the $\Sc$ reduct of ${\sf PEA}_{n+1, n}$ 
is not in $S\Nr_n\Sc_{n+3}$ and that ${\sf PEA}_{n+1, n}$ embeds into 
$\C$ as polyadic equality algebras.
Recall that the colours used for coloured graphs 
involved in building  the finite atom structure of the algebra ${\sf PEA}_{n+1, n}$ are:
\begin{itemize}

\item greens: $\g_i$ ($1\leq i\leq n-2)$, $\g_0^i$, $1\leq i\leq n+1,$

\item whites : $\w_i,  i\leq n-2,$

\item reds:  $\r_{ij},$ $i<j\in n,$

\item shades of yellow : $\y_S,  S\subseteq n+1$.

\end{itemize}
with {\it forbidden triples}
\vspace{-.2in}
\begin{eqnarray*}
&&\nonumber\\
(\g, \g^{'}, \g^{*}), (\g_i, \g_{i}, \w_i)
&&\mbox{any }1\leq i\leq  n-2  \\
(\g^j_0, \g^k_0, \w_0) &&\mbox{ any } 1\leq j, k\leq n+1\\
\label{forb:match}(\r_{ij}, \r_{j'k'}, \r_{i^*k^*}) &&i,j,i', k', i^*, j^*\in n,\\ \mbox{unless }i=i^*,\; j=j'\mbox{ and }k'=k^*.
\end{eqnarray*}
and no other triple is forbidden. 
One can say metaphorically that $\At$ is isomorphic to the rainbow atom structure obtained from the finite atom structure of ${\sf PEA}_{n+1, n}$ 
if each red $\r_{ij}$, $i<j<n$  is `split' into $\omega$ many copies $\r_{ij}^l$: $l\in \omega$,
and adding the consistency condition stated above, namely, 
$(\r_{ij}^l, \r_{j'k'}^{l'}, \r_{i^*k^*}^{l''})$ for $i,j,j',k',i^*, k^*\in n$ 
is forbidden  unless $l=l'=l''$ and $i=i^*,\; j=j'\mbox{ and }k'=k^*.$
But splitting has to do with {\it splitting atoms}, in our context {\it coloured graphs}, not colours; 
so this `image', which, all the same,  we find useful to highlight at this stage, will be made
precise in a while.

A coloured graph  is red
if at least one of its edges is labelled red. Here {\it we do not} have $\rho$, it simply does not exist in the set of available reds. 
Now the algebra ${\sf PEA}_{n+1, n}$ embeds into $\C$ as polyadic equality algebras.
The map is defined on the atoms then extended the obvious way to the whole algebra. 
Every $[a]: n\to \Gamma$, where $\Gamma$ a red graph is mapped to {\it the join of its copies}, which exists 
because $\C$ is complete.
A copy of a red graph is one that is isomorphic to this graph, modulo removing superscripts of reds.
Every other atom (graph) is mapped to itself.

More precisely, for brevity write $\r$ for $\r_{jk}$($j<k<n$).
If $\Gamma$ is a coloured graph using the colours in $\At \PEA_{n+1, n}$, and $a:n\to \Gamma$ is in $\At\PEA_{n+1,n}$,
then $a':n\to \Gamma'$ with $\Gamma'\in \sf CGR$
is a {\it copy} of $a:n\to \Gamma$, if for any non red binary colour $\sf c$ and any red $\r$ we have for any $i,j, k_0,\ldots k_{n-1}<n$: 
\begin{itemize}
\item $a(i)=a(j)\Longleftrightarrow a'(i)=a'(j),$
\item $(a(i), a(j))\in {\sf c} \Longleftrightarrow (a'(i), a'(j))\in {\sf c},$ 
\item $(a(i), a(j))\in \r \Longleftrightarrow (a'(i), a'(j))\in \r^l$ for some $l\in \omega,$
\item $M_a(a(k_0),\dots, a(k_{n-2}))=M_a'(a'(k_0),\ldots, a'(k_{n-2}))$ whenever defined.
\end{itemize}
In other words, all non red edges and $n-1$ tuples have the same colour (whenever defined) 
and  for all $i<j<n$, for every red $\r$, if  $(a(i), a(j))\in \r$,
then there exits $l\in \omega$ such that $(a'(i), a'(j))\in \r^l$. Here we implicitly require that for distinct $i,j,k<n$, if 
$(a(i),a(j))\in \r$, $(a(j), a(k))\in \r'$, $(a(i), a(k))\in \r''$, and 
$(a'(i), a'(j))\in \r^l_1$, $(a'(j), a'(k))\in [\r']^{l_2}$ and $(a'(i), a'(k))\in [\r'']^{l_3}$, then $l_1=l_2=l_3=l$, say,
so that $(\r^l, [\r']^l, [\r'']^l)$ 
is a consistent triangle in $\Gamma'$.

Then 
every $a:n\to \Gamma$, where $\Gamma$ is red in ${\sf PEA}_{n+1, n}$ is mapped to the join of $\phi^W$, 
where $\phi$ is an $\sf MCA$ formula, corresponding 
to $a':n\to \Gamma'$ in $\C$, such that $a'$ is a 
copy of $\Gamma$. 

These joins exist because $\C$ is complete.

Hence we have a map $\Psi: \At(\sf PEA_{n+1, n})\to \C$.  
This induces another map,  which we denote also by $\Psi$ from the finite algebra ${\sf PEA}_{n+1, n}$ 
to $\mathfrak{C}$ taking finite (unions) sums of atoms to the 
corresponding sum in $\C$.


We now show that $\Psi$ is an injective homomorphism. It is injective since distinct atoms are mapped to distinct elements.
We proceed to show that it preserves the operations; in the process we clarify 
the fact that we can look at $\At$ 
as begotten from ${\sf PEA}_{n+1, n}$ by {\it splitting every red graph into $\omega$ many copies}, in the following
sense.

If $a':n\to \Gamma'$ and $\Gamma'$ is a red graph 
using the colours of the rainbow signature of $\At$, whose reds are $\{\r_{kj}^l: k<j<n, l\in \omega\},$
then there is a unique $a: n\to \Gamma$, $\Gamma$ 
a red graph using the red colours in the rainbow signature of $\sf PEA_{n+1,n}$, namely, $\{\r_{kj}: k<j< n\}$
such that $a'$ is a copy of $a$.
We denote $a$ by $o(a')$, $o$ short for {\it original}; $a$ is the original of its copy $a'$.

For $i<n$, let $T_i$ be the accessibility relation corresponding to the $i$th cylindrifier in $\At$. 
Let  
$T_i^{s}$, be that corresponding to the $i$th cylindrifier in ${\sf PEA}_{n+1, n}$. 
Then if $c:n\to \Gamma$ and $d: n\to \Gamma'$ are surjective maps 
$\Gamma, \Gamma'$ are coloured graphs for ${\sf PEA}_{n+1, n}$, that are not red, then for any $i<n$, we have 
$$([c],[d])\in T_i\Longleftrightarrow ([c],[d])\in T_i^s.$$

If $\Gamma$ is red using the colours for the rainbow signature of $\At$ (without $\rho$) 
and $a':n\to \Gamma$,then for any $b:n\to \Gamma'$ where $\Gamma'$ is not red and any $i<n$, we have   
$$([a'], [b])\in T_i\Longleftrightarrow  ([o(a')], [b])\in T_i^{s}.$$
Extending the notation, for $a:n\to \Gamma$ a graph that is not red in $\At$, set $o(a)=a$.
Then for any $a:n\to \Gamma$, $b:n\to \Gamma'$, where $\Gamma, \Gamma'$ are 
coloured graphs at least  one of which is not red in $\At$ and any $i<n$, we have
$$[a]T_i[b]\Longleftrightarrow [o(a)]T_i^s[o(b)].$$

Now we deal with the last case, when the two graphs involved are red. 
Now assume that $a':n\to \Gamma$ is as above, that is $\Gamma\in {\sf CGR}$ is red, 
$b:n\to \Gamma'$ and $\Gamma'$ is red too, using the colours in the rainbow signature of $\At$.

Say that two maps $a:n\to \Gamma$, $b:n\to \Gamma'$, with $\Gamma$ and $\Gamma'\in \sf CGR$ having the same size 
are  $\r$ related if all non red edges and $n-1$ tuples have the same colours (whenever defined), and 
for all every red $\r$, whenever  $i<j<n$, $l\in \omega$,  and $(a(i), a(j))\in \r^l$, then there exists
$k\in \omega$ such that $(b(i), b(j))\in \r^k$. In more detail, for any non red binary colour $\sf c$ and any red $\r$ we have  for any $i,j, k_0,\ldots k_{n-1}<n$: 
\begin{itemize}
\item $a(i)=a(j)\Longleftrightarrow a'(i)=a'(j),$
\item $(a(i), a(j))\in {\sf c} \Longleftrightarrow (a'(i), a'(j))\in {\sf c},$ 
\item $(a(i), a(j))\in \r^k \Longleftrightarrow (a'(i), a'(j))\in \r^l$ for some $l,k\in \omega,$
\item $M_a(a(k_0),\dots, a(k_{n-2}))=M_{a'}(a'(k_0),\ldots, a'(k_{n-2}))$ whenever defined.
\end{itemize}

Let $i<n$. Assume that $([o(a')], [o(b)])\in T_i^s$. Then there exists $c:n\to \Gamma$ that is $\r$ related to $a'$
such that $[c]T_i[b]$.  Conversely, if $[c]T_i[b]$, then $[o(c)]T_i[o(b)].$

Hence, by complete additivity of cylindrifiers,  the map $\Theta: \At({\sf PEA}_{n+1, n})\to \Cm\At$ defined via
\[
 \Theta(\{[a]\})=
  \begin{cases}
    \{ [a']: \text { $a'$   copy of $a$}\}  \text { if $a$ is red, } \\
        \{[a]\} \text { otherwise. } \\
   
  \end{cases}
\]
\\
induces an embedding from ${\sf PEA}_{n+1, n}$ to $\Cm\At$, which we denote also by 
$\Theta$.

We first check preservation of diagonal elements. 
If $a'$ is a copy of $a$, $i, j<n$, and  $a(i)=a(j)$, then $a'(i)=a'(j)$.

We next  check cylindrifiers. We show that for all $i<n$ and $[a]\in \At(\PEA_{n+1, n})$ we have: 
$$\Theta({\sf c}_i[a])=\bigcup \{\Theta([b]):[b]\in \At\PEA_{n+1,n}, [b]\leq {\sf c}_i[a]\}= {\sf c}_i\Theta ([a]).$$
Let $i<n$. If $[b]\in \At\PEA_{n+1}$,  $[b]\leq {\sf c}_i[a]$, and $b':n\to \Gamma$, $\Gamma\in \sf CGR$, is a copy of $b$, 
then there exists  $a':n\to \Gamma'$, $\Gamma'\in \sf CGR$, a copy of  $a$ such that
$b'\upharpoonright  n\setminus \{i\}=a'\upharpoonright n\setminus \{i\}$. Thus $\Theta([b])\leq {\sf c}_i\Theta([a])$.

Conversely, if $d:n\to \Gamma$, $\Gamma\in \sf CGR$ and $[d]\in {\sf c}_i\Theta([a])$, then there exist $a'$ a copy of $a$ such that 
$d\upharpoonright n\setminus \{i\}=a'\upharpoonright n\setminus \{i\}$. 
Hence $o(d)\upharpoonright n\setminus  \{i\}=a\upharpoonright n\setminus \{i\}$, and so $[d]\in \Theta({\sf c}_i[a]),$
and we are done.

Now we also have
$\Cm\At\cong \C$, via the map 
$g:X\mapsto \bigcup X$. The proof is similar to Hodkinson's \cite[Hodkinson] for his \d\ completion which is different than 
ours.
The map is clearly injective. It is surjective, since
$$\phi^W=\bigcup\{\alpha^W: \alpha \text { an $\sf MCA$-formula}, \alpha^W\subseteq \phi^W\}$$
for any $L_{\infty\omega}^n$ formula $\phi$.
Preservation of the Boolean operations and diagonals is clear.
We check cylindrifications. We require that for any $X\subseteq \At\A$,
we have
$\bigcup {\sf c}_i^{\C}X={\sf c}_i^{\cal D}(\bigcup X)$ that is
$$\bigcup \{S\in \At\A: S\subseteq {\sf c}_i^{\A}S'\text { for some $S'\in X$ }\}=$$
$$\{\bar{a}\in W: \bar{a}\equiv_i \bar{a'} \text { for some } \bar{a'}\in \bigcup X\}.$$
Let $\bar{a}\in S\subseteq {\sf c}_iS'$, where $S'\in X$. So there is $\bar{a'}\equiv_i \bar{a}$ with $\bar{a'}\in S'$, and so $\bar{a'}\in \bigcup X$.

Conversely, let $\bar{a}\in W$ with $\bar{a}\equiv_i \bar{a'}$ for some $\bar{a'}\in \bigcup X$.
Let $S\in \At\A$, $S'\in X$ with $\bar{a}\in S$ and $\bar{a'}\in S'$.
Choose $\sf MCA$ formulas $\alpha$ and $\alpha'$ with $S=\alpha^W$ and $S'=\alpha'^{W}$,
then $\bar{a}\in \alpha^{W}\cap (\exists x_i\alpha')^W$ so $\alpha^W\subseteq (\exists x_i\alpha')^W$, or $S\subseteq c_i^{\A}(S')$.
The required now follows. We leave the checking of substitutions to the reader.
Thus $$g\circ \Theta=\Psi,$$
and so $\Psi: \sf PEA_{n+1,n}\to \C$ is an embedding. 

A subset of $\C$ is red if it consists only of red graphs. 
Notice that for every red atom $a$, we have $|\Theta(a)|\geq \omega$, and it is {\it not} co-finite, that is its complement is also 
infinite.  These sets do not exist in the term algebra, which contains only finite or co-finite 
red sets. 
Here $\rho$ functions in a sense as a non standard red colour, corresponding to the non-principal ultrafilter of 
$\A$ generated by all co-finite sets of red atoms.

\item {\sf \pa\ winning  $F^{n+3}$ on $\At(\PEA_{n+1,n})$}

But now we can show that \pa\ can win the game $F^{n+3}$ on $\At({\sf PEA}_{n+1,n})$ in only $n+2$ rounds 
as follows. 
Viewed as an \ef\ forth game  pebble game, with finitely many rounds and pairs of pebbles, 
played on the two complete irreflexive graphs $n+1$ and $n$, in each round $0,1\ldots n$, \pa\ places a  new pebble  on  an element of $n+1$.
The edge relation in $n$ is irreflexive so to avoid losing
\pe\ must respond by placing the other  pebble of the pair on an unused element of $n$.
After $n$ rounds there will be no such element,
and she loses in the next round.
Hence \pa\ can win the graph game on $\At({\sf PEA}_{n+1,n})$ in $n+2$ rounds using  $n+3$ nodes.

In the game $F^{n+3}$ \pa\ forces a win on a red clique using his excess of greens by bombarding \pe\
with $\alpha$ cones having the same base ($1\leq \alpha\leq n+2)$.

In his zeroth move, \pa\ plays a graph $\Gamma$ with
nodes $0, 1,\ldots, n-1$ and such that $\Gamma(i, j) = \w_0 (i < j <
n-1), \Gamma(i, n-1) = \g_i ( i = 1,\ldots, n-2), \Gamma(0, n-1) =
\g^0_0$, and $ \Gamma(0, 1,\ldots, n-2) = \y_{n+2}$. This is a $0$-cone
with base $\{0,\ldots , n-2\}$. In the following moves, \pa\
repeatedly chooses the face $(0, 1,\ldots, n-2)$ and demands a node
$\alpha$ with $\Phi(i,\alpha) = \g_i$, $(i=1,\ldots n-2)$ and $\Phi(0, \alpha) = \g^\alpha_0$,
in the graph notation -- i.e., an $\alpha$-cone, without loss $n-1<\alpha\leq  n+1$,  on the same base.
\pe\ among other things, has to colour all the edges
connecting new nodes $\alpha, \beta$ created by \pa\ as apexes of cones based on the face $(0,1,\ldots, n-2)$, that is $\alpha,
\beta\geq n-2$. 
By the rules of the game
the only permissible colours would be red. Using this, \pa\ can force a
win in $n+2$ rounds, using $n+3$ nodes  without needing to re-use them, 
thus forcing \pe\ to deliver an inconsistent triple 
of reds.

Let $\B={\sf PEA}_{n+1, n}$. 
Then $\Rd_{sc}\B$   is
outside $S\Nr_n\Sc_{n+3}$ for if it was in $S\Nr_n\Sc_{n+3}$, then being finite it would be in $S_c\Nr_n\Sc_{n+3}$
because $\Rd_{sc}\B$ is the same as its canonical extension $\D$, say, and $\D\in S_c\Nr_n\Sc_{n+3}$. 
But then by theorem \ref{thm:n}, \pe\ would have won. 

Hence $\Rd_{sc}\mathfrak{Cm}\At\notin S\Nr_n\Sc_{n+3}$,  
because $\Rd_{sc}\B$ is embeddable in it
and $S\Nr_n\Sc_{n+3}$ is a variety; in particular, it is closed
under forming subalgebras. 
It now readily follows that $\Rd_{sc}\Cm\At\notin S\Nr_n{\sf Sc}_{n+3}$.

Finally $\Rd_{df}\A$ is not completely representable, because if it were then $\A$,
generated by elements whose dimension sets $<n$,  as a $\PEA_n$ would be completely representable by theorem \ref{AU}
and this induces a representation of its \d\ completion $\Cm\At\A$.

\end{enumarab}

\end{proof}

\section{Complete representability}

Now we approach the notion of complete representations for any $\K$ between $\Sc$ and $\PEA$.
Rainbows will offer solace 
here as well. Throughout this subsection $n$ will be finite and $>1$.
We identify notationally set algebras with their universes.

Our next couple of theorems are formulated and proved for $\CA$s but the proof works for any $\K$ as specified above.

Let $\A\in \sf CA_{n}$ and $f:\A\to \wp(V)$ be a representation of $\A$, where $V$ is 
a generalized space of dimension $n$.
If $s\in {}V$ we let
$$f^{-1}(s)=\{a\in \A: s\in f(a)\}.$$
An {\it atomic representation} $f:\A\to \wp(V)$ is a representation such that for each 
$s\in V$, the ultrafilter $f^{-1}(s)$ is principal.

\begin{theorem}\label{complete} Let $\A\in \sf CA_n$. 
Let $f:\A\to \wp(V)$ be a representation of $\A$. Then $f$ is 
a  complete representation iff $f$ is an atomic one. 
Furthermore, if $\A$ is completely representable, then $\A$
is atomic and $\A\in S_c\Nr_n\sf CA_{\omega}$.
\end{theorem}
\begin{proof} Witnesss \cite[Theorems 5.3.4, 5.3.6]{Sayedneat}, \cite[Theorem 3.1.1]{HHbook2} for the first three parts.
It remains to  show that if $\A$ is completely representable, then $\A\in S_c\Nr_n\CA_{\omega}$.
Assume that $M$ is the base of
a complete representation of $\A$, whose
unit is a generalized space,
that is, $1^M=\bigcup_{i\in I} {}^nU_i$, where $U_i\cap U_j=\emptyset$ for distinct $i$ and $j$ in 
$I$ where $I$ is an 
index set $I$. Let $t:\A\to \wp(1^M)$ be the complete representation.
 For $i\in I$, let $E_i={}^nU_i$, pick $f_i\in {}^{\omega}U_i$, let $W_i=\{f\in  {}^{\omega}U_i: |\{k\in \omega: f(k)\neq f_i(k)\}|<\omega\}$,
and let ${\C}_i$ be the $\CA_n$ with universe $\wp(W_i)$, with the $\CA$ operations defined the usual way
on weak set algebras.  
Then $\C_i$ is atomic; indeed the atoms are the singletons.

Let $x\in \Nr_n\C_i$, that is ${\sf c}_jx=x$ for all $n\leq j<\omega$.
Now if  $f\in x$ and $g\in W_i$ satisfy $g(k)=f(k) $ for all $k<n$, then $g\in x$.
Hence $\Nr_n \C_i$
is atomic;  its atoms are $\{g\in W_i:  \{g(0),\ldots g(n-1)\}\subseteq U_i\}.$
Define $h_i: \A\to \Nr_n\C_i$ by
$$h_i(a)=\{f\in W_i: \exists a\in \At\A: (f(0)\ldots f(n-1))\in t(a)\}.$$

Let $\C=\prod _i \C_i$. Let $\pi_i:\C\to \C_i$ be the $i$th projection map.
Now clearly  $\C$ is atomic, because it is a product of atomic algebras,
and its atoms are $(\pi_i(\beta): \beta\in \At(\C_i)\}$.
Now  $\A$ embeds into $\Nr_n\C$ via $I:a\mapsto (\pi_i(a) :i\in I)$. 
and we may assume that the map is surjective. 

If $a\in \Nr_n\C$,
then for each $i$, we have $\pi_i(x)\in \Nr_n\C_i$, and if $x$ is non
zero, then $\pi_i(x)\neq 0$. By atomicity of $\C_i$, there is a tuple $\bar{m}$ such that
$\{g\in W_i: g(k)=[\bar{m}]_k\}\subseteq \pi_i(x)$. Hence there is an atom
$a$ of $\A$, such that $\bar{m}\in t(a)$,  so $x\cdot  I(a)\neq 0$, and so the embedding is complete
and we are done. Note that in this argument {\it no cardinality condition} is required.
(The reverse inclusion does not hold in general for uncountable algebras, as will be shown in theorem \ref{complete},
though it holds for atomic algebras with countably many atoms as shown in our next theorem).
Hence $\A\in S_c\Nr_n\CA_{\omega}$.

\end{proof}
Conversely, we have \cite[Theorem 5.3.6]{Sayedneat}:
\begin{theorem}\label{completerepresentation} If $\A$ is countable and atomic 
and  $\A\in S_c\Nr_n\CA_{\omega},$ then $\A$ is completely representable.
\end{theorem}

Now we use a rainbow construction. Coloured graphs and rainbow algebras are defined like
above. The algebra constructed now is very similar to ${\sf PEA}_{\omega,\omega}$ but is not identical; for in coloured graphs
we add a new triple of forbidden colours involving two greens and one red synchronized by an order preserving 
function. In particular, we consider the underlying set of $\omega$ endowed 
with two orders, the usual order and its converse.

\begin{theorem}\label{rainbow}
Let $3\leq n<\omega$. Then there exists an atomic $\C\in \sf PEA_n$ with countably many atoms
such that $\Rd_{Sc}\C\notin S_c\Nr_n\Sc_{n+3}$, and there exists
a countable $\B\in S_c\Nr_n\sf QEA_{\omega}$ such that $\C\equiv \B$ (hence $\B$ is also atomic). In particular,
for any class $\sf L$, such that $S_c\Nr_n\K_{\omega}\subseteq \sf L\subseteq S_c\Nr_n\K_{n+3}$, $\sf L$ is not elementary, and
the class of completely representable $\K$ algebras of dimension $n$
is not elementary.
\end{theorem}
\begin{proof}  Let $\N^{-1}$ denote $\N$ with reverse order, let $f:\N\to \N^{-1}$ be the identity map, and denote $f(a)$ by
$-a$, so that for $n,m \in \N$, we have $n<m$ iff $-m<-n$.
We assume that $0$  belongs to $\N$ and we denote the domain of $\N^{-1}$ (which is $\N$)  by $\N^{-1}$.
We alter slightly the standard rainbow construction. The colours we use are the same colours used in rainbow constructions:
\begin{itemize}

\item greens: $\g_i$ ($1\leq i\leq n-2)$, $\g_0^i$, $i\in \N^{-1}$,

\item whites : $\w_i: i\leq n-2,$

\item reds:  $\r_{ij}$ $(i,j\in \N)$,

\item shades of yellow : $\y_S: S\subseteq_{\omega} \N^{-1}$ or $S=\N^{-1}$.

\end{itemize}

And the class $\G$ consists of all coloured graphs $M$ such that
\begin{enumarab}

\item $M$ is a complete graph.

\item $M$ contains no triangles (called forbidden triples)
of the following types:

\vspace{-.2in}
\begin{eqnarray}
&&\nonumber\\
(\g, \g^{'}, \g^{*}), (\g_i, \g_{i}, \w_i)
&&\mbox{any }1\leq i\leq  n-2  \\
(\g^j_0, \g^k_0, \w_0)&&\mbox{ any } j, k\in \N\\
\label{forb:pim}(\g^i_0, \g^j_0, \r_{kl})&&\mbox{unless } \set{(i, k), (j, l)}\mbox{ is an order-}\\
&&\mbox{ preserving partial function }\N^{-1}\to\N\nonumber\\
\label{forb:match}(\r_{ij}, \r_{j'k'}, \r_{i^*k^*})&&\mbox{unless }i=i^*,\; j=j'\mbox{ and }k'=k^*.
\end{eqnarray}
and no other triple of atoms is forbidden.

\item The last two items concerning shades of yellow are as before.

\end{enumarab}

But the forbidden triple $(\g^i_0, \g^j_0, \r_{kl})$
is not present in standard rainbow constructions, adopted example in \cite{HH} and in a more general form in
\cite{HHbook2}. Therefore, we canot use the usual rainbow argument adopted in \cite{HH}; we have to be selective for the choice of the indices
of reds if we are labelling the appexes of two cones having green tints; {\it not any red} will do.
Inspite of such a restriction (that makes it harder for \pe\ to win),  we will show that \pe\ will always succeed to choose a 
suitable red in the finite rounded atomic games. On the other hand this bonus for \pa\ will enable him to win the game $F^{n+3}$ by forcing 
\pe\ to play a decreasing sequence in $\N$. Using and re-using $n+3$ nodes will suffice for this 
purpose.

One then can define (what we continue to call) a rainbow atom structure
of dimension $n$ from the class $\G$.   
Let $$\At=\{a:n \to M, M\in \G: \text { $a$ is surjective }\}.$$
We write $M_a$ for the element of $\At$ for which
$a:n\to M$ is a surjection.
Let $a, b\in \At$ define the
following equivalence relation: $a \sim b$ if and only if
\begin{itemize}
\item $a(i)=a(j)\Longleftrightarrow b(i)=b(j),$

\item $M_a(a(i), a(j))=M_b(b(i), b(j))$ whenever defined,

\item $M_a(a(k_0),\dots, a(k_{n-2}))=M_b(b(k_0),\ldots, b(k_{n-2}))$ whenever
defined.
\end{itemize}
Let $\At$ be the set of equivalences classes. Then define
$$[a]\in E_{ij} \text { iff } a(i)=a(j).$$
$$[a]T_i[b] \text { iff }a\upharpoonright n\smallsetminus \{i\}=b\upharpoonright n\smallsetminus \{i\}.$$
Define accessibility relations corresponding to the polyadic (transpositions) operations as follows:
$$[a]S_{ij}[b] \text { iff } a\circ [i,j]=b.$$
This, as easily checked, defines a $\sf PEA_n$
atom structure. Let $\C$ be the complex algebra.
Let $k>0$ be given. We show that \pe\ has a \ws\ in the usual graph game in $k$ rounds
(now there is no restriction here on the size of the graphs) on $\At\C$.
We recall the `usual atomic' $k$ rounded game $G_k$ played on coloured graphs.
\pa\ picks a graph $M_0\in \G$ with $M_0\subseteq n$ and
$\exists$ makes no response
to this move. In a subsequent round, let the last graph built be $M_i$.
\pa\ picks
\begin{itemize}
\item a graph $\Phi\in \G$ with $|\Phi|=n,$
\item a single node $m\in \Phi,$
\item a coloured graph embedding $\theta:\Phi\smallsetminus \{m\}\to M_i.$
Let $F=\phi\smallsetminus \{m\}$. Then $F$ is called a face.
\pe\ must respond by amalgamating
$M_i$ and $\Phi$ with the embedding $\theta$. In other words she has to define a
graph $M_{i+1}\in C$ and embeddings $\lambda:M_i\to M_{i+1}$
$\mu:\phi \to M_{i+1}$, such that $\lambda\circ \theta=\mu\upharpoonright F.$
\end{itemize}

We define \pe\ s strategy for choosing labels for edges and $n-1$ tuples in response to \pa\ s moves.
Assume that we are at round $r+1$. Our  arguments are similar to the arguments in \cite[Lemmas, 41-43]{r}.

Let $M_0, M_1,\ldots M_r$, $r<k$ be the coloured graphs at the start of a play of $G_k(\alpha)$ just before round $r+1$.
Assume inductively that \pe\ computes a partial function $\rho_s:\N^{-1}\to \N$, for $s\leq r$, that will help her choose
the suffices of the chosen red in the critical case. In our previous rainbow construction we had the additional shade of red 
$\rho$ that did the job. Now we do not have it, so we proceed differently.  Inductively
for $s\leq r$ we assume:

\begin{enumarab}
\item  If $M_s(x,y)$ is green then $(x,y)$ belongs  \pa\ in $M_s$ (meaning he coloured it),

\item $\rho_0\subseteq \ldots \rho_r\subseteq\ldots,$
\item $\dom(\rho_s)=\{i\in \N^{-1}: \exists t\leq s, x, x_0, x_1,\ldots, x_{n-2}
\in \nodes(M_t)\\
\text { where the $x_i$'s form the base of a cone, $x$ is its appex and $i$ its tint }\}.$

The domain consists of the tints of cones created at an earlier stage,

\item $\rho_s$ is order preserving: if $i<j$ then $\rho_s(i)<\rho_s(j)$. The range
of $\rho_s$ is widely spaced: if $i<j\in \dom\rho_s$ then $\rho_s(i)-\rho_s(j)\geq  3^{m-r}$, where $m-r$
is the number of rounds remaining in the game,

\item For $u,v,x_0\in \nodes(M_s)$, if $M_s(u,v)=\r_{\mu,\delta}$, $M_s(x_0,u)=\g_0^i$, $M_s(x_0,v)=\g_0^j$,
where $i,j$ are tints of two cones, with base $F$ such that $x_0$ is the first element in $F$ under the induced linear order,
then $\rho_s(i)=\mu$ and $\rho_s(j)=\delta,$

\item $M_s$ is a a  coloured graph,

\item If the base of a cone $\Delta\subseteq M_s$  with tint $i$ is coloured $y_S$, then $i\in S$.

\end{enumarab}

To start with if \pa\ plays $a$ in the initial round then $\nodes(M_0)=\{0,1,\ldots, n-1\}$, the
hyperedge labelling is defined by $M_0(0,1,\ldots, n)=a$.

In response to a cylindrifier move for some $s\leq r$, involving a $p$ cone, $p\in \N^{-1}$,
\pe\ must extend $\rho_r$ to $\rho_{r+1}$ so that $p\in \dom(\rho_{r+1})$
and the gap between elements of its range is at least $3^{m-r-1}$. Properties (3) and (4) are easily
maintained in round $r+1$. Inductively, $\rho_r$ is order preserving and the gap between its elements is
at least $3^{m-r}$, so this can be maintained in a further round.
If \pa\ chooses a green colour, or green colour whose suffix
already belong to $\rho_r$, there would be fewer
elements to add to the domain of $\rho_{r+1}$, which makes it easy for \pe\ to define $\rho_{r+1}$.

Now assume that at round $r+1$, the current coloured graph is $M_r$ and that   \pa\ chose the graph $\Phi$, $|\Phi|=n$
with distinct nodes $F\cup \{\delta\}$, $\delta\notin M_r$, and  $F\subseteq M_r$ has size
$n-1$.  We can  view \pe\ s move as building a coloured graph $M^*$ extending $M_r$
whose nodes are those of $M_r$ together with the new node $\delta$ and whose edges are edges of $M_r$ together with edges
from $\delta$ to every node of $F$.

Now \pe\ must extend $M^*$ to a complete graph $M^+$ on the same nodes and
complete the colouring giving  a graph $M_{r+1}=M^+$ in $\G$ (the latter is the class of coloured graphs).
In particular, she has to define $M^+(\beta, \delta)$ for all nodes
$\beta\in M_r\sim F$, such that all of the above properties are maintained.

\begin{enumarab}

\item  If $\beta$ and $\delta$ are both apexes of two cones on $F$.

Assume that the tint of the cone determined by $\beta$ is $a\in \N^{-1}$, and the two cones
induce the same linear ordering on $F$. Recall that we have $\beta\notin F$, but it is in $M_r$, while $\delta$ is not in $M_r$,
and that $|F|=n-1$.
By the rules of the game  \pe\ has no choice but to pick a red colour. \pe\ uses her auxiliary
function $\rho_{r+1}$ to determine the suffices, she lets $\mu=\rho_{r+1}(p)$, $b=\rho_{r+1}(q)$
where $p$ and $q$ are the tints of the two cones based on $F$,
whose apexes are $\beta$ and $\delta$. Notice that $\mu, b\in \N$; then she sets $N_s(\beta, \delta)=\r_{\mu,b}$
maintaining property (5), and so $\delta\in \dom(\rho_{r+1})$
maintaining property (4). We check consistency to maintain property (6).

Consider a triangle of nodes $(\beta, y, \delta)$ in the graph $M_{r+1}=M^+$.
The only possible potential problem is that the edges $M^+(y,\beta)$ and $M^+(y,\delta)$ are coloured green with
distinct superscripts $p, q$ but this does not contradict
forbidden triangles of the form involving $(\g_0^p, \g_0^q, \r_{kl})$, because $\rho_{r+1}$ is constructed to be
order preserving.  Now assume that
$M_r(\beta, y)$ and $M_{r+1}(y, \delta)$ are both red (some $y\in \nodes(M_r)$).
Then \pe\ chose the red label $N_{r+1}(y,\delta)$, for $\delta$ is a new node.
We can assume that  $y$ is the apex of a $t$ cone with base $F$ in $M_r$. If not then $N_{r+1}(y, \delta)$ would be coloured
$\w$ by \pe\   and there will be no problem. All properties will be maintained.
Now $y, \beta\in M$, so by by property (5) we have $M_{r+1}(\beta,y)=\r_{\rho+1(p), \rho+1(t)}.$
But $\delta\notin M$, so by her strategy,
we have  $M_{r+1}(y,\delta)=\r_{\rho+1(t), \rho+1(q)}.$ But $M_{r+1}(\beta, \delta)=\r_{\rho+1(p), \rho+1(q)}$,
and we are done.  This is consistent triple, and so have shown that
forbidden triples of reds are avoided.

\item If there is no $f\in F$ such that $M^*(\beta, f), M^*(\delta,f)$ are coloured $g_0^t$, $g_0^u$ for some $t, u$ respectively,
then \pe\ defines $M^+(\beta, \delta)$ to be $\w_0$.

\item If this is not the case, and  for some $0<i<n-1$ there is no $f\in F$ such
that $M^*(\beta, f), M^* (\delta, f)$ are both coloured $\g_i$,
she chooses $\w_i$ for  $M^+{(\beta,\delta)}$.

It is clear that these choices in the last two items  avoid all forbidden triangles (involving greens and whites).

\end{enumarab}

She has not chosen green maintaining property (1).  Now we turn to colouring of $n-1$ tuples,
to make sure that $M^+$ is a coloured graph maintaining property (7).

Let $\Phi$ be the graph chosen by \pa\, it has set of node $F\cup \{\delta\}$.
For each tuple $\bar{a}=a_0,\ldots a_{n-2}\in {M^+}^{n-1}$, $\bar{a}\notin M^{n-1}\cup \Phi^{n-1}$,  with no edge
$(a_i, a_j)$ coloured green (we already have all edges coloured), then  \pe\ colours $\bar{a}$ by $\y_S$, where
$$S=\{i\in A: \text { there is an $i$ cone in $M^*$ with base $\bar{a}$}\}.$$
We need to check that such labeling works, namely that last property is maintained.

Recall that $M$ is the current coloured graph, $M^*=M\cup \{\delta\}$ is built by \pa\ s move
and $M^+$ is the complete labelled graph by \pe\, whose nodes are labelled by \pe\ in response to \pa\ s moves.
We need to show that $M^+$ is labelled according to
the rules of the game, namely, that it is in $\G$.
It can be checked $(n-1)$ tuples are labelled correctly, by yellow colours using
the same argument in \cite[p.16]{Hodkinson} \cite[p.844]{HH}
and \cite{HHbook2}.

We show that \pa\ has a \ws\ in $F^{n+3}$, the argument used is the $\CA$ analogue of \cite[Theorem 33, Lemma 41]{r}.
The difference is that in the relation algebra case, the game is played on atomic networks, but now it is translated to playing on coloured graphs,
\cite[lemma 30]{HH}.

In the initial round \pa\ plays a graph $\Gamma$ with nodes $0,1,\ldots n-1$ such that $\Gamma(i,j)=\w_0$ for $i<j<n-1$
and $\Gamma(i, n-1)=\g_i$
$(i=1, \ldots, n-2)$, $\Gamma(0,n-1)=\g_0^0$ and $\Gamma(0,1\ldots, n-2)=\y_{B}$.

In the following move \pa\ chooses the face $(0,\ldots n-2)$ and demands a node $n$
with $\Gamma_2(i,n)=\g_i$ $(i=1,\ldots, n-2)$, and $\Gamma_2(0,n)=\g_0^{-1}.$
\pe\ must choose a label for the edge $(n+1,n)$ of $\Gamma_2$. It must be a red atom $r_{mn}$. Since $-1<0$ we have $m<n$.
In the next move \pa\ plays the face $(0, \ldots, n-2)$ and demands a node $n+1$, with $\Gamma_3(i,n)=\g_i$ $(i=1,\ldots, n-2)$,
such that  $\Gamma_3(0,n+2)=\g_0^{-2}$.
Then $\Gamma_3(n+1,n)$ and  $\Gamma_3(n+1,n-1)$ both being red, the indices must match.
$\Gamma_3(n+1,n)=r_{ln}$ and $\Gamma_3(n+1, n-1)=r_{lm}$ with $l<m$.
In the next round \pa\ plays $(0,1,\ldots, n-2)$ and reuses the node $2$ such that $\Gamma_4(0,2)=\g_0^{-3}$.
This time we have $\Gamma_4(n,n-1)=\r_{jl}$ for some $j<l\in \N$.
Continuing in this manner leads to a decreasing sequence in $\N$.

Now that \pa\ has a \ws\ in $F^{n+3},$ it follows by theorem \ref{thm:n} that $\Rd_{sc}\C\notin S_c\Nr_n\Sc_{n+3}$, but it is elementary equivalent
to a countable completely representable algebra. Indeed, using ultrapowers and an elementary chain argument,
we obtain $\B$ such  $\C\equiv \B$ \cite[lemma 44]{r},
and \pe\ has a \ws\ in $G_{\omega}$ (the usual $\omega$ rounded atomic game) 
on $\B$ so by \cite[theorem 3.3.3]{HHbook2}, $\B$ is completely representable. 

In more detail we have \pe\ has a \ws\ $\sigma_k$ in $G_k$.
We can assume that $\sigma_k$ is deterministic.
Let $\D$ be a non-principal ultrapower of $\C$.  One can show that
\pe\ has a \ws\ $\sigma$ in $G(\At\D)$ --- essentially she uses
$\sigma_k$ in the $k$'th component of the ultraproduct so that at each
round of $G(\At\D)$,  \pe\ is still winning in co-finitely many
components, this suffices to show she has still not lost.

We can also assume that $\C$ is countable. If not then replace it by its subalgebra generated by the countably many atoms
(the term algebra); \ws\ s that depended only on the atom structure persist
for both players.

Now one can use an
elementary chain argument to construct a chain of countable elementary
subalgebras $\C=\A_0\preceq\A_1\preceq\ldots\preceq\ldots \D$ inductively in this manner.
One defines  $\A_{i+1}$ be a countable elementary subalgebra of $\D$
containing $\A_i$ and all elements of $\D$ that $\sigma$ selects
in a play of $G(\At\D)$ in which \pa\ only chooses elements from
$\A_i$. Now let $\B=\bigcup_{i<\omega}\A_i$.  This is a
countable elementary subalgebra of $\D$, hence necessarily atomic,  and \pe\ has a \ws\ in
$G(\At\B)$, so by \cite[Theorem 3.3.3]{HHbook2}, noting that $\B$ is countable,
then $\B$ is completely representable; furthermore $\B\equiv \C$.
\end{proof}

In \cite{r} more sophisticated  games are devised for relation algebras and Robin Hirsch deduces from the fact that \pe\ can win the
$k$ rounded game on a  certain atomic relation algebra $\A$ for every finite $k$,
then $\A$ is elementary equivalent to $\B\in \Ra\CA_{\omega}$. This is a mistake.
All we can deduce from \pe\ s \ws\ is that $\B\in S_c\Ra\CA_{\omega}$ but $\At\B\cong \At\D$ with $\D\in \Ra\CA_{\omega}$;
this does not mean that $\B$ itself is in $\Ra\CA_{\omega}$.

Indeed we show that we may well have algebras $\A, \B\in \CA_n$ $(n>1$),
and even more in $\RCA_n$
such that $\At\A=\At\B$, $\A\in \Nr_n\CA_{\omega}$ but $\B\notin \Nr_n\CA_{n+1}$, {\it a fortiori}
$\B\notin \Nr_n\CA_{\omega}$.

\begin{example}\label{SL}
$FT_{\alpha}$ denotes the set of all finite transformations on $\alpha$.
Let $\alpha$ be an ordinal $>1$; could be infinite. Let $\F$ is field of characteristic $0$.
$$V=\{s\in {}^{\alpha}\F: |\{i\in \alpha: s_i\neq 0\}|<\omega\},$$
$${\C}=(\wp(V),
\cup,\cap,\sim, \emptyset , V, {\sf c}_{i},{\sf d}_{i,j}, {\sf s}_{\tau})_{i,j\in \alpha, \tau\in FT_{\alpha}}.$$
Then clearly $\wp(V)\in \Nr_{\alpha}\sf QEA_{\alpha+\omega}$.
Indeed let $W={}^{\alpha+\omega}\F^{(0)}$. Then
$\psi: \wp(V)\to \Nr_{\alpha}\wp(W)$ defined via
$$X\mapsto \{s\in W: s\upharpoonright \alpha\in X\}$$
is an isomorphism from $\wp(V)$ to $\Nr_{\alpha}\wp(W)$.
We shall construct an algebra $\A$, $\A\notin \Nr_{\alpha}{\sf QPEA}_{\alpha+1}$.
Let $y$ denote the following $\alpha$-ary relation:
$$y=\{s\in V: s_0+1=\sum_{i>0} s_i\}.$$
Let $y_s$ be the singleton containing $s$, i.e. $y_s=\{s\}.$
Define
${\A}\in {\sf QEA}_{\alpha}$
as follows:
$${\A}=\Sg^{\C}\{y,y_s:s\in y\}.$$

Now clearly $\A$ and $\wp(V)$ {\it share the same atom structure}, namely, the singletons.
Then we claim that
$\A\notin \Nr_{\alpha}{\sf QEA}_{\beta}$ for any $\beta>\alpha$.
The first order sentence that codes the idea of the proof says
that $\A$ is neither an elementary nor complete subalgebra of $\wp(V)$. We use $\land$ and $\to$ in the meta language
with their usual meaning.
Let $\At(x)$ be the first order formula asserting that $x$ is an atom.
Let $$\tau(x,y) ={\sf c}_1({\sf c}_0x\cdot {\sf s}_1^0{\sf c}_1y)\cdot {\sf c}_1x\cdot {\sf c}_0y.$$
Let $${\sf Rc}(x):={\sf c}_0x\cap {\sf c}_1x=x,$$
$$\phi:=\forall x(x\neq 0\to \exists y(\At(y)\land y\leq x))\land
\forall x(\At(x) \to {\sf Rc}(x)),$$
$$\alpha(x,y):=\At(x)\land x\leq y,$$
and  $\psi (y_0,y_1)$ be the following first order formula
$$\forall z(\forall x(\alpha(x,y_0)\to x\leq z)\to y_0\leq z)\land
\forall x(\At(x)\to {\sf At}({\sf c}_0x\cap y_0)\land {\sf At}({\sf c}_1x\cap y_0))$$
$$\to [\forall x_1\forall x_2(\alpha(x_1,y_0)\land \alpha(x_2,y_0)\to \tau(x_1,x_2)\leq y_1)$$
$$\land \forall z(\forall x_1 \forall x_2(\alpha(x_1,y_0)\land \alpha(x_2,y_0)\to
\tau(x_1,x_2)\leq z)\to y_1\leq z)].$$
Then
$$\Nr_{\alpha}{\sf QEA}_{\beta}\models \phi\to \forall y_0 \exists y_1 \psi(y_0,y_1).$$
But this formula does not hold in $\A$.
We have $\A\models \phi\text {  and not }
\A\models \forall y_0\exists y_1\psi (y_0,y_1).$
In words: we have a set $X=\{y_s: s\in V\}$ of atoms such that $\sum^{\A}X=y,$ and $\A$
models $\phi$ in the sense that below any non zero element there is a
{\it rectangular} atom, namely a singleton.

Let $Y=\{\tau(y_r,y_s), r,s\in V\}$, then
$Y\subseteq \A$, but it has {\it no supremum} in $\A$, but {\it it does have one} in any full neat reduct $\B$ containing $\A$,
and this is $\tau_{\alpha}^{\B}(y,y)$, where
$\tau_{\alpha}(x,y) = {\sf c}_{\alpha}({\sf s}_{\alpha}^1{\sf c}_{\alpha}x\cdot {\sf s}_{\alpha}^0{\sf c}_{\alpha}y).$

In $\wp(V)$ this last is $w=\{s\in {}^{\alpha}\F^{(\bold 0)}: s_0+2=s_1+2\sum_{i>1}s_i\},$
and $w\notin \A$. The proof of this can be easily distilled from \cite[main theorem]{SL}.
For $y_0=y$, there is no $y_1\in \A$ satisfying $\psi(y_0,y_1)$.
Actually the above proof proves more. It proves that there is a
$\C\in \Nr_{\alpha}{\sf QEA}_{\beta}$ for every $\beta>\alpha$ (equivalently $\C\in \Nr_{\alpha}\QEA_{\omega}$), and $\A\subseteq \C$, such that
$\Rd_{\Sc}\A\notin \Nr_{\alpha}\Sc_{\alpha+1}$.
See \cite[Theorems 5.1.4-5.1.5]{Sayedneat} for an entirely different example.

\end{example}

\subsection{ Neat embeddings in connection to
complete and strong representations}

Here we approach the notion of complete representations and strong representations using neat embeddings.
In this subsection $\K$ is any class between $\Sc$ and $\PEA$ and $n$ denotes a finite ordinal.
Lifting from atom structures, we set:

\begin{definition} An atomic completely additive algebra $\A\in \K_n$ is {\it strongly representable} 
if its \d\ completion, namely, $\Cm\At\A$ is representable.
\end{definition}

It is not hard to see that a completely representable algebra is completely additive and also 
strongly representable. Indeed if $\A$ is completely representable then it is isomorphic to a generalized set algebra where joins, even infinite ones, 
are 
unions, and it is easy to see that the extra non-Boolean operators of substitutions and cylindrifiers distribute over arbitrary joins.
Also if $\A$ is completely representable, then it satisfies all Lyndon conditions hence will be strongly representable.
The converse is false as will be shown in our next example. On the other hand,  not every atomic 
completely additive representable algebra is strongly representable; our algebra $\A$ in theorem \ref{can} is an example; in fact we shall see that
unlike the class of representable algebras which is a variety, 
the class of strongly representable $\K$ algebras of dimension $>2$ is not elementary.

Let ${\sf CRK_n}$ denote the class of completely representable $\K$ algebras of dimension $n$.
In \cite{Sayedneat} it is shown that $S_c\Nr_n\CA_{\omega}$ and ${\sf CRCA}_{\omega}$ coincide on countable atomic algebras.
This characterization works for any $\sf K$ between $\Sc$ and $\PEA$, the argument used is an omitting types argument implemented 
via the Baire category theorem.
The result can be slightly generalized to allow algebras with countably many atoms, that may not be countable.
In our next theorem we show that this does not generalize any
further as far as cardinalities are concerned.

\begin{theorem}\label{complete} For any $n>2$ we have, $\Nr_n\sf K_{\omega}\nsubseteq {\sf CRK}_{n}$,
while for $n>1$, ${\sf CRK}_n\nsubseteq {\bf UpUr}\Nr_{n}\sf K_{\omega}.$
In particular, there are completely representable, hence strongly representable algebras that are not in $\Nr_n\sf K_{\omega}$.
Furthermore, such algebras can be countable.
However, for any $n\in \omega$, if $\A\in \bold S_c\Nr_n\sf K_{\omega}$ is atomic and completely additive, 
then $\A$ is strongly
representable.
\end{theorem}
\begin{proof}
\begin{enumarab}
\item We first show that $\Nr_n\sf K_{\omega}\nsubseteq {\sf CRK}_{\omega}$.
This cannot be witnessed on countable algebras, 
so our constructed neat reduct that is not
completely representable,  must be uncountable.

In \cite{r} a sketch  of constructing  an uncountable relation algebra
$\R\in \Ra\CA_{\omega}$ (having an $\omega$ dimensional cylindric basis)
with no complete representation is given. It has a precursor in \cite{BSL}
which is the special case of this example when $\kappa=\omega$
but the idea in all three proofs are very similar
using a variant of the rainbow  relation algebra
$\R_{\omega_1, \omega}$.

Now we give the details of the construction in \cite[Remark 31]{r}.

Using the terminology of rainbow constructions,
we allow the greens to be of cardinality $2^{\kappa}$ for any
infinite cardinal $\kappa$, and the reds to be of cardinality $\kappa$.
Here a \ws\ for \pa\ witnesses that  the algebra has {\it no} complete
representation. But this is not enough because we want our algebra to be
in $\Ra\CA_{\omega}$; we will show that it will be.

As usual we specify the atoms and forbidden triples.
The atoms are $\Id, \; \g_0^i:i<2^{\kappa}$ and $\r_j:1\leq j<
\kappa$, all symmetric.  The forbidden triples of atoms are all
permutations of $(\Id, x, y)$ for $x \neq y$, \/$(\r_j, \r_j, \r_j)$ for
$1\leq j<\kappa$ and $(\g_0^i, \g_0^{i'}, \g_0^{i^*})$ for $i, i',
i^*<2^{\kappa}.$  In other words, we forbid all the monochromatic
triangles.

Write $\g_0$ for $\set{\g_0^i:i<2^{\kappa}}$ and $\r_+$ for
$\set{\r_j:1\leq j<\kappa}$. Call this atom
structure $\alpha$.

Let $\A$ be the term algebra on this atom
structure; the subalgebra of $\Cm\alpha$ generated by the atoms.  $\A$ is a dense subalgebra of the complex algebra
$\Cm\alpha$. We claim that $\A$, as a relation algebra,  has no complete representation.

Indeed, suppose $\A$ has a complete representation $M$.  Let $x, y$ be points in the
representation with $M \models \r_1(x, y)$.  For each $i< 2^{\kappa}$, there is a
point $z_i \in M$ such that $M \models \g_0^i(x, z_i) \wedge \r_1(z_i, y)$.

Let $Z = \set{z_i:i<2^{\kappa}}$.  Within $Z$ there can be no edges labeled by
$\r_0$ so each edge is labelled by one of the $\kappa$ atoms in
$\r_+$.  The Erdos-Rado theorem forces the existence of three points
$z^1, z^2, z^3 \in Z$ such that $M \models \r_j(z^1, z^2) \wedge \r_j(z^2, z^3)
\wedge \r_j(z^3, z_1)$, for some single $j<\kappa$.  This contradicts the
definition of composition in $\A$ (since we avoided monochromatic triangles).

Let $S$ be the set of all atomic $\A$-networks $N$ with nodes
 $\omega$ such that $\{\r_i: 1\leq i<\kappa: \r_i \text{ is the label
of an edge in N}\}$ is finite.
Then it is straightforward to show $S$ is an amalgamation class, that is for all $M, N
\in S$ if $M \equiv_{ij} N$ then there is $L \in S$ with
$M \equiv_i L \equiv_j N.$
Hence the complex cylindric algebra $\Ca(S)\in \sf QEA_\omega$.

Now let $X$ be the set of finite $\A$-networks $N$ with nodes
$\subseteq\omega$ such that
\begin{enumerate}
\item each edge of $N$ is either (a) an atom of
$\A$ or (b) a cofinite subset of $\r_+=\set{\r_j:1\leq j<\kappa}$ or (c)
a cofinite subset of $\g_0=\set{\g_0^i:i<2^{\kappa}}$ and
\item $N$ is `triangle-closed', i.e. for all $l, m, n \in \nodes(N)$ we
have $N(l, n) \leq N(l,m);N(m,n)$.  That means if an edge $(l,m)$ is
labeled by $\Id$ then $N(l,n)= N(mn)$ and if $N(l,m), N(m,n) \leq
\g_0$ then $N(l,n)\cdot \g_0 = 0$ and if $N(l,m)=N(m,n) =
\r_j$ (some $1\leq j<\omega$) then $N(l,n).\r_j = 0$.
\end{enumerate}
For $N\in X$ let $N'\in\Ca(S)$ be defined by
\[\set{L\in S: L(m,n)\leq
N(m,n) \mbox{ for } m,n\in nodes(N)}\]
For $i\in \omega$, let $N\restr{-i}$ be the subgraph of $N$ obtained by deleting the node $i$.
Then if $N\in X, \; i<\omega$ then $\cyl i N' =
(N\restr{-i})'$.
The inclusion $\cyl i N' \subseteq (N\restr{-i})'$ is clear.

Conversely, let $L \in (N\restr{-i})'$.  We seek $M \equiv_i L$ with
$M\in N'$.  This will prove that $L \in \cyl i N'$, as required.
Since $L\in S$ the set $X = \set{\r_i \notin L}$ is infinite.  Let $X$
be the disjoint union of two infinite sets $Y \cup Y'$, say.  To
define the $\omega$-network $M$ we must define the labels of all edges
involving the node $i$ (other labels are given by $M\equiv_i L$).  We
define these labels by enumerating the edges and labeling them one at
a time.  So let $j \neq i < \omega$.  Suppose $j\in \nodes(N)$.  We
must choose $M(i,j) \leq N(i,j)$.  If $N(i,j)$ is an atom then of
course $M(i,j)=N(i,j)$.  Since $N$ is finite, this defines only
finitely many labels of $M$.  If $N(i,j)$ is a cofinite subset of
$a_0$ then we let $M(i,j)$ be an arbitrary atom in $N(i,j)$.  And if
$N(i,j)$ is a cofinite subset of $\r_+$ then let $M(i,j)$ be an element
of $N(i,j)\cap Y$ which has not been used as the label of any edge of
$M$ which has already been chosen (possible, since at each stage only
finitely many have been chosen so far).  If $j\notin \nodes(N)$ then we
can let $M(i,j)= \r_k \in Y$ some $1\leq k < \kappa$ such that no edge of $M$
has already been labeled by $\r_k$.  It is not hard to check that each
triangle of $M$ is consistent (we have avoided all monochromatic
triangles) and clearly $M\in N'$ and $M\equiv_i L$.  The labeling avoided all
but finitely many elements of $Y'$, so $M\in S$. So
$(N\restr{-i})' \subseteq \cyl i N'$.

Now let $X' = \set{N':N\in X} \subseteq \Ca(S)$.
Then the subalgebra of $\Ca(S)$ generated by $X'$ is obtained from
$X'$ by closing under finite unions.
Clearly all these finite unions are generated by $X'$.  We must show
that the set of finite unions of $X'$ is closed under all cylindric
operations.  Closure under unions is given.  For $N'\in X$ we have
$-N' = \bigcup_{m,n\in \nodes(N)}N_{mn}'$ where $N_{mn}$ is a network
with nodes $\set{m,n}$ and labeling $N_{mn}(m,n) = -N(m,n)$. $N_{mn}$
may not belong to $X$ but it is equivalent to a union of at most finitely many
members of $X$.  The diagonal $\diag ij \in\Ca(S)$ is equal to $N'$
where $N$ is a network with nodes $\set{i,j}$ and labeling
$N(i,j)=\Id$.  
Let $\C$ be the subalgebra of $\Ca(S)$ generated by $X'$.
Then we claim that $\A = \Ra(\C)$.
Each element of $\A$ is a union of a finite number of atoms and
possibly a co-finite subset of $a_0$ and possibly a co-finite subset
of $a_+$.  Clearly $\A\subseteq\Ra(\C)$.  Conversely, each element
$z \in \Ra(\C)$ is a finite union $\bigcup_{N\in F}N'$, for some
finite subset $F$ of $X$, satisfying $\cyl i z = z$, for $i > 1$. Let $i_0,
\ldots, i_k$ be an enumeration of all the nodes, other than $0$ and
$1$, that occur as nodes of networks in $F$.  Then, $\cyl
{i_0} \ldots
\cyl {i_k}z = \bigcup_{N\in F} \cyl {i_0} \ldots
\cyl {i_k}N' = \bigcup_{N\in F} (N\restr{\set{0,1}})' \in \A$.  So $\Ra(\C)
\subseteq \A$.
We have shown that $\A$ is relation algebra reduct of $\C\in\QEA_\omega$ but has no
complete representation. 

Let $n>2$. Let $\B=\Nr_n\C$. Then
$\B\in \Nr_n\sf QEA_{\omega}$, is atomic, but has no complete representation; in fact because it is binary generated its
diagonal free  reduct is not completely representable \cite[Theorem 5.4.26]{HMT2}. 

\item That ${\sf CRTA}_n\nsubseteq {\bf UpUr}\Nr_{n}\sf QEA_{\omega}$ for $n>1$ follows
from example \ref{SL}.

\item Now for the last part, namely, that  complete subneat reducts are strongly representable.
Let $n>2$. Let $\A\in \bold S_c\Nr_n\CA_{\omega}$ be atomic and completely additive. Then \pe\ has a \ws\ in $F^{\omega}$
by theorem \ref{thm:n}
hence it has a \ws\ in $G$ (the usual $\omega$ rounded atomic game) and so it has
a \ws\ for $G_k$ for all finite $k$ ($G$ truncated to $k$ rounds.)
Thus $\A\models \sigma_k$ which is the $k$ th Lyndon sentence coding that \pe\ has a
\ws\ in $G_k$, called the $k$the Lyndon condition. Since $\A$ satisfies the $k$th
Lyndon conditions for each $k$,  then any algebra on its atom structure is representable,
so that $\Cm\At\A$ is representable, hence  it is strongly representable, and we are done.
\end{enumarab}
\end{proof}

\section{Non-finite axiomatizability}

We now prove a non-finite axiomatizability result, addressing  diagonal free reducts of cylindric and polyadic algebras answering
an open problem formulated in \cite{ST}. 
Here the methods of `Andr\'eka's splitting', adopted in \cite{Andreka} to prove the analogous result for cylindric algebras
depend essentially on the presence of diagonal elements and it  does not work when the signature lacks diagonal elements, 
so accordingly we
use a rainbow construction instead.   The rainbow construction 
we use is a cylindric version of that used in \cite[Definition 17.8]{HHbook}.
Throughout this subsection $n$ is a finite ordinal $>2$.

We alter the used colours a little bit. The reds will  have only a single index. 
We deal with finite rainbow algebras. Let $\kappa, \mu$ be finite ordinals $>0$.
The colours we use are:
\begin{itemize}

\item greens: $\g_i$ ($1\leq i<n-2)\cup \{\g^0_i: i\in \mu\},$
\item whites : $\w_i, i <n-1,$
\item reds:  $\r_{i}$, $i\in \kappa,$

\item shades of yellow : $\y_S: S\subseteq \mu.$
\end{itemize}

And coloured graphs are:

\begin{enumarab}

\item $M$ is a complete graph.

\item $M$ contains no triangles (called forbidden triples)
of the following types:

\vspace{-.2in}
\begin{eqnarray*}
&&\nonumber\\
(\g, \g^{'}, \g^{*}), (\g_i, \g_{i}, \w_i)
&&\mbox{any }1\leq i\leq n-1\  \\
(\g^j_0, \g^k_0, \w_0)&&\mbox{ any } j, k\in \mu\\
(\r_{i}, \r_{i}, \r_{j}) && i,j\in \kappa. \\
\end{eqnarray*}
and no other triple of atoms is forbidden.
The items concerning cones and $n-1$ tuples are the same as before.
\end{enumarab}
We denote the resulting rainbow polyadic equality (finite) algebra defined as before 
from the new coloured graphs by $\A_{\mu,\kappa}$.

Let $m>1$. Let $\alpha=n.2^m$ and $\beta=(\alpha+1)(\alpha+2)/2.$
Let $\A=\A_{\alpha+2,\beta}$ and $\B=\A_{\alpha+2,\alpha}$, here $\alpha+2$ is the number of greens.
We consider the usual atomic $k$ rounded game $G_k$ played on coloured graphs. Recall that 
$G$ denotes the $\omega$ rounded atomic game.
\begin{theorem}
\begin{enumarab}
\item \pa\ has a \ws\  for $\B$  in $G_{\alpha+2}$; hence $\Rd_{df}\B\notin {\sf RDf_n}.$
\item \pe\ has a \ws\ for $G$ on $\A$, hence $\A\in {\sf RPEA_n}.$
\end{enumarab}
\end{theorem}
\begin{proof}

We first show that \pa\ has a \ws\  for $\B$  in $\alpha+2$ rounds; hence it is not representable,
and by \cite[Theorem 5.4.26]{HMT2}, its diagonal free reduct $\Rd_{df}\B\notin {\sf RDf_n}.$

\pa\ plays a coloured graph $M$ with
nodes $0, 1,\ldots, n-1$ and such that $M(i, j) = \w_0 (i < j <
n-1), M(i, n-1) = \g_i ( i = 1,\ldots, n), M(0, n-1) =
\g^0_0$, and $ M(0, 1,\ldots, n-2) = \y_{\alpha+2}$. This is a $0$-cone
with base $\{0,\ldots , n-2\}$. In the following moves, \pa\
repeatedly chooses the face $(0, 1,\ldots, n-2)$
and demands a node
$t<\alpha+2$ with $\Phi(i,\alpha) = \g_i (i = 1,\ldots,  n-2)$ and $\Phi(0, t) = \g^t_0$,
in the graph notation -- i.e., a $t$ -cone on the same base.
\pe\, among other things, has to colour all the edges
connecting nodes. By the rules of the game
the only permissible colours would be red. Using this, \pa\ can force a
win in $\alpha+2$ rounds eventually using her enough supply of greens,
which \pe\ cannot match using his $<$ number of reds. 
The conclusion now follows since $\B$ is generated by elements whose dimension sets
are $<n$.

But we claim that $\A\in {\sf RPEA_n}.$
If \pa\ plays like before, now \pe\ has more  reds, so \pa\ cannot force a win. In fact \pa\
can only force a red clique of size $\alpha+2$, not bigger. So \pe\ s
strategy within red cliques is to choose a label for each edge using
a red colour and to ensure that each edge within the clique has a label unique to this edge (within the clique).
Since there are $\beta$ many reds she can do that.

Let $M$ be a coloured graph built at some stage, and let \pa\ choose the graph $\Phi$, $|\Phi|=n$, then $\Phi=F\cup \{\delta\}$,
where  $F\subseteq M$ and $\delta\notin M$.
So we may view \pa\ s move as building a coloured graph $M^*$ extending $M$
whose nodes are those of $\Gamma$ together with $\delta$ and whose edges are edges of $\Gamma$ together with edges
from $\delta$ to every node of $F$.

Colours of edges and $n-1$ tuples in $M^*$ but not
in $M$ are determined by \pa\ moves.
No $n-1$ tuple containing both $\delta$ and elements of $M\sim F$
has a colour in $M^*.$

Now \pe\ must extend $M^*$ to a complete the graph on the same nodes and
complete the colouring giving  a complete coloured graph $M.$ 
In particular, she has to define $M^+(\beta, \delta)$
for all nodes  $\beta\in M\sim F$.
\begin{enumarab}
\item  if $\beta$ and $\delta$ are both apexes of cones on $F$, that induces the same linear ordering on $F$, the
\pe\ has no choice but to pick a  red atom, and as we described above, she can choose one avoiding 
inconsistencies.

\item Other wise, this is not the case, so for some $i<n-1$ there is no $f\in F$ such
that $M^*(\beta, f), M^* (f,\delta)$ are both coloured $\g_i$ or if $i=0$, they are coloured
$\g_0^l$ and $\g_0^{l'}$ for some $l$ and $l'$.
\end{enumarab}
In the second case \pe\ uses the normal strategy in rainbow constructions.
She chooses $\w_0$, for $M^+(\beta,\delta)$.

Now we turn to colouring of $n-1$ tuples. For each tuple $\bar{a}=a_0,\ldots a_{n-2}\in M^{n-1}$ with no edge
$(a_i, a_j)$ coloured green, then  \pe\ colours $\bar{a}$ by $\y_S$, where
$$S=\{i\in \alpha+2: \text { there is an $i$ cone in $M^*$ with base $\bar{a}$}\}.$$
This works exactly as in the previous proofs of theorems \ref{can} and \ref{rainbow}.
\end{proof}

Recall that a coloured graph is red, if at least one of its edges are labelled red.
We write $\r$ for $a:n\to \Gamma$, where $\Gamma$ is a red graph, and we call it a red atom.
(Here we identify an atom with its representative, but no harm will follow).

\begin{theorem} Let $m>1$. Then for any $m$ variable equation in the signature  of $\sf Df_n$ 
the two algebras $\Rd_{df}\A$ and $\Rd_{df}\B$,  as defined above, 
falsify it together or satisfy it together.
\end{theorem}
\begin{proof}
Cf. \cite[Lemma 17.10]{HHbook}. We consider the $\sf Df$ reducts of $\A$ and $\B$, which we continue to denote, with a slight abuse of notation, 
$\A$ and $\B$.
Let $\R$ be the set of red atoms of $\A$, and $\R'$ be the set of red atoms in $\B$. Then $|\R|\geq |\R'|\geq n.2^m$.
Assume that the equation $s=t$, using $m$ variables does not hold in $\A$.
Then there is an assignment $h:\{x_0,\ldots, x_{m-1}\}\to \A$, such that $\A, h\models s\neq t$.
We construct an assignment $h'$ into $\B$ that also falsifies $s=t$.
Now $\A$ has more red atoms, but $\A$ and $\B$ have identical non-red atoms. So for any non red atom $a$ of $\B$, and for any
$i<m$, let $a\leq h'(x_i)$ iff $a\leq h(x_i)$.
To complete the definition of $h'$ it remains to identify which red atoms of $\B$ (which are only a part of those of $\A$)
lie below $h'(x_i)$ ($i<m$).
The assignment $h$ induces a partition of $\R$ into $2^m$ parts $\R_S$, $S\subseteq m=\{0,1\ldots, m-1\}$, by
$$\R_S=\{\r: \r\leq h(x_i) , i\in S, \r\cdot h(x_i)=0, i\in m\sim S\}.$$
Partition  $\R'$ into $2^m$ parts $\R'_S$ for $S\subseteq m$ such that $|\R'_S|=|\R_S|$ if $|\R_S|<n$, and $|\R_S'|\geq n$ iff $|\R_S|\geq n$.
This possible because $|\R|\geq n.2^m$.
Now for each $i<m$ and each red atom $\r'$ in $\R'$, we complete the definition of
$h'(x_i)\in \B$ by
$\r'\leq h'(x_i)$ iff $\r'\in \R'_S$ for some $S$ such that $i\in S$.
It can be easily shown inductively that for any term $\tau$ using only the first $m$ variables
and any $S\subseteq m$, we have
\begin{align*}
\R_S\subseteq h(\tau)&\Longleftrightarrow  \R'_S\subseteq h'(\tau),\\
h(\tau)\sim \R& = h'(\tau)\sim \R',\\
|h(\tau)\cap \R|& =|h'(\tau) \cap \R'| \text {  iff   } |h(\tau)\cap \R|<n,\\
|h(\tau) \cap \R| \geq n&\Longleftrightarrow |h'(\tau)\cap \R'|\geq n.\\
\end{align*}
Hence $\B$ does not model $s=t$.
The converse is entirely analogous.
\end{proof}

\begin{corollary}\label{df} For any $\sf K$ such that $\sf Df_n\subseteq \sf K\subseteq \sf PEA_n$, there is no finite variable universal prenex 
axiomatization of $\sf RK_n$, namely,  the class of representable algebras of dimension $n$.
\end{corollary}

\begin{proof} If $\Sigma$ is any $m$ variable equational theory then the appropriate reduct of $\A$ and $\B$ either both validate
$\Sigma$ or neither do. Since one algebra is in $V$ while the other is not, it follows that $\Sigma$ does not axiomatize $V$.
But  $\A$ and $\B$ are simple (in any of the considered signatures), then they
are subdirectly irreducible. In a discriminator variety every universal prenex formula 
is equivalent in subdirectly irreducible members to an
equation using the same number of variables. 
Hence the desired.
\end{proof}

\subsection{Omitting types for clique guarded semantics}

Throughout this section, unless otherwise indicated, 
$m$ will denote the dimension of algebras considered.  
We consider only cylindric algebras but all results extend to all of its relatives
dealt with above; that is for any class between $\Sc$ and $\PEA$.
Unless otherwise indicated $m$ will be always
finite and $>2$.

\begin{definition}
A non-empty set  $M$ is a {\it relativized representation} of an abstract algebra $\A$ in $\sf CA_m$, 
if there exists $V\subseteq {}^mM$, and an injective
homomorphism $f:\A\to \wp(V)$, where cylindric operations are the usual concrete ones relativized to $V$ 
and $M=\bigcup_{s\in V}\rng(s).$ We write $M\models 1(\bar{s})$, if $\bar{s}\in V$.
More generally, we write for $a\in A$,  $M\models a(\bar{s})$ if $\bar{s}\in f(a)$.
We say that {\it $M$ is a relativized representation of $\A$ with representing function  
$f$}, if we want to highlight the role of $f$, or if $f$ is relevant to the situation at hand.
\end{definition}

Let $L(A)$ be the $\L_m$ signature with an $m$ predicate symbol  for
every element of $a\in A$ and $L(A)_n$ denotes the set of 
$n$ variable  formulas in this signature. 

Next we specify certain relativized representations. The idea, borrowed from Hirsch and Hodkinson who introduced such notions for 
relation algebras \cite[Section 1.3.1., p.400-404]{HHbook}, 
is that the properties of genuine representations manifest itself
only locally, that is on $m$- cliques
so we have a notion of {\it clique guarded  semantics}, where witnesses of cylindrifiers are only found if
we zoom in adequately
on the representation by a `movable window'. We may also require that cylindrifiers commute on this localized
level; this is the case with
{\it $n$ flat  representations.}

\begin{definition}\label{rel}
Let $M$ be a relativized representation of a $\CA_m$ and $n>m$.
An {\it $m$ clique in $M$}, or simply a clique in $M$, is a subset $C$ of $M$ such that $M\models 1(\bar{s})$ for all $\bar{s}\in {}^mC$, that is
it can be viewed as a hypergraph such that every $m$ tuple is labelled by the top element.
Let $C^{n}(M)=\{\bar{a}\in {}^nM: \text { $\rng(\bar{a})$ is a clique in $M$}\}.$
\end{definition}

\begin{definition} Let $M$ be a relativized  representation of $\A$ with representing function $f$. 
We define the {\it $n$ dimensional clique guarded semantics}
$M\models_C \phi[\bar{a}]$ for $\phi\in L(A)_n$ and
$\bar{a}\in C^n(M)$ as follows:
\begin{itemize}
\item If $\phi$ is $r(x_{i_0}\ldots x_{i_{m-1}})$, $i:m\to n$,  then $M\models \phi[\bar{a}]$ iff 
$(a_{i_0},\ldots a_{i_{m-1}})\in f(r).$
\item Boolean clauses are as expected.
\item For $i<n$, $M\models \exists x_i\phi[\bar{a}]$ iff $M\models \phi[\bar{b}]$ 
for some $b\in C^n(M)$ with $\bar{b}\equiv_i \bar{a}.$

\end{itemize}
\end{definition}

\begin{definition}
\begin{enumarab}
\item Let $\A\in \CA_m$, and $M$ be a relativized representation of $\A$ with representing function $f$. 
$M$ is said to be {\it $n$ square,}
if whenever $\bar{s}\in C^n(M)$, $a\in A$, $i<m$,
and $l:m\to n$
if $M\models {\sf c}_ia(s_{l(0)}\ldots s_{l(m-1)})$,
then there is a $t\in C^n(M)$ with $\bar{t}\equiv _i \bar{s}$,
and $M\models a(t_{l(0)},\ldots t_{l(m-1)})$, that is, $(t_{l(0)},\ldots t_{l(m-1)})\in f(a)$.

\item Let $M$ be a relativized $n$ square representation of a $\CA_m$ and $n>m$.
$M$ is {\it $n$ flat} if  for all $\phi\in L(A)^n$, for all $\bar{a}\in C^n(M)$, for all $i,j<n$, we have
$$M\models_C [\exists x_i\exists x_j\phi\longleftrightarrow \exists x_j\exists x_i\phi](\bar{a}).$$
\end{enumarab}
\end{definition}

Complete relativized representations are defined exactly like the classical case. 
One can easily show that if $\A$ has a complete relativized representation then it is 
atomic. The following theorem says that a relativized complete representation of the term algebra induces a relativized 
representation of its \d\ completion.

\begin{lemma}\label{complete} If $\A\in \CA_m$ is countable and atomic  and $M$ is an $n$$(n>m)$
flat complete
representation of $\A$ via $f$,  then $M$ gives an $n$ flat representation of $\mathfrak{Cm}\At\A$.
\end{lemma}
\begin{proof} 
Let $f:\A\to \wp(V)$ be a representing function for $\A$ where $V\subseteq {}^nM$.
For $a\in \mathfrak{Cm}\At\A$, $a=\sum X$, say, with $X\subseteq \A$,
define $g(a)=\bigcup_{x\in X}f(x)$. Then $g:\mathfrak{Cm}\At\A\to \wp(V)$
is a complete $n$ flat representation of $\A$.
\end{proof}

\begin{lemma}\label{flat} If $\A\in \CA_m$ has an $n$ flat representation, 
then $\A\in S\Nr_m\CA_n.$ If $\A$ has an $n$ complete flat representation,
then $\A\in S_c\Nr_m\CA_n$. For any $n\geq m+3$, 
the class of algebras having 
$n$ complete flat representations is not 
elementary.
\end{lemma}
\begin{proof}
Let $L(A)$ denotes the $\L_m$ signature that contains
an $m$ ary predicate for every $a\in A$.
Let $M$ be an $n$ flat representation.
For $\phi\in L(A)^n$,
let $\phi^{M}=\{\bar{a}\in C^n(M):M\models_C \phi(\bar{a})\}$. Here $\models_C$ is the $n$ dimensional clique guarded semantics.
Let $\D$ be the algebra with universe $\{\phi^{M}: \phi\in L(A)^n\}$ with usual
Boolean operations, cylindrifiers and diagonal elements \cite[Theorem 13.20]{HHbook}.

Certainly $\D$ is a subalgebra of the $\sf Crs_n$ (the class
of algebras whose units are arbitrary sets of $n$ ary sequences)
with domain $\wp(C^n(M))$ so $\D\in {\sf Crs_n}$. The unit $C^n(M)$ of $\D$ is symmetric,
closed under substitutions, so
$\D\in {\sf G}_n$ (these are relativized set algebras whose units are locally cube, they
are closed under substitutions.)
Because the representation is $n$ flat, quantifier commute, so cylindrifiers commute,
hence we have $\D\in \CA_n$. Now define the map $\theta: \A\to \D$ by $\theta(r)=r(\bar{x})^{M}$.
Preservation of operations is straightforward.  We show that $\theta$ is injective.
Let $r\in A$ be non-zero. But $M$ is a relativized representation, so there $\bar{a}\in M$
with $r(\bar{a})$ hence $\bar{a}$ is a clique in $M$,
and so $M\models r(\bar{x})(\bar{a})$, and $\bar{a}\in \theta(r)$ proving the required.

Let $L(A)_{\infty, \omega}^n$ be the expansion of $L(A)$ by infinite conjunctions. 
Let $\D=\{\phi^M: \phi\in L(A)_{\infty, \omega}^n\}$, that is 
$\D$ consists of all sets definable by infinitary $n$ variable formulas in the clique relativized semantics (we allow infinite disjunctions). 
Then as before $\D\in \CA_n$, but in this case we claim that $\D$ is also atomic. 
Let $\phi^M\in \D\sim \{0\}$.  Let $\bar{a}\in \phi^C$ and let $\tau=\bigwedge \{\psi\in L(A)_{\infty,\omega}^n: M\models \psi(a)\}$.
Then $\tau^{M}$ is an atom. Since $\A$ has a relativized complete representation then $\A$ is atomic.
We show that $\theta$ is a complete embedding.
Let $\phi\in L(A)_{\infty, \omega}^n$ 
be such that $\phi^M\neq 0$. Let $\bar{a}\in \phi^M$. Since $M$ is a complete relativized representation
and $a\in C^n(M)$, then there is an atom $\alpha$ of $\A$ such
that $M\models \alpha(\bar{a})$, hence $\theta(\alpha)\cdot \phi^M\neq 0$ 
and we are done.

The last part follows from theorem \ref{rainbow} since $\A=\Rd_{ca}\sf PEA_{\N^{-1}, \N}\notin S_c\Nr_m\CA_{m+3}$ but
$\A$ is elementary equivalent to a completely representable algebra $\B$, 
and so $\Rd_{ca}\B\in S_c\Nr_m\sf CA_{\omega}$.
\end{proof}

Results in algebraic logic are most attractive when they have direct new impact on variants
of first order logic. We show that our previously proved algebraic results, theorems \ref{can}, \ref{rainbow} are such.

We consider omiting types theorems for $\L_m$ a task implemented also in \cite{ANT}, 
but our approach is different because we move away from classical 
Tarskian semantics. 
We show that when we broaden considerably the class of allowed models,
permitting  the so called $m+3$ flat ones, there still might not be countable models omitting a single non-principal
type. Furthermore, this single-non principal type can be chosen to be the set of co-atoms in an atomic
theory. In the next theorem by $T\models \phi$ 
we mean that $\phi$ is valid in any (classical) model of $T$.

Let $T$ be a consistent  $L_m$ theory, and $m>n$. A model $M$ of $T$ is $n$ flat it  is the base of an $n$ flat relativized
representation of the Tarski-Lindenbaum algebra
$\Fm_T$, that is there exists $V\subseteq {}^mM$, where $M$ is the smallest such set,
and an injective homomorphism $f:\Fm_T\to \wp(V)$.

If $T$ is countable then an $\omega$ flat model is an ordinary model, and if it is consistent, then
it has an $n$ flat model for every finite $n\geq m$.
A type $\Gamma$ is isolated if there exists $\phi\in \L_m$ such that $T\models \phi\to \alpha$ for all $\alpha\in \Gamma$; otherwise it
is non-principal.

\begin{theorem}\label{OTT} There is a countable consistent $\L_m$ theory $T$
with a non-principal type $\Gamma$ that cannot be omitted in any $m+3$ flat model of $T.$
\end{theorem}

\begin{proof} We give two different proofs relying on theorems \ref{can} and \ref{rainbow}, respectively.
\begin{enumarab}

\item  Let $\A$ be the term algebra as in theorem \ref{can}.
Then $\A$ is countable and atomic. Let $\Gamma'$ be the set of  co-atoms.
We can assume that $\A$ is simple \cite{Hodkinson}.
Then $\A=\Fm_T$ for some countable consistent $\L_m$ complete theory $T$.
Then $\Gamma=\{\phi: \phi_T\in \Gamma'\}$ is not principal.
If it can be omitted in an $m+3$ flat model $\M$, then
this gives an $m+3$ complete flat representation of $\A=\Fm_T$,
which gives
an $m+3$ flat representation of $\mathfrak{Cm}\At\A$, which we know, by theorems \ref{can} 
and  \ref{flat},
does not exist.

\item Now we use theorem \ref{rainbow}. 
Let $\A$ be the term algebra of the atom structure of that $\sf CA_m$ $\C$ constructed in theorem \ref{rainbow}. 
Then $\A\notin S_c\Nr_n\CA_{m+3}$.
Let $T$ be such that $\A\cong \Fm_T$ and let $\Gamma$ be the set of co-atoms.
Then $\Gamma$ cannot be omitted in an $m+3$ flat $\M$,
for else, this gives a complete $m+3$ flat representation of $\A$, which means that
$\A\in S_c\Nr_m\CA_{m+3}$ by theorem \ref{flat}.
\end{enumarab}
\end{proof}

\subsection{A stronger result on omitting types}

For finite $m>2$ and $k\geq m$, we introduce a new variety
of cylindric algebras ${\sf CAB}_{m,k}$  of dimension $m$ satisfying
$S\Nr_m\CA_{m+k}\subseteq {\sf CAB}_{m,k}\subseteq \sf RCA_m$ and $\bigcap_{k\geq m}{\sf CAB}_{m,k}=\RCA_m.$
Every $\A\in {\sf CAB}_{k,m}$ has a $k$ dimensional basis that is determined by an $\omega$ rounded game
using $k$ pebbles where \pa\ has only a cylindrifier move.
We show that for any finite $m>2$ and any
$\sf L$ such that ${\sf CAB}_{m,m+3}\subseteq \sf L\subseteq \RCA_m$,
$\sf L$ is not atom-canonical. We use this result to show that $OTT$ fails for $\L_m$ ($m$ finite $>2$)
even if we dispense with the condition of $m$ flatness and require 
only $m$ squareness.
It can be also proved that $n\geq m+3$, $S\Nr_m\CA_{n}$ is not finitely axiomatizable
over ${\sf CAB}_{m,n}$ and that ${\sf CAB}_{m, n+1}$ is not finitely axiomatizable
over ${\sf CAB}_{m,n}$, but we defer this to another paper. 
The results in this subsection apply equally well to any $\sf K$ between $\Sc$ and $\PEA$. We formulate and prove our results just
for $\CA$s.

We recall the definition of networks for $\CA$s obtained by deleting the clause tha has to to with substitutions corresponding 
to transpositions in $\PEA$s.

\begin{definition}\label{network}
Let $2< m<\omega.$ Let $\C$ be an atomic ${\sf CA}_{m}$.
An \emph{atomic  network} over $\C$ is a map
$$N: {}^{m}\Delta\to \At\C,$$
where $\Delta$ is a non-empty set called a set of nodes,
such that the following hold for each $i,j<n$, $\delta\in {}^{m}\Delta$
and $d\in \Delta$:
\begin{itemize}
\item $N(\delta^i_j)\leq {\sf d}_{ij}$
\item $N(\delta[i\to d])\leq {\sf c}_iN(\delta)$
\end{itemize}
\end{definition}

\begin{definition}\label{basis} An {\it $n$ dimensional basis of an algebra in $\A\in \CA_m$}
is a set of $n$ dimensional $\CA$ networks that satisfy only

\begin{enumarab}
\item For all $a\in \At\A$, there is an $N\in H$ such that $N(0,1,\ldots, m-1)=a.$
\item For all $N\in H$ all $\bar{x}\in {}^n\nodes(N)$, for all $i<m$ for all $a\in\At\A$ such that
$N(\bar{x})\leq {\sf c}_ia$, there exists $\bar{y}\equiv_i \bar{x}$ such that $N(\bar{y})=a.$
\end{enumarab}

\end{definition}

We define a game on an atomic $\A\in \CA_m$s that characterizes algebras with $n$ dimensional basis.
Let $2<m<n$. Let $r\leq \omega$, the game $G_r^n(\A)$ is played over atomic $\A$ networks
between \pa\ and \pe\ and has $r$ rounds, and $n$ pebbles,
where $r,n\leq \omega.$

\begin{enumerate}
\item In round $0,$ \pa\ plays an atom $a\in \At\A$. \pe\ must respond with a network $N_a$ with set of nodes $\bar{x}$ such that $N_a(\bar{x})=a$.
\item In any round $t$ with $0<t<r$, assume that the current network is $N_{t-1}.$
Then \pa\ plays as follows.
First if $|N_{t-1}|=n$ then he deletes a node $z\in N_{t-1}$ and defines $N_t$ to be the resulting network.
Otherwise, he lets $N=N_{t-1}$; now \pa\ chooses $\bar{x}\in {}^nN$, $i<m$ and an atom $a\in \At\A$,
such that $N(\bar{x})\leq {\sf c}_ia$. \pe\ must respond to this cylindrifier move
by providing a network $N_t\supseteq N$, with $|N_t|\leq n$, having a node $z$
such
and $N_t(\bar{y})=a$, where $\bar{y}\equiv_i \bar{x}$ and $\bar{y}(i)=z$.
\end{enumerate}

The number of pebbles $n$ measures the squareness, flatness of the relativized representations; these are
significantly distinct notions
in relativized semantics, though the distinction between the 
first two diffuses at the limit of genuine representations.

Crudely, the classical
case then becomes a limiting case, when $n$ goes to infinity.
However, this is not completely accurate because $\omega$ complete relativized representations
may not coincide with classical ones on uncountable algebras. Indeed the uncountable algebra proved in theorem \ref{complete} not have a complete
representation  can be easily proved to have an $\omega$ relativized complete representation using that \pe\ 
can win the $\omega$ rounded atomic game.
\begin{theorem}\label{gamebasis}
Let $2<m<n$. Then following are equivalent for an atomic $\A\in \sf CA_m$
\begin{enumarab}
\item $\A$ has an $n$ dimensional bases.
\item \pe\ has a \ws\ in $G_{\omega}^n(\A).$
\end{enumarab}
\end{theorem}

\begin{proof} The proof is similar to \cite[proposition 12.25]{HHbook}.
\end{proof}

\begin{definition} Let  $2<m<n$ and suppose that $m$ is finite.
We define $\A\in {\sf CAB_n}$ if $\A$ is a subalgebra of a complete atomic $\B\in \CA_m$
that has an $n$ dimensional basis.
\end{definition}

\begin{theorem}\label{cann} Fix a finite  dimension $m>2,$ and let $n\in \omega$ with $n\geq m$.
\begin{enumarab}
\item ${\sf CAB}_{m,n}$ is a canonical variety,
\item If $\A\in \CA_m$ then $\A\in {\sf CAB}_{m,n}$ iff $\A^+$ has an $n$ dimensional basis,
\item If $\A\in \CA_m$, then $\A\in {\sf CAB}_{m,n}$ iff $\A$ has
an $n$ square relativized representation,
\item $\bigcap_{k\in \omega, k\geq m} {\sf CAB}_{m, k}=\sf RCA_m.$
\end{enumarab}
\end{theorem}
\begin{proof}
\begin{enumarab}
\item  The proof is similar to that of \cite[proposition 12.31]{HHbook2} 

\item One side is trivial. The other side is similar to the proof of \cite[Proposition 13.37]{HHbook}. 

\item One side is easy.
Let $M$ be the given representation.  For a network $N$ and $v: N\to M$, we say that
$v$ embeds $N$ in $M$ if for all $r\in \At\A$,
we have $N(\bar{x})=r$ iff $M\models r(v(\bar{x}))$.

All such networks that embed in $M$ are the desired basis. The other more involved part is also similar to \cite[Proposition 13.37]{HHbook}.

\item Since a classical  representation is $\omega$ square, then $\bigcap_{k\geq m}{\sf CAB}_{m,k}$
and $\RCA_m$ coincide on simple countable algebras.
But each is a discriminator variety and so
they are equal.
\end{enumarab}
\end{proof}
The next theorem (that we state without proof) 
says that the class of algebras having $n$ flat relativized representations is not finitely axiomatizable over that having
$n$ square representations. This means that commutativity of cylindrifiers adds a lot.

\begin{theorem}\label{squarerepresentation}
\begin{enumarab}
\item  For $m\geq 3$ and $k\in \omega\sim \{0\}$, $S\Nr_m\CA_{m+k}\subseteq {\sf CAB}_{m, m+k}$
\item For $3\leq m\leq n<\omega$, $S\Nr_m\CA_{n+1}$ is not finitely axiomatizable over ${\sf CAB}_{m, n+1}$
\end{enumarab}
\end{theorem}
\begin{theorem}\label{OTT2}.
\begin{enumarab}
\item The atom structure $\At$ constructed in theorem \ref{can}
satisfies that  $\Rd_{ca}\Tm\At\in \sf RCA_m$, but $\Rd_{ca}\Cm\At\A\notin {\sf CAB}_{m,m+3}.$
\item Thew omitting types theorem fails if we allow $k+3$ square models.
\end{enumarab}
\end{theorem}
\begin{proof} Like the proof of theorem \ref{can} 
by noting that \pa\ can win $G_{m+2}^{m+3}$
on the finite rainbow algebra $m$ dimensional algebra $\sf CA_{m+1, m}$.
\end{proof}

Our next corollary formulated only for $\CA$s is true for any $\K$ between $\Sc$ and $\PEA$.
It also holds for $\sf CAB_{n,n+k}$ in place of $S\Nr_n\CA_{n+k}$, by replacing `flat' by `square' in the last item.

\begin{corollary}Let $n$ be finite with $n>2$ and $k\geq 3$. Then the following hold:

\begin{enumarab}

\item There exist two atomic
cylindric algebras of dimension $n$  with the same atom structure,
one  representable and the other is not in $S\Nr_n\CA_{n+k}$.

\item The variety $\bold S\Nr_n\CA_{n+k}$
is not closed under \d\ completions and is not atom-canonical.
In particular, $\RCA_n$ is not atom-canonical.

\item There exists an algebra outside $\bold S\Nr_n\CA_{n+k}$  with a dense representable
subalgebra.

\item The variety $\bold S\Nr_n\CA_{n+k}$
is not Sahlqvist axiomatizable. In particular, $\RCA_n$ is not Sahlqvist axiomatizable.

\item There exists an atomic representable
$\CA_n$ with no $n+k$ flat complete representation; in particular it has no complete
representation.
\end{enumarab}

\end{corollary}

\begin{proof} Throughout $\A$ is the algebra constructed in theorem \ref{can}.
\begin{enumarab}

\item $\A$ and $\Cm\At\A$ are such.

\item $\Cm\At\A$ is the \d\ completion of $\A$ (even in the $\PA$ and $\Sc$ cases, because $\A$,
hence its $\Sc$ and $\PA$ reducts are completely additive), hence $S\Nr_n\CA_{n+k}$ is not atom canonical \cite[Proposition 2,88,
Theorem, 2.96]{HHbook}.

\item $\A$ is dense in $\Cm\At\A$.

\item Completely additive varieties defined by Sahlqvist
equations are closed under \d\ completions \cite[Theorem 2.96]{HHbook}.

\item Like the proof of theorem \ref{flat}.
\end{enumarab}
\end{proof}
\section{Strongly representable atom structures}

Here we extend Hirsch and Hodkinson's celebrated result that the class of strongly representable atom structures
of cylindric algebras of finite dimension $>2$ is not elementary, to any class of algebras with signature between
$\Df$ and $\PEA$, reproving a result by Bu and Hodkinson \cite{bh}. We believe that our proof is simpler. 
Throughout this section, $n$ is a finite ordinal $>2$.

\begin{definition}
Let $\Gamma=(G, E)$ be a graph.
\begin{enumerate}
\item{A set $X\subseteq G$ is said to be \textit{independent} if $E\cap(X\times X)=\phi$.}
\item{The \textit{chromatic number} $\chi(\Gamma)$ of $\Gamma$ is the smallest $\kappa<\omega$
such that $G$ can be partitioned into $\kappa$ independent sets,
and $\infty$ if there is no such $\kappa$.}
\end{enumerate}
\end{definition}

\begin{definition}\ \begin{enumerate}
\item For an equivalence relation $\sim$ on a set $X$, and $Y\subseteq
X$, we write $\sim\upharpoonright Y$ for $\sim\cap(Y\times Y)$. For
a partial map $K:n\rightarrow\Gamma\times n$ and $i, j<n$, we write
$K(i)=K(j)$ to mean that either $K(i)$, $K(j)$ are both undefined,
or they are both defined and are equal.
\item For any two relations $\sim$ and $\approx$. The composition of $\sim$ and $\approx$ is
the set
$$\sim\circ\approx=\{(a, b):\exists c(a\sim c\wedge c\approx b)\}.$$\end{enumerate}
\end{definition}
\begin{definition}Let $\Gamma$ be a graph.
We define an atom structure $\eta(\Gamma)=\langle H, D_{ij},
\equiv_{i}, \equiv_{ij}:i, j<n\rangle$ as follows:
\begin{enumerate}
\item$H$ is the set of all pairs $(K, \sim)$ where $K:n\rightarrow \Gamma\times n$ is a partial map and $\sim$ is an equivalence relation on $n$
satisfying the following conditions \begin{enumerate}\item If
$|n\diagup\sim|=n$, then $\dom(K)=n$ and $\rng(K)$ is not an independent
subset of $n$.
\item If $|n\diagup\sim|=n-1$, then $K$ is defined only on the unique $\sim$ class $\{i, j\}$ say of size $2$ and $K(i)=K(j)$.
\item If $|n\diagup\sim|\leq n-2$, then $K$ is nowhere defined.
\end{enumerate}
\item $D_{ij}=\{(K, \sim)\in H : i\sim j\}$.
\item $(K, \sim)\equiv_{i}(K', \sim')$ iff $K(i)=K'(i)$ and $\sim\upharpoonright(n\setminus\{i\})=\sim'\upharpoonright(n\setminus\{i\})$.
\item $(K, \sim)\equiv_{ij}(K', \sim')$ iff $K(i)=K'(j)$, $K(j)=K'(i)$, and $K(\kappa)=K'(\kappa) (\forall\kappa\in n\setminus\{i, j\})$
and if $i\sim j$ then $\sim=\sim'$, if not, then $\sim'=\sim\circ[i,
j]$.
\end{enumerate}
\end{definition}
It may help to think of $K(i)$ as assigning the nodes $K(i)$ of
$\Gamma\times n$ not to $i$ but to
the set $n\setminus\{i\}$, so long as its elements are pairwise non-equivalent via $\sim$.\\

For a set $X$, $\mathcal{B}(X)$ denotes the Boolean algebra
$\langle\wp(X), \cap, \cup, \sim, \emptyset,  X\rangle$. 
We denote ${\sf s}_{[i/j]}$ by ${\sf s}^i_j$.

\begin{definition}\label{our algebra}
Let $\B(\Gamma)=\langle\B(\eta(\Gamma)), {\sf c}_{i},
{\sf s}^{i}_{j}, {\sf s}_{[i,j]}, {\sf d}_{ij}\rangle_{i, j<n}$ be the algebra, with
extra non-Boolean operations defined 
for all $X\subseteq \eta(\Gamma)$ as follows:
\begin{align*}
{\sf d}_{ij}&=D_{ij},\\
{\sf c}_{i}X&=\{c: \exists a\in X, a\equiv_{i}c\},\\
{\sf s}_{[i,j]}X&=\{c: \exists a\in X, a\equiv_{ij}c\},\\
{\sf s}^{i}_{j}X&=\begin{cases}
{\sf c}_{i}(X\cap D_{ij}), &\text{if $i\not=j$,}\\
X, &\text{if $i=j$.}
\end{cases}
\end{align*}
\end{definition}

\begin{definition}
For any $\tau\in\{\pi\in n^{n}: \pi \text{ is a bijection}\}$, and
any $(K, \sim)\in\eta(\Gamma)$, we define $\tau(K,
\sim)=(K\circ\tau, \sim\circ\tau)$.\end{definition}

The proof of the following two Lemmas is straightforward.
\begin{lemma}\label{Lemma 1} 
For any $\tau\in\{\pi\in {}^nn: \pi \text{ is a bijection}\}$, and
any $(K, \sim)\in\eta(\Gamma)$. $\tau(K, \sim)\in\eta(\Gamma)$.
\end{lemma}
\begin{lemma}\label{Lemma 2}\ For any $(K, \sim)$, $(K', \sim')$, and $(K'', \sim'')\in\eta(\Gamma)$, and $i, j\in n$:
\begin{enumerate}
\item$(K, \sim)\equiv_{ii}(K', \sim')\Longleftrightarrow (K, \sim)=(K', \sim')$.
\item$(K, \sim)\equiv_{ij}(K', \sim')\Longleftrightarrow (K, \sim)\equiv_{ji}(K', \sim')$.
\item If $(K, \sim)\equiv_{ij}(K', \sim')$, and $(K, \sim)\equiv_{ij}(K'', \sim'')$, then it follows that $(K', \sim')=(K'', \sim'')$.
\item If $(K, \sim)\in D_{ij}$, then we have 
$(K, \sim)\equiv_{i}(K', \sim')\Longleftrightarrow\exists(K_{1},
\sim_{1})\in\eta(\Gamma), (K, \sim) \equiv_{j}(K_{1},
\sim_{1})\wedge(K', \sim')\equiv_{ij}(K_{1}, \sim_{1})$.
\item ${\sf s}_{[i,j]}(\eta(\Gamma))=\eta(\Gamma)$.
\end{enumerate}
\end{lemma}
The proof of the next lemma is tedious but not too hard.
\begin{theorem}\label{it is qea}
For any graph $\Gamma$, $\B(\Gamma)$ is a simple
$\PEA_n$.
\end{theorem}
\begin{proof}\
We follow the axiomatization in \cite{ST} except renaming the items by $Q_i$.
Let $X\subseteq\eta(\Gamma)$, and $i, j, \kappa\in n$:
\begin{enumerate}
\item ${\sf s}^{i}_{i}=Id$ by definition \ref{our algebra}, ${\sf s}_{[i,i]}X=\{c:\exists a\in X, a\equiv_{ii}c\}=\{c:\exists a\in X, a=c\}=X$
(by Lemma \ref{Lemma 2} (1));\\
${\sf s}_{[i,j]}X=\{c:\exists a\in X, a\equiv_{ij}c\}=\{c:\exists a\in X,
a\equiv_{ji}c\}={\sf s}_{[j,i]}X$ (by Lemma \ref{Lemma 2} (2)).
\item Axioms $Q_{1}$, $Q_{2}$ follow directly from the fact that the
reduct $\mathfrak{Rd}_{ca}\mathfrak{B}(\Gamma)=\langle\mathcal{B}(\eta(\Gamma)), {\sf c}_{i}$, ${\sf d}_{ij}\rangle_{i, j<n}$
is a cylindric algebra which is proved in \cite{hirsh}.
\item Axioms $Q_{3}$, $Q_{4}$, $Q_{5}$
follow from the fact that the reduct $\mathfrak{Rd}_{ca}\mathfrak{B}(\Gamma)$ is a cylindric algebra, and from \cite{HMT2}
(Theorem 1.5.8(i), Theorem 1.5.9(ii), Theorem 1.5.8(ii)).
\item ${\sf s}^{i}_{j}$ is a Boolean endomorphism by \cite{HMT2} (Theorem 1.5.3).
\begin{align*}
{\sf s}_{[i,j]}(X\cup Y)&=\{c:\exists a\in(X\cup Y), a\equiv_{ij}c\}\\
&=\{c:(\exists a\in X\vee\exists a\in Y), a\equiv_{ij}c\}\\
&=\{c:\exists a\in X, a\equiv_{ij}c\}\cup\{c:\exists a\in Y,
a\equiv_{ij}c\}\\
&={\sf s}_{[i,j]}X\cup {\sf s}_{ij}Y.
\end{align*}
${\sf s}_{[i,j]}(-X)=\{c:\exists a\in(-X), a\equiv_{ij}c\}$, and
${\sf s}_{[i,j]}X=\{c:\exists a\in X, a\equiv_{ij}c\}$ are disjoint. For, let
$c\in({\sf s}_{[i,j]}(X)\cap {\sf s}_{ij}(-X))$, then $\exists a\in X\wedge b\in
(-X)$, such that $a\equiv_{ij}c$, and $b\equiv_{ij}c$. Then $a=b$, (by
Lemma \ref{Lemma 2} (3)), which is a contradiction. Also,
\begin{align*}
{\sf s}_{[i,j]}X\cup {\sf s}_{[i,j]}(-X)&=\{c:\exists a\in X,
a\equiv_{ij}c\}\cup\{c:\exists a\in(-X), a\equiv_{ij}c\}\\
&=\{c:\exists a\in(X\cup-X), a\equiv_{ij}c\}\\
&={\sf s}_{[i,j]}\eta(\Gamma)\\
&=\eta(\Gamma). \text{ (by Lemma \ref{Lemma 2} (5))}
\end{align*}
therefore, ${\sf s}_{[i,j]}$ is a Boolean endomorphism.
\item \begin{align*}{\sf s}_{[i,j]}{\sf s}_{[i,j]}X&={\sf s}_{[i,j]}\{c:\exists a\in X, a\equiv_{ij}c\}\\
&=\{b:(\exists a\in X\wedge c\in\eta(\Gamma)), a\equiv_{ij}c, \text{ and }
c\equiv_{ij}b\}\\
&=\{b:\exists a\in X, a=b\}\\
&=X.\end{align*}
\item\begin{align*}{\sf s}_{[i,j]}{\sf s}^{i}_{j}X&=\{c:\exists a\in {\sf s}^{i}_{j}X, a\equiv_{ij}c\}\\
&=\{c:\exists b\in(X\cap {\sf d}_{ij}),a\equiv_{i}b\wedge
a\equiv_{ij}c\}\\
&=\{c:\exists b\in(X\cap {\sf d}_{ij}), c\equiv_{j}b\} \text{ (by Lemma
\ref{Lemma 2} (4))}\\
&={\sf s}^{j}_{i}X.\end{align*}
\item We need to prove that ${\sf s}_{[i,j]}{\sf s}_{[i,\kappa]}X={\sf s}_{[j,\kappa]}{\sf s}_{[i,j]}X$ if $|\{i, j, \kappa\}|=3$.

Let $(K, \sim)\in {\sf s}_{[i,j]}{\sf s}_{[i,\kappa]}X$ then
$\exists(K', \sim')\in\eta(\Gamma)$, and $\exists(K'', \sim'')\in X$
such that $(K'', \sim'')\equiv_{i\kappa}(K', \sim')$ and $(K',
\sim')\equiv_{ij}(K, \sim)$.\\
Define $\tau:n\rightarrow n$ as follows:
\begin{align*}\tau(i)&=j\\
\tau(j)&=\kappa\\
\tau(\kappa)&=i, \text{ and}\\
\tau(l)&=l \text{ for every } l\in(n\setminus\{i, j,
\kappa\}).\end{align*} Now, it is easy to verify that $\tau(K',
\sim')\equiv_{ij}(K'', \sim'')$, and $\tau(K',
\sim')\equiv_{j\kappa}(K, \sim)$. Therefore, $(K, \sim)\in
{\sf s}_{[j,\kappa]}{\sf s}_{[i,j]}X$, i.e., ${\sf s}_{[i,j]}{\sf s}_{[i,\kappa]}X\subseteq
{\sf s}_{[j,\kappa]}{\sf s}_{[i,j]}X$. Similarly, we can show that
${\sf s}_{[j,\kappa]}{\sf s}_{[i,j]}X\subseteq {\sf s}_{[i,j]}{\sf s}_{[i,\kappa]}X$.
\item Axiom $Q_{10}$ follows from \cite{HMT2} (Theorem 1.5.7)
\item Axiom $Q_{11}$ follows from axiom 2, and the definition of ${\sf s}^{i}_{j}$.
\end{enumerate}
Since $\Rd_{ca}\B$ is a simple $\CA_{n}$, by
\cite{hirsh}, then $\B$ is a simple $\PEA_n$. This follows from the fact that ideals $I$ is an ideal in $\Rd_{ca}\B$ if and only if it is an ideal in
$\B$.
\end{proof}

\begin{theorem}
$\B(\Gamma)$ is a simple $\PEA_{n}$ generated using infinite unions 
by the set of the $n-1$ dimensional elements.
\end{theorem}
\begin{proof}
$\mathfrak{B}(\Gamma)$ is a simple $\PEA_{n}$ from  Theorem \ref{it
is qea}. It remains to show, since $\B$ is atomic, so that every element is an infinite union of atoms
 that $\{(K, \sim)\}=\prod\{{\sf c}_{i}\{(K,
\sim)\}: i<n\}$ for any $(K, \sim)\in H$. Let $(K, \sim)\in H$,
clearly $\{(K, \sim)\}\leq\prod\{{\sf c}_{i}\{(K, \sim)\}: i<n\}$. For the
other direction assume that $(K', \sim')\in H$ and $(K,
\sim)\not=(K', \sim')$. We show that $(K',
\sim')\not\in\prod\{{\sf c}_{i}\{(K, \sim)\}: i<n\}$. Assume toward a
contradiction that $(K', \sim')\in\prod\{{\sf c}_{i}\{(K, \sim)\}:i<n\}$,
then $(K', \sim')\in {\sf c}_{i}\{(K, \sim)\}$ for all $i<n$, i.e.,
$K'(i)=K(i)$ and
$\sim'\upharpoonright(n\setminus\{i\})=\sim\upharpoonright(n\setminus\{i\})$
for all $i<n$. Therefore, $(K, \sim)=(K', \sim')$ which makes a
contradiction, and hence we get the other direction.
\end{proof}
\begin{theorem}\label{chr. no.}
Let $\Gamma$ be a graph.
\begin{enumerate}\item Suppose that $\chi(\Gamma)=\infty$. Then $\mathfrak{B}(\Gamma)$
is representable.\item If $\Gamma$ is infinite and
$\chi(\Gamma)<\infty$ then $\Rd_{df}\B(\Gamma)$
is not
representable
\end{enumerate}
\end{theorem}
\begin{proof}
\begin{enumerate}
\item For the sake of brevity, we denote $\B(\Gamma)$ by $\B$. We have $\Rd_{ca}\B$ is representable (c.f., \cite{hirsh}).
Let $X=\{x\in \B:\Delta x\not=n\}$. Call $J\subseteq \B$ inductive if
$X\subseteq J$ and $J$ is closed under infinite unions and
complementation. Then $B$ is the smallest inductive
subset of $\B$. Let $f$ be an isomorphism of
$\Rd_{ca}\B$ onto a cylindric set algebra with
base $U$. Clearly, by definition, $f$ preserves ${\sf s}^{i}_{j}$ for each
$i, j<n$. It remains to show that $f$ preserves ${\sf s}_{[i,j]}$ for every
$i< j<n$. Let $i< j<n$, since ${\sf s}_{[i,j]}$ is a  Boolean endomorphism and
completely additive, it suffices to show that $f({\sf s}_{[ij]}x)={\sf s}_{[ij]}f(x)$
for all $x\in \At\B$. Let $x\in \At\B$ and $\mu\in
n\setminus\Delta x$. If $\kappa=\mu$ or $l=\mu$, say $\kappa=\mu$,
then
\begin{align*}
f({\sf s}_{[\kappa, l]}x)&=f({\sf s}_{[\kappa, l]}{\sf c}_{\kappa}x)\\
&=f({\sf s}^{\kappa}_{l}x)\\
&={\sf s}^{\kappa}_{l}f(x)\\
&={\sf s}_{[\kappa, l]}f(x).
\end{align*}
If $\mu\not\in\{\kappa, l\}$ then
\begin{align*}
f({\sf s}_{[\kappa, l]}x)&=f({\sf s}^{l}_{\mu}{\sf s}^{\kappa}_{l}{\sf s}^{\mu}_{\kappa}{\sf c}_{\mu}x)\\
&={\sf s}^{l}_{\mu}{\sf s}^{\kappa}_{l}{\sf s}^{\mu}_{\kappa}{\sf c}_{\mu}f(x)\\
&={\sf s}_{[\kappa, l]}f(x).
\end{align*}
\item Assume for contradiction that $\Rd_{df}\B$ is representable. Since $\Rd_{ca}\B$
is generated by $n-1$ dimensional elements then
$\mathfrak{Rd}_{ca}\B$ is representable. But this
contradicts \cite[Proposition 5.4]{hirsh}.
\end{enumerate}
\end{proof}
\begin{theorem}\label{el}
Let $2<n<\omega$ and $\mathcal{T}$ be any signature between $\Df_{n}$
and $\PEA_{n}$. Then the class of strongly representable atom
structures of type $\mathcal{T}$ is not elementary.
\end{theorem}
\begin{proof}
By Erd\"{o}s's famous 1959 Theorem \cite{Erdos}, for each finite
$\kappa$ there is a finite graph $G_{\kappa}$ with
$\chi(G_{\kappa})>\kappa$ and with no cycles of length $<\kappa$.
Let $\Gamma_{\kappa}$ be the disjoint union of the $G_{l}$ for
$l>\kappa$. Clearly, $\chi(\Gamma_{\kappa})=\infty$. So by Theorem
\ref{chr.
no.} (1), $\mathfrak{B}(\Gamma_{\kappa})=\mathfrak{B}(\Gamma_{\kappa})^{+}$ is representable.\\
\- Now let $\Gamma$ be a non-principal ultraproduct
$\prod_{D}\Gamma_{\kappa}$ for the $\Gamma_{\kappa}$. It is
certainly infinite. For $\kappa<\omega$, let $\sigma_{\kappa}$ be a
first-order sentence of the signature of the graphs stating that
there are no cycles of length less than $\kappa$. Then
$\Gamma_{l}\models\sigma_{\kappa}$ for all $l\geq\kappa$. By
Lo\'{s}'s Theorem, $\Gamma\models\sigma_{\kappa}$ for all
$\kappa$. So $\Gamma$ has no cycles, and hence by, \cite{hirsh}
Lemma 3.2, $\chi(\Gamma)\leq 2$. By Theorem \ref{chr. no.} (2),
$\mathfrak{Rd}_{df}\mathfrak{B}$ is not representable. It is easy to
show (e.g., because $\mathfrak{B}(\Gamma)$ is first-order
interpretable in $\Gamma$, for any $\Gamma$) that$$
\prod_{D}\mathfrak{B}(\Gamma_{\kappa})\cong\mathfrak{B}(\prod_{D}\Gamma_{\kappa}).$$
Combining this with the fact that: for any $n$-dimensional atom
structure $\mathcal{S}$

$\mathcal{S}$ is strongly
representable $\Longleftrightarrow$ $\mathfrak{Cm}\mathcal{S}$ is
representable,
the desired follows.
\end{proof}
Again like in the case of rainbow algebras, define the polyadic operations corresponding to transpositions
using the notation in \cite[Definition 3.6.6]{HHbook2}, in the context of defining atom structures from
classes of models by
$R_{{\sf s}_{[ij]}}=\{([f], [g]): f\circ [i,j]=g\}$. This is well defined.
In particular, we can (and will)
consider the Monk-like algebra $\M(\Gamma)$ as defined in \cite[top of p. 78]{HHbook2} as a polyadic equality
algebra.

\begin{theorem}\label{hhchapterbook}Let $\M(\Gamma)$ be the polyadic equality algebra defined above.
If $\chi(\Gamma)=\infty$, then $\M(\Gamma)$ is representable as a  polyadic
equality algebra. If $\chi(\Gamma)<\infty$,
then $\Rd_{df}\M(\Gamma)$ is not representable.
\end{theorem}
\begin{proof} Only networks are changed but the atomic game is the same, so clearly \pe\ has a \ws\
infinite (possibly transfinite)  game over $\M(\Gamma)^{\sigma}$ which is
the canonical extension of $\M(\Gamma)$ \cite[Lemma 3.6.4, Lemma 3.6.7]{HHbook2}

In using the latter lemma the proof is unaltered, in the former lemma \pe\ s strategy to win the
$G^{|{\sf Uf}(M(\Gamma))|+\omega)}\M(\Gamma)^{\sigma}$ game is exactly the same
so  that  $\M(\Gamma)^{\sigma}$ is completely representable, hence $\M(\Gamma)$
is representable as a polyadic equality algebra.

The second part follows from the fact that $\M(\Gamma)$ is generated by elements whose dimension
sets $<n$,  \cite[Theorem 5.4.29]{HMT2}.
\end{proof}

\begin{definition}
\begin{enumarab}

\item A Monks algebra $\M(\Gamma)$ is good if $\chi(\Gamma)=\infty.$

\item A Monk's algebra $\M(\Gamma)$ 
is bad if $\chi(\Gamma)<\infty.$

\end{enumarab}
\end{definition}
It is easy to construct a good Monks algebra
as an ultraproduct (limit) of bad Monk algebras. Monk's original algebras can be viewed this way.
The converse is, as illustrated above, is much harder.
It took Erdos probabilistic graphs, to get a sequence of good graphs converging to a bad one.
Recall that we defined a strongly representable $\A\in \CA_n$ to be  atomic, such that $\Cm\At\A\in \sf RCA_n$. Exdending the definition 
to any $\sf K$ between 
$\Df$ and $\PEA$, we stipulate that $\A\in \sf K$ is {\it strongly representable} if it is atomic, completely additive and $\Cm\At\A\in \sf RK_n$. 
Atomicity and complete additivity on an atomic $\sf BAO$ are first order definable.
Let $\sf SRK_n$ denote the class of strongly repesentable $\K$ algebras of dimension $n$. 
Using our modified Monk-like algebras one can obtain the result formulated in theorem \ref{el}, using the same graphs.
Also answering a question of Hodkinson's in 
\cite[p. 284]{AU} we get:

\begin{corollary}\label{proper}
For any finite $n>2$,  and any $\sf K$ between $\Df$ and $\PEA$
the class $\sf SRK_n$  is not elementary. 
In fact, for any class $\K$ between $\Df$ and $\PEA$ and any finite $n>2$,
the class $\K_s=\{\A\in \K_n: \Cm\At\A\in \sf RK_n\}$, is not elementary.
Furthermore, 
for $\sf K$ between $\Sc$ and $\PEA$, $\Nr_n\K_{\omega}\subseteq S_c\Nr_n\K_{\omega}\subseteq {\sf SRK}_n$ and all inclusions are proper.
\end{corollary}
\begin{proof}

For the first part and second parts, we proceed like in the proof of theorem \ref{el} using the Monk algebras $\M(\Gamma)$, 
$\Gamma$ a graph. By Erd\"{o}s's celebrated  Theorem,  for each finite
$\kappa$ there is a finite graph $\G_{\kappa}$ with
$\chi(\G_{\kappa})>\kappa$ and with no cycles of length $<\kappa$.
Let $\Gamma_{\kappa}$ be the disjoint union of the $\G_{l}$ for
$l>\kappa$. Clearly, $\chi(\Gamma_{\kappa})=\infty$
hence $\M(\Gamma_{\kappa})$ is representable.\\
- Now let $\Gamma$ be a non-principal ultraproduct
$\Pi_{D}\Gamma_{\kappa}$ of the $\Gamma_{\kappa}$'s. It is
certainly infinite. For $\kappa<\omega$, let $\sigma_{\kappa}$ be a
first-order sentence of the signature of the graphs stating that
there are no cycles of length less than $\kappa$. Then
$\Gamma_{l}\models\sigma_{\kappa}$ for all $l\geq\kappa$. By
Lo\'{s}'s Theorem, $\Gamma\models\sigma_{\kappa}$ for all
$\kappa$. So $\Gamma$ has no cycles, and hence $\chi(\Gamma)\leq 2$
hence $\mathfrak{Rd}_{df}\M(\Gamma)$ is not representable.
But $$\Pi_{D}\M(\Gamma_{\kappa})\cong\M(\Pi_{D}\Gamma_{\kappa}),=\M(\Gamma)$$
and we are done. 

So $\M(\Gamma_i)$ is a family of algebras in $\K_s$ whose ultraproduct 
is not in $\K_s$.

We prove the last part. The first inclusion follows from example \ref{SL} and the last follows from the construction in theorem \ref{rainbow}. 
Taking for example cylindric algebras we have
$\A=\Rd_{ca}\PEA_{\N^{-1}, \N}\notin S_c\Nr_n\CA_{n+3}\supseteq S_c\Nr_n\CA_{\omega}$, but 
$\A$ satisfies the Lyndon conditions since \pe\ can win $G_k$ for all
finite $k$ hence $\A$ is strongly representable.
\end{proof}

Henceforth  we carry out our discussions and  formulate our theorems for cylindric algebras.
Everything said about cylindric algebras carries over to the other cylindric-like algebras
approached in our previous investigations like $\Sc, \PA$ and $\PEA$.

However, $\Df$s does not count here, for the notion of neat reducts
for such algebras is trivial.
Lifting from atom structures, we define several classes of atomic representable algebras.
For a class $\K$ with a Boolean reduct, $\K\cap \At$ denotes the class of atomic algebras in $\K$; the former is elementary iff the
latter is.

Fix $n>2$ finite. ${\sf CRA_n}$ denotes the class of completely representable algebras of dimension $n$.
Let ${\sf SRCA_n}$ be the class of strongly representable atomic algebras of dimension $n$,
$\A\in {\sf SRCA_n}$ iff $\Cm\At\A\in \RCA_n$.  ${\sf WRCA_n}$ denotes the class of weakly representable algebras of dimension $n$,
and this is just $\RCA_n\cap \At$.  We  have the following strict inclusions lifting them up from atom structures \cite{HHbook2} (*):
$${\sf CRA}_n\subset {\sf LCA_n}\subset {\sf SRCA_n}\subset {\sf WCRA}_n$$

The second  and fourth classes are elementary but not finitely axiomatizable, bad Monk  algebras converging to a good one (Monk's original algebras
are like that),
can witness this, while ${\sf SRCA_n}$ is not closed under both ultraroots and ultraproducts,
good Monks algebras converging to a bad one witnesses
this. The algebra $\Rd_{ca}\PEA_{\N^{-1}, \N}$  witness that ${\sf CRA_n}$ is not elementary,
since it is not completely representable (in fact it is not in $S_c\Nr_n\CA_{n+3}$), but is elementary equivalent to (countable atomic) 
algebra
that is. From this we readily conclude that ${\sf CRCA_n}$ is properly contained in ${\sf LCA}_n$.

For a cylindric algebra atom structure $\F$ the first order algebra over $\F$ is the subalgebra
of $\Cm\F$ consisting of all sets of atoms that are first order
definable with parameters from $S$. ${\sf FOCA_n}$ denotes the class of atomic such algebras of dimension $n$.

This class is strictly bigger than ${\sf SRCA_n}$.
Indeed, let $\A$ be any the  rainbow term algebra  obtained by blowing up and blurring the  finite rainbow algebra
$\CA_{n+1, n}$ proving that
$S\Nr_{n}\CA_{n+3}$ is not atom canonical.
This algebra was defined using first order formulas in the rainbow signature
(the latter is first order since we had only
finitely many greens).  Though the usual semantics was perturbed, first order logic did not see the relativization, only
infinitary formulas saw it,
and thats why the complex algebra could not be represented.
These examples all show that ${\sf SRCA_n}$ is properly contained in ${\sf FOCA_n}.$
Another way to view this is to notice that ${\sf FOCA_n}$ is elementary, that ${\sf SRCA_n}\subseteq {\sf FOCA_n}$, but
${\sf SRCA_n}$ is not elementary.

Last inclusion follows from the following  $\sf RA$ to $\CA$ adaptation of an example of Hirsch and Hodkinson
which we use to show that
${\sf FOCA_n}\subset  {\sf WRCA_n}$. This is not at all obvious because they are both elementary.
\begin{example}\label{fo}
Take an $\omega$ copy of the  $n$ element graph with nodes $\{1,2,\ldots, n\}$ and edges
$1\to 2\to\ldots\to n$. Then of course $\chi(\Gamma)<\infty$. Now  $\Gamma$  has an $n$ first order definable colouring.
Since $\M(\Gamma)$ as defined above and in \cite[top of p. 78]{HHbook2} is not representable, then the algebra of first order
definable sets is also not representable because $\Gamma$ is first order interpretable in
$\rho(\Gamma)$, the atom structure constructed from $\Gamma$ as defined in \cite{HHbook2}.
However, it can be shown that the term algebra is representable. (This is not so easy to prove).
\end{example}

\begin{theorem}
Let $n>2$ be finite. Then we have the following inclusions (note that $\At$ commutes with ${\bf UpUr})$:
$$\Nr_n\CA_{\omega}\cap \At\subset {\bf UpUr}\Nr_n\CA_{\omega}\cap \At$$
$$\subset {\bf Up Ur}\bold S_c\Nr_n\CA_{\omega}\cap \At={\bf UpUr}{\sf CRA_n}={\sf LCA}_n\subset {\sf SRCA_n}$$
$$\subset
{\bf Up}{\sf SRCA_n}={\bf Ur}{\sf  SRCA_n}={\bf UpUr}{\sf SRCA}_n\subseteq {\sf FOCA_n}$$
$$\subset S\Nr_n\CA_{\omega}\cap \At={\sf WRCA_n}= \RCA_n\cap \At.$$
\end{theorem}
\begin{proof}
The majority of  inclusions, and indeed their strictness,
can be distilled without much difficulty from our previous work.
It is known that ${\bf Up}{\sf SRCA_n}={\bf Ur}{\sf  SRCA_n}$ \cite{hirsh}. 
To show that the elementary closure of $\sf CRA_n$ is $\sf LCA_n$, let $\A\in \sf LCA_n$, 
then \pe\ can win the $k$ rounded atomic game $G_k$ for all $k\in \omega$ on $\At\A$. 
Using ultrapowers and an elementary 
chain argument one gets a countable (atomic) $\C$ such that $\C\equiv \A$, and \pe\ can win $G$, the $\omega$ rounded atomic game, 
on $\At\C$, hence $\C$ being countable is completely representable. 
It follows that $\A\in {\bf UrUp}\{\C\}\subseteq {\bf UrUp}{\sf CRA_n},$ 
and we are done.

The first  inclusion is witnessed by a slight modification
of the algebra $\B$ used in the proof of \cite[Theorem 5.1.4]{Sayedneat},
showing that for any pair of ordinals $1<n<m\cap \omega$, the class $\Nr_n\CA_m$ is not elementary.

In the constructed model $\sf M$ \cite[lemma 5.1.3]{Sayedneat} on which (using the notation in {\it op.cit}),
the two algebras $\A$ and $\B$  are based,
one requires (the stronger) that the interpretation of the $3$ ary relations symbols in the signature
in $\sf M$ are {\it disjoint} not only distinct as above.
Atomicity of $\B$ follows immediately,  since its Boolean reduct  is now a product of atomic algebras.
For $n=3$ these are denoted by $\A_u$ except for one countable component $\B_{Id}$, $u\in {}^33\sim \{Id\}$, cf. \cite{Sayedneat}
p.113-114. Second inclusion follows from example \ref{SL}. Third follows from theorem \ref{can},  
and fourth  inclusion follows from theorem \ref{el}. 
Last one follows from example \ref{fo}.
\end{proof}

\end{document}